\documentclass[a4paper,10pt, reqno]{amsart}

\usepackage{amsmath,amssymb,latexsym, amsfonts}
\usepackage{verbatim}
\usepackage{times}
\usepackage{mathbbol}
\usepackage{hyperref}
\usepackage{mathabx}
\usepackage{bbm}

\usepackage{tikz}

\usepackage[all,ps]{xy}
\usepackage{color}

\usepackage{stmaryrd}

\usepackage{capt-of}

\usepackage[OT2, T1]{fontenc} 

\usepackage[main=english, russian]{babel}

\newcommand{\compactlist}[1]{\setlength{\itemsep}{0pt} \setlength{\parskip}{0pt} \setlength{\leftskip}{-0.#1em}}

\numberwithin{equation}{section}

\theoremstyle{plain}

\newtheorem{theorem}{Theorem}[section]

\newtheorem{prop}[theorem]{Proposition}
\newtheorem{lemma}[theorem]{Lemma}
\newtheorem{lem}[theorem]{Lemma}

\newtheorem{corollary}[theorem]{Corollary}
\newtheorem{cor}[theorem]{Corollary}

\theoremstyle{definition}
\newtheorem{definition}[theorem]{Definition}
\newtheorem{dfn}[theorem]{Definition}
\newtheorem{example}[theorem]{Example}

\newtheorem{rem}[theorem]{Remark}

\newcommand{\ahha}{{\scriptscriptstyle{A}}}

\newcommand{\ikks}{{\scriptscriptstyle{X}}}

 \newcommand{\N}{{\mathbb{N}}}

 \newcommand{\K}{{\mathbb{K}}}
 
 \newcommand{\Z}{{\mathbb{Z}}}

\newcommand{\frkg}{{\mathfrak g}}

\newcommand{\gd}{\delta}

\newcommand{\gvf}{\varphi}

\newcommand{\gO}{\Omega}

\newcommand{\gs}{\sigma} 
\newcommand{\gS}{\Sigma}

\newcommand{\cG}{{\mathcal G}}

\newcommand{\cI}{{\mathcal I}}

\newcommand{\cL}{{\mathcal L}}
\newcommand{\cM}{{\mathcal M}}
\newcommand{\cN}{{\mathcal N}}
\newcommand{\cO}{{\mathcal O}}

\newcommand{\cS}{{\mathcal S}}
\newcommand{\cT}{{\mathcal T}}

\newcommand{\cX}{{\mathcal X}}

\newcommand{\End}{\operatorname{End}}

\newcommand{\Hom}{\operatorname{Hom}}

\newcommand{\Tor}{{\rm Tor}}
\newcommand{\Ext}{{\rm Ext}}

\newcommand{\id}{{\rm id}}

\newcommand{\g}{{\frkg}}                                                        

\newcommand{\pl}{\partial}

\newcommand{\rmref}[1]{{\rm (}\ref{#1}{\rm )}}

\newcommand{{\Hl}}{{H^{\ell}}} 
\newcommand{{\mHop}}{{m_{H^{\rm op}}}} 
\newcommand{{\Hop}}{{H^{\rm op}}} 
\newcommand{{\mUop}}{{m_{U^{\rm op}}}} 
\newcommand{{\mUopp}}{{m_{\scriptscriptstyle{U^{\rm op}}}}} 
\newcommand{{\Uop}}{{U^{\rm op}}}
\newcommand{{\mVop}}{{m_{V^{\rm op}}}} 
\newcommand{{\Vop}}{{V^{\rm op}}}  
\newcommand{{\Ae}}{{A^{\rm e}}}
\newcommand{{\Be}}{{B^{\rm e}}}
\newcommand{{\Ue}}{{U^{\rm e}}}
\newcommand{{\He}}{{H^{\rm e}}}
\newcommand{{\Aop}}{{A^{\rm op}}}
\newcommand{{\Aope}}{({A^{\rm op}})^{\rm e}}
\newcommand{{\Aopl}}{{A^{\rm op}_\pl}}

\newcommand{{\Bop}}{{B^{\rm op}}}
\newcommand{{\Bopp}}{{\scriptscriptstyle{{B^{\rm op}}}}}
\newcommand{{\Bope}}{({B^{\rm op}})^{\rm e}}
\newcommand{{\Bpl}}{{B_\pl}}

\newcommand{{\op}}{{{\rm op}}}
\newcommand{{\coop}}{{{\rm coop}}}
\newcommand{{\sop}}{{*^{\rm op}}}
\newcommand{{\co}}{{{\rm co}}}

\newcommand{\kmod}{\K\mbox{-}\mathbf{Mod}}                     %
\newcommand{{\gog}}{{G \rightrightarrows G_0}}

\newcommand{{\rra}}{\rightrightarrows}

\newcommand{{\lra}}{\ \longrightarrow \ }
\newcommand{{\lla}}{\ \longleftarrow \ }
\newcommand{{\lma}}{\ \longmapsto \ }

\newcommand{{\bull}}{{\scriptscriptstyle{\bullet}}}
\newcommand{{\qqquad}}{{\quad\quad\quad}}
\newcommand{\Aopp}{{\scriptscriptstyle{\Aop}}}
\newcommand{\Aee}{{\scriptscriptstyle{\Ae}}}

\newsavebox{\foobox}

\newcommand{\Bc}{\mbox{\selectlanguage{russian}B}}
\newcommand{\bc}{\mbox{\selectlanguage{russian}b}}

\begin{document}

\title{Higher brackets on cyclic and negative cyclic (co)homology}

\author{Domenico Fiorenza}
\author{Niels Kowalzig}

\begin{abstract}
The purpose of this article is to embed the string topology bracket developed by Chas-Sullivan and Menichi on negative cyclic cohomology groups as well as the dual bracket found by de Thanhoffer de V\"olcsey-Van den Bergh on negative cyclic homology groups into the global picture of a noncommutative differential (or Cartan) calculus up to homotopy on the (co)cyclic bicomplex in general, in case a certain Poincar\'e duality is given. For negative cyclic cohomology, this in particular leads to a  Batalin-Vilkoviski\u\i\ algebra structure on the  underlying Hochschild cohomology. 
In the special case in which this BV bracket vanishes, one obtains an $e_3$-algebra structure on Hochschild cohomology.
%
%
The results are given in the general and unifying setting of (opposite) cyclic modules over (cyclic) operads.
%
\end{abstract}

\address{Dipartimento di Matematica, Universit\`a degli Studi di Roma La
Sapienza, P.le Aldo Moro 5, 00185 Roma, Italia}

\email{fiorenza@mat.uniroma1.it}
\email{kowalzig@mat.uniroma1.it}



\maketitle

\setcounter{tocdepth}{1}
\tableofcontents

\section{Introduction}
\subsection{Aims and objectives}
%
%
%
%
%

It is a classical result that the Hochschild cohomology $HH^\bullet(A)$ of an associative (unital) $\K$-algebra $A$ carries a natural Gerstenhaber algebra structure \cite{Ger:TCSOAAR}. Namely, one has an associative and graded commutative product $\smallsmile \colon HH^p(A)\otimes HH^q(A)\to HH^{p+q}(A)$ and a degree $-1$ Lie bracket $\{\cdot,\cdot\}\colon HH^p(A)\otimes HH^q(A)\to HH^{p+q-1}(A)$, which are compatible in the sense that the degree $-1$ Poisson algebra identity holds. 

Structures of this kind do not only exist on Hochschild cohomology for associative algebras but on quite a large class of cohomology groups which can be described as an $\Ext$-group, as ist the case for $HH^\bullet(A) = \Ext_\Ae^\bullet(A,A)$ if $A$ is $\K$-projective. For example, if $U$ is a $\K$-bialgebra, then $\Ext_U^\bullet(\K,\K)$ is a Gerstenhaber algebra again \cite{Men:CMCMIALAM}, and this is even the case for $\Ext_U^\bullet(A,A)$, when $U$ is a bialgebroid over a possibly noncommutative algebra $A$, see \cite[\S3.6]{KowKra:BVSOEAT}. 

Moreover, if the bialgebra is a Hopf algebra with involutive antipode, then $\Ext_U^\bullet(\K,\K)$ is even a Batalin-Vilkoviski\u\i\ (BV) algebra \cite{Men:CMCMIALAM}, that is, a Gerstenhaber algebra the bracket of which measures the failure of the cyclic coboundary $B$ (which will be called $\Bc$ in the main text to better distinguish between homology and cohomology) to be a derivation of the cup product, that is
\begin{equation}
\label{avvisoagliutenti}
 \{f, g\} =(-1)^{f}B  (f \smallsmile g) -(-1)^{f} (B  f \smallsmile g)   -(f\smallsmile B  g).
\end{equation}
An analogous statement for (left) Hopf algebroids $(U, A)$, that is, when $\Ext_U(A,A)$ happens to be a Batalin-Vilkoviski\u\i\ algebra appears to be more involved \cite{Kow:BVASOCAPB}; this includes the case for Hochschild cohomology as $(\Ae, A)$ can be seen as a Hopf algebroid (but not as a Hopf algebra). It is known that the Gerstenhaber structure on Hochschild cohomology does not always extend (or rather restrict) to a Batalin-Vilkoviski\u\i\ algebra structure but only in certain cases as, for example, symmetric algebras, certain Frobenius algebras, as well as (twisted) Calabi-Yau algebras \cite{Tra:TBVAOHCIBIIP, LamZhoZim:THCROAFAWSSNAIABVA, Gin:CYA, KowKra:BVSOEAT}; a sufficiency criterion when this is the case can be found in \cite{Kow:BVASOCAPB}.

On the other hand, if $U$ is a {\em braided} (in the sense of quasi-triangular) bialgebra over $\K$, then the Gerstenhaber bracket on $\Ext_U^\bullet(\K,\K)$ turns out to be zero \cite[Rem.~5.4]{Tai:IHBCOIDHAAGCOTYP}. At least for a cocommutative Hopf algebra, this can be directly seen from Eq.~\rmref{avvisoagliutenti} together with the fact that in this case one obtains $B=0$, as mentioned in \cite[p.~321]{Men:CMCMIALAM} again. In general, the 
possible cohomological vanishing of the Gerstenhaber bracket
can be considered a consequence from the fact that the binary operation up to homotopy on the cochain complex inducing it, is homotopic to zero. This homotopy induces in turn a degree $-2$ Lie bracket on cohomology, which together with the cup product makes it  
an $e_3$-algebra. The existence of an $e_3$-algebra structure on the cohomology of a braided Hopf algebra (or rather bialgebra) has been conjectured in \cite[Conj.~25]{Men:CMCMIALAM}, inspired by a conjecture attributed to Kontsevich in \cite{Sho:TOABFATMC} that this is the case for the Gerstenhaber-Schack cohomology of a Hopf algebra. This in turn amounts to an $\Ext$-group in the category of tetramodules (or Hopf bimodules) \cite[Cor.~3.9]{Tai:IHBCOIDHAAGCOTYP}, or, in the finite dimensional case, equivalently to an $\Ext$-group over the Drinfel'd double as a braided bialgebra. In \cite{Sho:TOABFATMC, Sho:DGCADC}, Shoikhet developed an approach based on $n$-fold monoidal abelian categories that implies that the Gerstenhaber-Schack cohomology groups indeed do carry the structure of an $e_3$-algebra. 
More recently, Ginot and Yalin \cite{GinYal:DTOBHHCAF} showed how to obtain this from
the $E_3$-algebra structure on the higher Hochschild complex of
a $E_2$-algebras given by the solution to the higher Deligne conjecture. This indeed implies the existence of an $E_3$-structure on the deformation complex of a dg bialgebra, but an explicit expression for the degree $-2$ Lie bracket remains elusive.

There is, however, quite an explicit degree $-2$ Lie bracket on the negative cyclic cohomology $HC_-^\bullet(A)$ of, for example, a symmetric algebra $A$, obtained by Menichi \cite{Men:BVAACCOHA} generalising the construction of Chas-Sullivan's string topology bracket \cite{ChaSul:ST}. 


Since negative cyclic cohomology is related to Hochschild cohomology by a long exact sequence (that is, a version of the so-called $SBI$ sequence), it is therefore tempting to presume that a possible braiding implies suitable vanishings 
such that the Chas-Sullivan-Menichi bracket can be transferred to the Hochschild cohomology,
inducing the Lie bracket of the $e_3$-algebra structure.
\par
This motivated us to a deeper investigation of Menichi's construction, leading to the generalisation described in the present article. Namely, Menichi obtains his degree $-2$ bracket on $HC_-^\bullet(A)$ as a particular case of a more general construction, where the endomorphism operad $\mathrm{Hom}_\K(A^{\otimes \bullet},A)$ is replaced by an arbitrary cyclic operad $\cO$ with multiplication in the category of $\K$-modules. Here we remark how this is, in turn, best seen as a particular instance of a construction of a degree $-d-2$ Lie bracket on the negative cyclic cohomology of a cyclic $\K$-module $\cM$ over $\cO$ endowed with a distinguished degree $d$ cocycle inducing an isomorphism $H^\bullet(\cO)\xrightarrow{\sim} H^{\bullet+d}(\cM)$. 
The Menichi bracket then corresponds to the special case of the cyclic operad $\cO$ considered as a module over itself, with distinguished degree zero cocycle given by the unit element. 
\par
The degree $-d-2$ Lie bracket on $HC_-^\bullet(\cM)$ is obtained as a consequence of the existence of a noncommutative (or Cartan-Gerstenhaber) differential calculus up to homotopy on the cocyclic bicomplex 
of $\cM$. We briefly remind the reader here that by this we essentially mean the pair of 
a Gerstenhaber algebra $(\mathcal{G}, \gd, \smallsmile, \{.,.\})$ up to homotopy
and a mixed complex $(M, b, B)$ such that $M$ is both a module and a Lie module over the dg-Lie algebra $\mathcal{G}[1]$
 by means of a contraction $\iota$ resp.\ a Lie derivative $\cL$, which obey a certain Cartan homotopy generalising the known one from differential geometry, see \S\ref{cartan} for all details. Since we explicitly need to remain on the level of (co)chains instead of descending to (co)homology, various homotopies $\cS$ and $\cT$ come into play. This has been more or less explicitly outlined in \cite{GelDalTsy:OAVONCDG}, where the case of Hochschild complex of an associative superalgebra is considered. As soon as a special (Poincar\'e) duality element is given, one has a natural Lie bracket on the cyclic cohomology of $M$, generalising the de Thanhoffer de V\"olcsey-Van den Bergh bracket, as well as a Lie bracket on the shifted negative cyclic cohomology of $M$, generalising the Chas-Sullivan-Menichi bracket. Moreover, the shifted cohomology $H^{\bullet-1}(M)$ of $M$ carries a Batalin-Vilkoviski\u\i\ algebra structure and, 
if the Gerstenhaber bracket vanishes, even 
an $e_3$-algebra structure.

The existence of 
a noncommutative differential calculus up to homotopy 
on an operadic (bi)module $\cM$ over an operad $\cO$ (in the category of $\K$-modules) is not direct, but rather a corollary of the fact that the $\K$-linear dual $\cM^*$ carries a natural structure of a so-called {\em cyclic opposite module} over the operad $\cO$ in a sense described in \cite[\S3]{Kow:GABVSOMOO} (called {\em cyclic comp modules} therein): more precisely, it does not come as a surprise that $\cM^*$ becomes an opposite module over $\cO$ in the sense of {\it loc.~cit.} 
If then $\cO$ is even a cyclic operad (with multiplication), by introducing the notion of a {\em cyclic $\cO$-module}, its linear dual is a cyclic opposite $\cO$-module, which means that it carries the structure of a noncommutative (or Cartan-Gerstenhaber) differential calculus up to homotopy. 
By adjunction via the natural pairing between $\cM$ and $\cM^*$, this reflects to $\cM$ and yields a calculus on $\cO$-modules rather than the already known construction using opposite modules from \cite{Kow:GABVSOMOO}. 

We therefore directly consider the situation
 in which a cyclic opposite module $\cN$ is given over a (not necessarily cyclic) operad with multiplication $\cO$, and are able to present 
a complete construction of a Cartan-Gerstenhaber calculus associated with the pair $(\cO,\cN)$, 
the trickiest part of which is exhibiting the (Gel'fand-Daletski\u\i-Tsygan) homotopy $\cT$.
A Poincar\'e duality element induces a natural Lie bracket on the (shifted) negative cyclic homology of $\mathcal{N}$, and 
 one recovers the result in \cite[Thm.~10.2]{dTdVVdB:CYDANCH} 
that inspired this construction
by simply applying it to the endomorphism operad for a $d$-Calabi-Yau algebra along with the cyclic opposite module given by the Hochschild chain spaces. The operadic formulation given here allows for the inclusion of more general situations, such as {\em twisted} Calabi-Yau algebras, that is, those where the bimodule structure of the algebra over itself is twisted by a Nakayama automorphism, which happens, {\it e.g.}, for standard quantum groups, 
Koszul algebras whose Koszul dual is Frobenius as, for example, 
Manin's quantum plane or still the Podle\'s quantum $2$-sphere \cite{BroZha:DCATHCFNHA, VdB:ARBHHACFGR, Kra:OTHCOQHS}.

\subsection{Main results}

Our first main result, Theorem \ref{fapurefreddoquadentro}, then reads as follows; see the main text for all technical details and notation.

{\renewcommand{\thetheorem}{{A}}
\begin{theorem}
\label{A}
Let $M$ be a mixed complex.
The Lie bracket on $HC^\bullet(M)$ induced by a homotopy Cartan-Gerstenhaber calculus with a duality cocycle has the form
\begin{equation*}
[{z},{w}]=(-1)^{z-1}\beta((\pi {z})\smallsmile (\pi {w})),
\end{equation*}
where $\pi: HC^\bullet(M) \to H^\bullet(M)$ and 
$\beta: H^\bullet(M) \to HC^{\bullet -1 }(M)$ are the canonical maps appearing in Connes' long exact sequence in cyclic cohomology.
\end{theorem}
}

Although not every mixed complex arises from a cyclic or cocyclic module, 
the way we typically (but not necessarily) think of the mixed complex in Theorem \ref{A} is by considering $M$ as a {\em cosimplicial} complex with a {\em cocyclic} operator $\tau$, from which one builds the cyclic boundary $B$ (which will be later called $\Bc$ for better distinction) of degree $-1$ in this case.
Observe at this point that applying Theorem \ref{A} to the complex $M^i := N_{-i}$, where $N$ is a {\em chain} complexes with a cyclic {\em co}boundary $B$ of degree $+1$, one obtains a statement on {\em negative cyclic homology} of $N$.

The probably most powerful class of examples for this theorem is given by the so-called cyclic opposite modules over operads with multiplication. 
To this end, the noncommutative calculus operations for cyclic opposite modules over operads with multiplication given in \cite{Kow:GABVSOMOO} has to be extended by explicitly giving a formula for the so-called  {\em Gelfan'd-Daletski\u\i-Tsygan homotopy} $\cT$, which reads:
\begin{equation*}
\begin{split}
&\cT: \cO(p) \otimes \cO(q) \otimes \cN(n) \to \cN(n-p-q+2), \\
&\quad
(\gvf,\psi,x) \mapsto 
\displaystyle \sum^{p-1}_{j=1}  \sum^{p-1}_{i=j} (-1)^{n(j-1) + (q-1)(i-j) + p} (\gvf \circ_{p-i+j} \psi) \bullet_0 t^{j-1}(x),
\end{split}
\end{equation*}
where $\cO$ is an operad with multiplication and $\cN$ a cyclic opposite $\cO$-module.
In the context of associative algebras, this homotopy has been alluded to in \cite{GelDalTsy:OAVONCDG}, but to our knowledge no explicit expression has appeared in the literature before, let alone in the operadic context. We then obtain (see again the main text for notation and definitions):

{\renewcommand{\thetheorem}{{B}}
\begin{theorem}
\label{B}
Writing $\cT(\gvf, \psi) := \cT(\gvf, \psi, \cdot)$,
one separately has for all $\gvf, \psi \in \cO$
\begin{equation*}
[\iota_\psi, \cL_\gvf] - \iota_{\{\psi, \gvf\}} = [b, \cT(\gvf,\psi)] - \cT(\gd \gvf, \psi) - (-1)^{\gvf-1} \cT(\gvf, \delta \psi) 
\end{equation*}
on $\cN$ along with 
\begin{equation*}
[\cS_\psi, \cL_\gvf] - \cS_{\{\psi,\gvf\}} = [B, \cT(\gvf,\psi)], 
\end{equation*}
for $\gvf, \psi \in \widebar{\cO}$ on the normalised complex $\overline{\cN}$.
Both formulae unite to give the Gelfan'd-Daletski\u\i-Tsygan homotopy formula
\begin{equation*}
[\iota_\psi + \cS_\psi, \cL_\gvf] - \iota_{\{\psi,\gvf\}} - \cS_{\{\psi,\gvf\}} = [b+B, \cT(\gvf,\psi)] - \cT(\gd \gvf, \psi) - (-1)^{p-1} \cT(\gvf, \delta \psi)
\end{equation*}
on $\overline{\cN}$ for $\gvf, \psi \in \widebar{\cO}$.
\end{theorem}
}

Putting together the two preceding theorems, we formulate as a corollary (Theorem \ref{omonicesoir} in the main text):

{\renewcommand{\thetheorem}{{C}}
\begin{cor}
Let $(\cO, \mu, e)$ be an operad with multiplication and $\cN$ a cyclic unital opposite $\cO$-module. Assume we have Poincar\'e duality between $\cO$ and $\cN$, i.e., there is a cocycle $\zeta \in \cN(d)$ such that the map 
$
\cO \to \cN$ defined by $\gvf \mapsto i_\gvf \zeta = \gvf \smallfrown \zeta  
$
induces an isomorphism
$
H^{d-p}(\cO) \cong H_{p}(\cN)$.
Then the negative cyclic homology $HC^{-}_\bullet(\cN)$ carries a degree $(1-d)$ Lie bracket 
defined by
\[
[{z},{w}]=(-1)^{z+d}\beta((\pi {z})\smallsmile (\pi {w})),
\]
where $\pi: HC^-_p(\cN) \to H_p(\cN)$ and $\beta\colon H_p(\cN)\to HC^{-}_{p+1}(\cN)$ are the natural operations in Connes' long exact sequence in negative cyclic homology, and   
where
$
\smallsmile\colon  H_p(\cN)\otimes H_q(\cN)\to H_{p+q-d}(\cN)
$
is induced via the Poincar\'e duality.
\end{cor}
}

In Section \ref{sete}, we consider a dual approach, in the sense that instead of restricting the homotopy calculus from the periodic cocyclic complex to the cocyclic bicomplex, we 
consider the induced calculus on the quotient complex leading to negative cyclic cohomology. If then again a duality cocycle is given, we obtain in Theorems \ref{ruwenogien} \& \ref{chas-sullivan-menichi}:

{\renewcommand{\thetheorem}{{D}}
\begin{theorem}
\label{D}
Let $M$ be a mixed complex. Then a homotopy Cartan-Gerstenhaber calculus with a duality cocycle induces a Batalin-Vilkoviski\u\i\ algebra structure $(H^{\bullet-1}(M),\{\cdot,\cdot\},\smallsmile, B)$ on the (shifted) cohomology of $M$. Moreover, the degree $-2$ {\em Chas-Sullivan-Menichi bracket} 
\[
[\cdot,\cdot]\colon HC^{\bullet-1}_-(M)\otimes HC^{\bullet-1}_-(M)\to HC^{\bullet-1}_-(M)[-2]
\]
on (shifted) negative cyclic cohomology, defined by
\[
[{x},{y}] := (-1)^x  j ((\beta{x})\smallsmile(\beta{y})),
\]
is a Lie bracket and has the property 
$$
\beta[\cdot, \cdot] = \{ \beta(\cdot), \beta(\cdot)\},
$$
where $j: H^\bullet(M) \to HC^\bullet_-(M)$ and $\beta: HC^{\bullet}_-(M) \to  H^{\bullet-1}(M)$ are the canonical maps appearing in the long exact sequence relating the cohomology of  $M$ to its negative cyclic cohomology.
\end{theorem}
}



As discussed before, an interesting situation arises when the BV-bracket $\{\cdot, \cdot\}$ vanishes on $H^{\bullet-1}(M)$.
An immediate consequence is then Theorem \ref{arkopharma}, which we state as a corollary here:

{\renewcommand{\thetheorem}{{E}}
\begin{cor}
\label{E}
If $\{\cdot,\cdot\}$ vanishes identically on $H^{\bullet-1}(M)$, then 
\[
\{\!\!\{ x,y\}\!\!\}:=(-1)^x (\Bc{x})\smallsmile(\Bc{y})
\]
defines a degree $-2$ Lie bracket on $H^{\bullet-1}(M)$ such that $j\{\!\!\{ x,y\}\!\!\}=[jx,jy]$ and $\Bc\{\!\!\{ x,y\}\!\!\}=0$. This bracket turns $\big(H^{\bullet-1}(M),\smallsmile,\{\!\!\{\cdot,\cdot\}\!\!\}\big)$ into an $e_3$-algebra.
\end{cor}
}

We already mentioned that a large class of examples to which Theorem \ref{D} can be applied is given by (bi)modules over operads if they are endowed with a cyclic operation so as to obtain a homotopy calculus on them. More precisely, in Theorem \ref{nochkeinname} we show:

{\renewcommand{\thetheorem}{{F}}
\begin{theorem}
\label{F}
The datum of a cyclic module $(\cM, \tau)$ over a cyclic operad with multiplication $(\cO, \mu, e, \tau)$ induces a noncommutative differential calculus on $\cM$ over $\cO$.
\end{theorem}
}

In particular, we can put $\cM = \cO$, and this way, along with the aforementioned Theorem \ref{D}, we obtain in Theorem \ref{nocheins} what we formulate as a corollary here:

{\renewcommand{\thetheorem}{{G}}
\begin{cor}
\label{G}
For a cyclic operad with multiplication $\cO$, one has
\begin{equation*}
\begin{split}
\{\psi,\gvf\} &= 
(-1)^\psi\bigl(\cL_\psi \gvf -(-1)^{\psi\gvf} \cL_\gvf\psi - B(\psi \smallsmile \gvf)\bigr) \\
&=
 - \psi \smallsmile B(\gvf) + (-1)^{(\gvf-1)\psi} B(\gvf \smallsmile \psi) - (-1)^{(\gvf-1)(\psi-1)} \gvf \smallsmile B(\psi)  
\\
        &\quad  - (-1)^{(\gvf-1)\psi} \delta\big(\cS_\gvf \psi)  + (-1)^{(\gvf-1)\psi} \cS_{\gd \gvf} \psi  + (-1)^{(\gvf-1)(\psi-1)} \cS_{\gvf} \gd \psi
\\
&\quad
+ (-1)^{\gvf} \delta(\cS_\psi\gvf)  +  (-1)^\psi \cS_{\psi}\gd\gvf + (-1)^{\psi+\gvf-1} \cS_{\gd \psi} \gvf
\end{split}
\end{equation*}
on the normalised complex $\overline{\cO}$. In particular, the cohomology groups $H^\bullet(\cO)$ carry the structure of a  Batalin-Vilkoviski\u\i\ algebra.
\end{cor}
}

The last statement in Corollary \ref{G} has been first proven in \cite[Thm.~1.4]{Men:BVAACCOHA}. Combining Theorems \ref{D} \& \ref{F}, one recovers Corollary 1.5 in {\em op.~cit.}, which states that the negative cyclic cohomology of a cyclic operad is endowed with a Lie bracket of degree $-2$.

\section{Preliminaries}

In this section, we recall some general standard facts about cyclic (co)homology in the form they will be needed in the subsequent sections; see, for example, \cite{Lod:CH} for an exhaustive treatment.

In order to fix notation, throughout the whole paper we will work over a fixed characteristic zero field $\K$, and $(M^\bullet, \bc)$ will be a cochain complex over $\K$, {\it i.e.}, a $\mathbb{Z}$-graded $\K$-vector space $M^\bullet=\bigoplus_{i\in \Z}M^i$ together with a degree $1$ differential
\[
\bc \colon M^\bullet\to M^\bullet[1]
\]
such that $\bc^2=0$. Here $M^\bullet[1]$ denotes the shifted graded vector space $(M^\bullet[1])^i=M^{i+1}$. We denote by $H^\bullet(M)$ the cohomology of $(M, \bc)$.

\begin{example}\label{chain-to-cochain} Let $(N_\bullet, b)$ be a \emph{chain} complex over $\K$, {\it i.e.}, a $\mathbb{Z}$-graded $\K$-vector space $N_\bullet=\bigoplus_{i\in \Z} N_i$ together 
with a degree $-1$ differential $b \colon N_i\to N_{i - 1}$ such that $b^2=0$. Out of $(N_\bullet, b)$ one can define two cochain complexes. The first one consists in setting $M^i=N_{-i}$ with differential $\bc \colon M^{i} \to M^{i+1}$ given by $b \colon N_{-i}\to N_{-i-1}$. In this case we have $H^\bullet(M)=H_{-\bullet}(N)$, where $H_i(N)$ denotes the $i$-th homology group of $(N_\bullet, b)$. The second one consists in setting $M^i=\mathrm{Hom}_\K(N_i,\K)$ with differential $\bc$ given by the adjoint of $b$, that is, by $b^* \colon \mathrm{Hom}_\K(N_i,\K)\to \mathrm{Hom}_\K(N_{i+1},\K)$. In this case we have $H^{\bullet}(N) := H^\bullet(M)$, and we call $H^i(N)$ the $i$-th {\em cohomology} group of $(N_\bullet, b)$. 
\end{example}

\begin{definition}
A \emph{cyclic} differential on $(M^\bullet,\bc)$ is a degree $-1$ differential
\[
\Bc\colon M^\bullet\to M^\bullet[-1]
\]
such that $\Bc^2=0$ and $[\Bc,\bc]=0$. Here $[\Bc,\bc]$ is the graded commutator $[\Bc,\bc]=\Bc\bc+\bc\Bc$ and such a datum $(M^\bullet, \bc, \Bc)$ is usually referred to as 
{\em mixed complex} \cite{Kas:CHCAMC}.
\end{definition}

\begin{example}
\label{ancorauno}
If $(N_\bullet,b)$ is a chain complex, a cyclic differential on $N_\bullet$ is a degree $1$ differential $B$ such that $B^2=0$ 
and $[B,b]=0$ such that  $(N_\bullet,b, B)$ is a mixed complex.
It naturally induces cyclic differentials on the cochain complexes $M^i=N_{-i}$ and $M^i=\mathrm{Hom}_\K(N_i,\K)$ of Example \ref{chain-to-cochain}.
\end{example}

\begin{rem}
Thus a mixed complex is both a chain
and a cochain complex and a priori there is a complete symmetry between the given boundary and coboundary operator. However, in the spirit of \cite{Kas:CHCAMC}, we shall view  $\bc$ (resp.\ $b$) as the ``main''
differential and consider $\Bc$ (resp.\ $B$) as an additional datum that ``perturbs'' the complex $(M^\bullet,\bc)$ (resp.\  $(N_\bullet,b)$). To better distinguish whether we consider a cochain or a chain complex as the main complex, we will sometimes use the (slightly misleading) terminology of {\em mixed cochain (resp.\ chain) complex}. 
\end{rem}

\begin{rem}
\label{rem.Bc-in-cohomology}
As $[\Bc,\bc]=0$, the differential $\Bc$ is a morphism of complexes $\Bc\colon (M^\bullet,\bc)\to (M^\bullet[-1],\bc[-1])$ and so it induces a linear operator (which we will denote by the same symbol) $\Bc\colon H^\bullet(M)\to H^{\bullet-1}(M)$. This operator can be seen as a degree $-1$ differential on $H^\bullet(M)$.
\end{rem}

It is convenient to introduce a degree $+2$ variable $u$ and consider the graded vector space $M^\bullet[[u,u^{-1}]]$, whose graded components are
\[
M[[u,u^{-1}]]^n=\prod_{i+2j=n}M^iu^j.
\]
Then the two differentials $\bc$ and $\Bc$ on $M^\bullet$ define the $\K[u,u^{-1}]$-linear differential of  degree $1$
\[
d_u=\bc+u\Bc
\]
on $M^\bullet[[u,u^{-1}]]$. 

\begin{definition}
Let $(M^\bullet,\bc,\Bc)$ be a mixed cochain complex. The \emph{periodic cocyclic complex} $CC_{\mathrm{per}}^\bullet(M)$ of $M^\bullet$ is the cochain complex $(M^\bullet[[u,u^{-1}]],d_u)$.
The 
\emph{cocyclic complex} $CC^\bullet(M)$
is the subcomplex $(M^\bullet[[u]],d_u)$ of $CC_{\mathrm{per}}^\bullet(M)$; its cohomology will be called the {\em cyclic cohomology} of $M^\bullet$ and denoted by $HC^\bullet(M)$. The 
{\em negative cocyclic complex} $CC^\bullet_-(M)$ 
is the quotient complex $(M^\bullet[[u,u^{-1}]]/uM^\bullet[[u]],d_u)$ of $CC_{\mathrm{per}}^\bullet(M)$; its cohomology will be called the {\em negative cyclic cohomology} of $M^\bullet$ and denoted $HC^\bullet_-(M)$. 
\end{definition}

\begin{dfn}\label{negative-cyclic-homology}
\label{brot}
Let $(N_\bullet,b,B)$ be a mixed (chain) complex and let $(M^\bullet,\bc,\Bc)$ be the mixed (cochain) complex defined by $M^i := N_{-i}$. 
Define then 
\begin{equation}
\label{trimalchio}
HC^-_{-\bullet}(N) := HC^\bullet(M),
\end{equation}
and we call $HC^{-}_i(N)$ the {\em $i$-th negative cyclic homology group} of $(N_\bullet,b,B)$. 
If instead $(M^\bullet,\bc,\Bc)$ is the mixed (cochain) complex defined by $M^i := \mathrm{Hom}_\K(N_i,\K)$, 
we then define 
$
HC_-^{\bullet}(N) := HC^\bullet_-(M),
$
and call $HC_{-}^i(N)$ the {\em $i$-th negative cyclic cohomology group} of the mixed (chain) complex $(N_\bullet,b,B)$. 
\end{dfn}

\begin{rem}
We remind the reader that in our definition the cochain complex $M$ is {\em not} concentrated in positive degree. By drawing a picture of the full bicomplex, it is clear that in the cohomology case the negative cocyclic complex is truncated. For a chain complex $N$ proceeding as in Eq.~\rmref{trimalchio} and analogously for the definition of the cyclic homology of $N$, a point reflection at the origin brings the bicomplex into the customary form where the cyclic complex is the truncated one.

For completeness, although we are not going to use this fact, we also mention that $HC^\bullet$ can be seen as dual to $HC^-_\bullet$ if considered as a $\K[u]$-module \cite[\S5.1.17]{Lod:CH}.
\end{rem}

\begin{lemma}
\label{connecting}
There is a short exact sequence of complexes
\begin{equation}
\label{taylor-short}
0\to CC^\bullet(M)[-2]\xrightarrow{u} CC^\bullet(M)\xrightarrow{\mathrm{ev}_0} M^\bullet\to 0,
\end{equation}
where the first map is multiplication by $u$ and the second map is evaluation at $u=0$. As a consequence, we have a long exact sequence
\begin{equation}
\label{taylor-long}
\cdots \to HC^{n-2}(M)\xrightarrow{S} HC^n(M)\xrightarrow{\pi} H^n(M)\xrightarrow{\beta} HC^{n-1}(M)\to \cdots.
\end{equation}
\end{lemma}

\begin{rem}\label{trivial-but-useful}
For later use, observe that the subcomplex $(\ker \Bc, \bc)$ of $(M^\bullet,\bc)$ is a subcomplex of 
$\big(CC^\bullet(M), d_u \big)$ via the obvious inclusion $M^\bullet\subseteq M^\bullet[[u,u^{-1}]]$, whereas $(M^\bullet, \bc)$ is in general not.
\end{rem}

\begin{rem}
\label{lanciano}
The long exact sequence in \rmref{taylor-long} is of course Connes' long exact sequence (the {\em $SBI$-sequence}) in cohomological form, see, for example, \cite[\S2.4.4]{Lod:CH}. In this language, the periodicity operator $S$ is simply given by multiplication with $u$. Keeping in mind Definition \ref{brot}, observe furthermore that for a mixed (chain) complex $(N_\bullet,b,B)$, the short exact sequence \rmref{taylor-short} reads
$$
0\to CC^-_\bullet(N)[2]\xrightarrow{} CC^-_\bullet(N)\xrightarrow{} N_\bullet\to 0,
$$
which in turn induces the long exact sequence
\begin{equation}
\label{exact-sequence-negative-cyclic}
\cdots \to HC^-_{n+2}(N)\to HC^-_n(N)\xrightarrow{\pi} H_n(N)\xrightarrow{\beta} HC^-_{n+1}(N)\to \cdots,
\end{equation}
connecting the negative cyclic homology of $(N_\bullet,b,B)$ to the homology of $(N_\bullet,b)$.
\end{rem}

\begin{lemma}\label{connecting-one}
The connecting homomorphism $\beta \colon H^n(M)\to HC^{n-1}(M)$ in the long exact sequence \rmref{taylor-long} is given by $\beta[m]=[\Bc m]$.
\end{lemma}
\begin{proof}
Given a $\bc$-cohomology class $[m]$ in $H^n(M)$, consider its representative element $m$ in $M^n$, and choose a power series $x(u)$ in $CC^n(M)$ such that $x(0)=m$. An obvious choice is to take the constant power series $x(u)\equiv m$. We apply the differential $d_u$ to this power series getting the element 
\[
d_um=\bc m+u\Bc m=u\Bc m,
\]
where we used the fact that $m$ is a $\bc$-cocycle. Finally, we have to take the cohomology class of an element in $CC^{n-1}(M)$ which is mapped to $u\Bc m$ by multiplication with $u$, and this is obviously $\Bc m$. 
\end{proof}

\begin{lemma}
\label{connecting2}
We have a short exact sequence of complexes
\[
0\to M^\bullet\to CC^\bullet_-(M)\to CC^\bullet_-(M)[2]\to 0,
\]
where the first map is the inclusion of $M^\bullet$ as (equivalence classes of) constant Laurent series with coefficients in $M^\bullet$ and the second map is the multiplication by $u$. As a consequence, we have a long exact sequence in cohomology
\begin{equation}
\label{laurent-long}
\cdots \to  H^n(M)\xrightarrow{j} HC^n_-(M)\xrightarrow{S} HC^{n+2}_-(M)\xrightarrow{\beta} H^{n+1}(M)\to \cdots, 
\end{equation}
connecting negative cyclic cohomology of $(M^\bullet,\bc,\Bc)$ to the cohomology of $(M^\bullet,\bc)$.
\end{lemma}
\begin{proof}
As multiplication by $u$ maps the subcomplex $uM^\bullet[[u]]$ into itself, it defines a morphism of complexes
\[
u\colon CC^\bullet_-(M)\to CC^\bullet_-(M)[2],
\]
which is clearly surjective as the morphism
\[
u\colon CC_{\mathrm{per}}^\bullet(M)\to CC_{\mathrm{per}}^\bullet(M)
\]
is an isomorphism. Therefore we have a short exact sequence of complexes
\[
0\to \ker(u)\to CC^\bullet_-(M)\xrightarrow{u} CC^\bullet_-(M)[2]\to 0,
\]
and we are left to determine the subcomplex $\ker(u)$. By definition, an element in $\ker(u)$ is represented by the constant term and  polar part of a Laurent series $f(u)$ in the variable $u$ (with coefficients in $M^\bullet$) such that $uf(u)$ is an element in $uM[[u]]$. This last condition implies that $f(u)$ is actually an element in $M[[u]]$ so it is (uniquely) represented by an element in $M^\bullet$ (corresponding to its constant term).
\end{proof}

\begin{lemma}
\label{B-returns}
The connecting homomorphism $\beta\colon HC^{n}_-(M)\to H^{n-1}(M)$ from the long exact sequence \rmref{laurent-long} is given by $\beta [f]= [\Bc f_0]$, where $f_0$ is the coefficient of $u^0$ in the Laurent series $f$.
In particular, the composition 
\begin{equation}
\label{B-davvero-returns}
H^\bullet(M)\xrightarrow{j}HC^\bullet_-(M)\xrightarrow{\beta}H^{\bullet-1}(M)
\end{equation}
on cohomology is the operator $\Bc\colon H^\bullet(M)\to H^{\bullet-1}(M)$ from Remark \ref{rem.Bc-in-cohomology}.
\end{lemma}
\begin{proof}
Given a cohomology class $[f]$ in $HC^{n}_-(M)$, let us pick a representative $f$ in $CC^{n}_-(M)$. We can choose $f$ to be a Laurent series with only constant term and  polar part such that $d_uf$ lies in $uM^\bullet[[u]]$. 
Writing
\[
f=\sum_{n=-\infty}^0 f_{-n} u^{n}
\]
and $
d_uf=\bc f+u\Bc f$, 
this means that $
d_uf-u\Bc f_0=0$, that is to say
\begin{equation}\label{B-appears}
d_uf=u\Bc f_0.
\end{equation}
To compute $\beta[f]$ we have to pick a preimage of $f$ in $CC^{n-2}_-(M)$ for the map given by the multiplication by $u$. Such a preimage is clearly given by the (equivalence class of) the Laurent series $u^{-1}f$. Apply the differential $d_u$ to this element to get
\[
d_u(u^{-1}f)=u^{-1}d_uf=\Bc f_0,
\]
by Eq.~\rmref{B-appears}. Now we only need to pick up an element in $M^{n-1}$ mapped to $\Bc f_0$ by the inclusion $M^{n-1}\to CC^{n-1}_-(M)$, which clearly is $\Bc f_0$ itself, and to take its cohomology class. 
\end{proof}

\section{Cartan calculi}
\label{cartan}

In this section, we introduce the notion 
of a Cartan-Gerstenhaber calculus up to homotopy in the spirit of the noncommutative differential calculus in the sense of Nest-Tamarkin-Tsygan \cite{NesTsy:OTCROAA, TamTsy:NCDCHBVAAFC}.

\begin{definition}
Let $(M^\bullet,\bc,\Bc)$ be a mixed cochain complex, and let $(\mathfrak{g}^\bullet,\delta,\{\cdot,\cdot\})$ be a dg-Lie algebra. A {\em homotopy pre-Cartan calculus} of $\mathfrak{g}^\bullet$ on $CC_{\mathrm{per}}^\bullet(M)$ is the datum of a \emph{contraction operator} (or {\em cap product})
 \[
 \iota \colon \mathfrak{g}^\bullet\otimes M^\bullet\to M^\bullet[1],
 \]
of a \emph{Lie derivative}
\[
 \cL \colon \mathfrak{g}^\bullet\otimes M^\bullet\to M^\bullet,
\]
and of an operator
 \[
\cS \colon \mathfrak{g}^\bullet\otimes M^\bullet\to M^\bullet[-1]
 \]
such that
\begin{equation}
\label{dellera1}
\begin{cases}
\cL_f= [\Bc, \iota_f]+[\bc, \cS_f] + \cS_{\gd f},\\
[\bc, \iota_f]+\iota_{\gd f}=0,\\
[\Bc, \cS_f]=0,
\end{cases}
\end{equation}
when we look at $\iota_f$, $\cL_f$ and $\cS_f$ as endomorphisms of the graded vector space $M^\bullet$.
\end{definition}
\begin{lemma}
Extend $\iota_f,\cL_f$ and $\cS_f$ by $\K[u,u^{-1}]$-linearity to operators on $CC_{\mathrm{per}}^\bullet(M)$, and let 
\[
\cI \colon \mathfrak{g}^\bullet\otimes CC_{\mathrm{per}}^\bullet(M)\to CC_{\mathrm{per}}^\bullet(M)[1]
\]
be the operator $\cI= \iota + u \cS$. Then 
\begin{equation}
\label{schonlos2}
u\cL_f= [d_u,\cI_f]  + \cI_{\gd f}.
\end{equation}
Moreover, $\cI$ and $\cL$ induce, by restriction or by induced action on the quotient, operators (which we will denote by the same symbols) obeying relation \rmref{schonlos2} on $CC^\bullet(M)$ and $CC^\bullet_-(M)$.
\end{lemma}
\begin{proof}

It is immediate to check that the single Equation \rmref{schonlos2} is actually equivalent to the three Equations \rmref{dellera1}. As $\cI$ and $\cL$ preserve the subcomplexes $(M[[u]],d_u)$ and $(uM[[u]],d_u)$ of $CC_{\mathrm{per}}^\bullet(M)$, the second part of the statement follows.
\end{proof}

\begin{rem}\label{acts-trivially}
It follows from Eq.~\rmref{schonlos2} that $u\cL_f$ acts trivially  on $HC_{\mathrm{per}}^\bullet(M)$, on $HC^\bullet(M)$ and on $HC_{-}^\bullet(M)$. Since the multiplication by $u$ is invertible in $CC_{\mathrm{per}}^\bullet(M)$, also $\cL_f$ acts trivially  on $HC_{\mathrm{per}}^\bullet(M)$. However, since multiplication with $u$ is not invertible in $CC^\bullet(M)$ and on $CC_{-}^\bullet(M)$, the Lie derivative $\cL_f$ may act nontrivially on cyclic and negative cyclic cohomologies.
\end{rem}

For the evident similarity of the identity \rmref{schonlos2} that holds on the cyclic bicomplex with the classical Cartan formula, one might want to call the operator $\cI= \iota + u \cS$ as the {\em cyclic cap product}.

\begin{lemma}\label{lemma:is-a-morphsm-of-complexes}
The degree zero operator $\cL$ satisfies
\begin{equation}
\label{g-module-structure}
[d_u, \cL_f] - \cL_{\gd f}=0,
\end{equation}
and therefore defines a morphism of complexes $\cL\colon \mathfrak{g}^\bullet\to \mathrm{End}(CC_{\mathrm{per}}^\bullet(M)[n])$ for any $n\in \Z$. This induces morphisms of complexes $\cL\colon \mathfrak{g}^\bullet\to \mathrm{End}(CC^\bullet(M)[n])$ and on $\cL\colon \mathfrak{g}^\bullet\to \mathrm{End}(CC^\bullet_-(M)[n])$, for any $n\in \Z$.
\end{lemma}
\begin{proof}
We have
\begin{align*}
[d_u, u\cL_f] - u\cL_{\gd f}&=[d_u,[d_u,\cI_f] + \cI_{\delta f}] - [d_u,\cI_{\delta f}] - \cI_{\delta^2f}\\
&=\frac{1}{2}[[d_u,d_u],\cI_f] + [d_u,\cI_{\delta f}] - [d_u,\cI_{\delta f}] =0.
\end{align*}
Multiplying by $u^{-1}$ we find $[d_u, \cL_f] - \cL_{\gd f} = 0$.
\end{proof}
\begin{rem}
Equation \rmref{g-module-structure} is equivalent to the pair of equations
\[
\begin{cases}
[\bc, \cL_f] - \cL_{\gd f} = 0, \\
[\Bc, \cL_f]  = 0,
\end{cases}
\]
which can be easily directly derived from Eqs.~\rmref{dellera1}.
\end{rem}

\begin{rem}
 In particular, it follows from Lemma \ref{lemma:is-a-morphsm-of-complexes} 
that if $f$ is a cocycle in $\mathfrak{g}^\bullet$ and $\zeta$ is a cocycle in $CC_{\mathrm{per}}^\bullet(M)$  (resp., in $CC^\bullet(M)$, and in $CC^\bullet_-(M)$),  then $\cL_f\zeta$ is a cocycle in $CC_{\mathrm{per}}^\bullet(M)$  (resp., in $CC^\bullet(M)$, and in $CC^\bullet_-(M)$).
\end{rem}

\begin{definition}
\label{ministeroperibenieleattivitaculturali}
A \emph{homotopy Cartan calculus} on $CC^\bullet_{\mathrm{per}}(M)$ is a homotopy pre-Cartan calculus $(\mathfrak{g}^\bullet,\iota,\cL,\cS)$ on $CC_{\mathrm{per}}^\bullet(M)$ endowed with a {\em Gelfan'd-Daletski\u\i-Tsygan homotopy}, that is, with an operator
\[
\cT: \mathfrak{g}^\bullet \otimes \mathfrak{g}^\bullet \otimes M^\bullet \to M^\bullet
\]
such that

\begin{equation}
\label{nunja}
[\cI_f, \cL_g] - \cI_{\{f,g\}} = [d_u, \cT(f,g)] - \cT(\gd f, g) -(-1)^f \cT(f, \delta g), 
\end{equation}
where $\cT$ has been extended $\K[u,u^{-1}]$-linearly to $CC^\bullet_{\mathrm{per}}(M)$.
\end{definition}

\begin{rem}
Equation \eqref{nunja} is equivalent to 
\begin{equation}
\label{panem}
\begin{cases}
\!\!\!
\begin{array}{lcl}
[\iota_f, \cL_g] - \iota_{\{f, g\}} &=& [\bc, \cT(f,g)] - \cT(\gd f, g) -(-1)^f \cT(f, \delta g), \\
[\cS_f, \cL_g] - \cS_{\{f,g\}} &=& [\Bc, \cT(f,g)], 
\end{array}
\end{cases}
\end{equation}
\end{rem}

\begin{rem}
The operator $\cT$ and the identity \rmref{nunja} implicitly appeared in the context of associative algebras for the first time ({\em cf.} Example \ref{algebras}) in \cite{GelDalTsy:OAVONCDG}; hence its name. We will see more general examples arising from operads in \S\ref{pling}--\ref{plong}.
\end{rem}

\begin{lemma}
\label{lemma:is-a-morphism-of-lie-algebras}
Looking at $\cL$ as a map $\cL\colon \mathfrak{g}\to \mathrm{End}(M[[u,u^{-1}]])$, we have
\begin{equation}
\label{is-a-morphism-of-lie-algebras}
\cL_{\{f,g\}}=[\cL_f,\cL_g].
\end{equation}
\end{lemma}
\begin{proof}
Applying $[d_u,-]$ to both sides of Eq.~\rmref{nunja},
we get
\begin{align*}
[d_u,\cI_{\{f,g\}}]&=[d_u,[\cI_f, \cL_g]] -  [d_u,[d_u, \cT(f,g)]] + [d_u,\cT(\gd f, g)] +(-1)^f[d_u, \cT(f, \delta g)]\\
&=[[d_u,\cI_f], \cL_g]-(-1)^f [\cI_f,[d_u,\cL_g]] + [d_u,\cT(\gd f, g)]
\\
&
\qqquad\qqquad
 +(-1)^f[d_u, \cT(f, \delta g)].
\end{align*}
By means of \rmref{schonlos2} and \rmref{g-module-structure}, we rewrite the above identity as
\begin{align*}
u\cL_{\{f,g\}}- \cI_{\gd \{f,g\}}&=[u\cL_{f}-\cI_{\gd f},\cL_g]-(-1)^f [\cI_f,\cL_{\gd g}]+ [d_u,\cT(\gd f, g)] 
\\
&
\qqquad\qqquad
+(-1)^f[d_u, \cT(f, \delta g)]
\\
&=
u[\cL_f,\cL_g]-[\cI_{\gd f},\cL_g]-(-1)^f [\cI_f,\cL_{\gd g}]+ [d_u,\cT(\gd f, g)] 
\\
&
\qqquad\qqquad
+(-1)^f[d_u, \cT(f, \delta g)].
\end{align*}
By the Gelfan'd-Daletski\u\i-Tsygan homotopy \rmref{nunja}, we have
\begin{equation*}
\begin{split}
[\cI_{\gd f},\cL_g]&=\cI_{\{\gd f,g\}} +[d_u, \cT(\gd f,g)] - \cT(\gd^2 f, g) +(-1)^f \cT(\gd f, \delta g)
\\
&=\cI_{\{\gd f,g\}} + [d_u, \cT(\gd f,g)] +(-1)^f \cT(\gd f, \delta g)
\end{split}
\end{equation*}
and
\begin{equation*}
\begin{split}
[\cI_f, \cL_{\gd g}]&=\cI_{\{f,\gd g\}} + [d_u, \cT(f,\gd g)] - \cT(\gd f, \gd g) -(-1)^f \cT(f, \delta^2 g)
\\
&=\cI_{\{f,\gd g\}} + [d_u, \cT(f,\gd g)] - \cT(\gd f, \gd g).
\end{split}
\end{equation*}
Hence, we get
\begin{align*}
-[\cI_{\gd f},\cL_g]&-(-1)^f [\cI_f,\cL_{\gd g}]+ [d_u,\cT(\gd f, g)] +(-1)^f[d_u, \cT(f, \delta g)] \\
&=-\cI_{\{\gd f,g\}}  -(-1)^f\cI_{\{f,\gd g\}} \\
&= -\cI_{\gd\{f,g\}},
\end{align*}
and so finally
\[
u\cL_{\{f,g\}}=u[\cL_f,\cL_g].
\]
Since multiplication by $u$ is an automorphism of $M[[u,u^{-1}]]$, this gives $\cL_{\{f,g\}}=[\cL_f,\cL_g]$. 
\end{proof}

From Lemmata \ref{lemma:is-a-morphsm-of-complexes} \& \ref{lemma:is-a-morphism-of-lie-algebras} together, we obtain the following:

\begin{corollary}
The Lie derivative $\cL$ defines a $\mathfrak{g}^\bullet$-dg-module structure on $CC^\bullet_{\mathrm{per}}(M)[n]$ for any $n\in \Z$. This induces a $\mathfrak{g}^\bullet$-dg-module structure both on $CC^\bullet(M)[n]$ as well as on $CC^\bullet_-(M)[n]$, for any $n\in \Z$.
\end{corollary}

In view of Lemmata \ref{connecting} \& \ref{connecting2}, we will be mainly interested in the case $n=2$ and $n=-2$.

\begin{example}
\label{forms'n'fields}
For a smooth manifold $P$, consider the dg-Lie algebra $(\cX(P), 0, [.,.]_{\mathrm{\scriptscriptstyle{SN}}})$ of smooth multivector fields equipped with the Schouten-Nijenhuis bracket, along with the mixed complex $(\gO(P), 0, d_{\mathrm{\scriptscriptstyle{dR}}})$ of differential forms equipped with the de Rham differential $d = d_{\mathrm{\scriptscriptstyle{dR}}}$. In this case, $\iota$ and $\cL$ are the customary contraction resp.\ Lie derivative of forms along (multivector) fields, whereas $\cS$ and $\cT$ can be chosen almost arbitrarily as the chain differential vanishes: for example, take $\cS=\cT=0$. Eqs.~\rmref{dellera1}--\rmref{nunja} then become the well-known identities
\begin{equation}
\label{forms}
\cL = [\iota, d], \quad [d,\cL] = 0, \quad [\cL, \iota] = \iota_{[.,.]_{\scriptscriptstyle{SN}}}, \quad \cL_{[.,.]_{\scriptscriptstyle{SN}}} = [\cL, \cL]
\end{equation}
from differential (or algebraic) geometry. 
Dually, the dg-Lie algebra $(\cX(P), 0, [.,.]_{\mathrm{\scriptscriptstyle{SN}}})$ also acts on itself, or rather on $(\cX(P), 0, d_{\mathrm{\scriptscriptstyle{CE}}})$, where the differential 
$d_{\mathrm{\scriptscriptstyle{CE}}}$ is given by the customary Chevalley-Eilenberg differential that defines Lie algebra homology. Somewhat less present in the literature, defining $\iota_X(Y) = X \wedge Y$ for two multivector fields $X$ and $Y$, one obtains formulae analogous to those in \rmref{forms} when using the Lie derivative on multivector fields along multivector fields themselves.
We will, however, see more interesting examples in the forthcoming sections.
\end{example}

\begin{example}
\label{algebras}
Let $A$ be an associative algebra over a commutative ring $k$, and consider the mixed complex $(\overline{C}_\bull(A, A), b, B)$ of (normalised) Hochschild chains $\overline{C}_n(A,A) := A \otimes \overline{A}^{\otimes n}$ with Hochschild boundary $b$ and Connes' cyclic coboundary $B$, that is, 
\begin{footnotesize}
\begin{eqnarray*}
b(a_0, \ldots, a_n) &\!\!\!\!\!=&\!\!\!\!\! \sum^{n-1}_{i=0} (-1)^{i}  
(a_0, \ldots, a_{i} a_{i+1}, \ldots, a_n) + (-1)^{n}  (a_{n} a_0 , a_1, \ldots, a_{n-1}),
\\
B(a_0, \ldots, a_n)  &\!\!\!\!\!=&\!\!\!\!\! \sum^n_{i=0} (-1)^{in} (1, a_{i+1}, \ldots, a_n, a_0, \ldots, a_i),
\end{eqnarray*}
\end{footnotesize}
see, for example, \cite[\S2.5.10]{Lod:CH}, where
for compactness in notation, we have written $(a_0, \ldots, a_n) := a_0 \otimes \cdots \otimes a_n$ for a (normalised) Hochschild chain. 
Furthermore, let $\gvf, \psi \in \overline{C}^\bull(A,A)$ be (normalised) Hochschild cochains of degree $p$ and $q$, respectively (that is, elements in $\Hom_k(A^{\otimes \bull}, A)$ that vanish on elements $(a_1, \ldots, a_n)$ if one of the $a_i$ is equal to $1$). Then the formulae
\begin{footnotesize}
\begin{eqnarray*}
 \iota_{\varphi}(a_0, \ldots, a_n) 
&\!\!\!\!\!=&\!\!\!\!\! \big(a_0 {\varphi}(a_1, \ldots, a_p), a_{p+1}, \ldots,  a_n\big), 
\\
\cL_{\varphi}(a_0, \ldots, a_n)  
&\!\!\!\!\!=&\!\!\!\!\! 
\nonumber
\sum^{n-p+1}_{i=1} (-1)^{(p-1)(i-1)}  
\big(a_0, \ldots, a_{i-1}, \varphi(a_{i}, \ldots, a_{i+p-1}), a_{i+p}, \ldots, a_n\big) \\
&& \!\!\!\!\!\!\!\!\!\! + \!\!
\sum^{p}_{i=1} (-1)^{n(i-1)+(p-1)}  
\big(\varphi(a_{n-i+2}, \ldots, a_n, a_0, \ldots, a_{i+p-2}), a_{i+p-1}, \ldots, a_{n-i+1}\big), 
\\
\cS_{\varphi}(a_0, \ldots, a_n) 
&\!\!\!\!\!=&\!\!\!\!\! \sum^{n-p+1}_{j=1} \sum^{n-p+1}_{i=j} (-1)^{n(j-1)+(p-1)(i-1)} 
\big(1_\ahha, a_{n-j+2}, \ldots, a_n, a_0, \ldots, a_{i-1}, 
\\
&& 
\qquad\qquad\qquad\qquad\qquad\qquad\qquad\qquad
\gvf(a_{i}, \ldots, a_{i+p-1}), a_{i+p}, \ldots, a_{n-j+1}\big),
\\
\cT(\varphi, \psi)(a_0, \ldots, a_n) 
&\!\!\!\!\!=&\!\!\!\!\!   \sum^{p-1}_{j=1}  \sum^{p-1}_{i=j} (-1)^{n(j-1) + (q-1)(i-j) + p}
\big(\gvf(a_{n-j+2}, \ldots, a_n, a_0, a_1, \ldots, a_{p-i-1}, 
\\
&& 
\!\!
\psi(a_{p-i}, \ldots, a_{p+q-i-1}), a_{p+q-i} , \ldots, a_{p+q-1-j}), 
a_{p+q-j}, \ldots, a_n\big)
\end{eqnarray*}
\end{footnotesize}
define a homotopy Cartan calculus in the sense of Def.~\ref{ministeroperibenieleattivitaculturali}, when considering  $(C_\bull(A, A), b, B)$ as a mixed (cochain) complex by passing to negative degrees as in Example \ref{ancorauno}, resp.\ Eq.~\rmref{trimalchio}. These formulae explicitly appeared, with the exception of $\cT$, in \cite{NesTsy:OTCROAA}. The corresponding formulae for the case of $A_\gs$, that is, where the coefficient module $A$ as an $A$-bimodule is twisted by a {\em Nakayama automorphism} $\gs$ can be found in \cite[\S7.2]{KowKra:BVSOEAT}. In the case of Frobenius or (twisted) Calabi-Yau algebras
({\em cf.}~Example \ref{calabibabi2}), one can even define a homotopy calculus structure of Hochschild cochains over themselves: the main difficulty here consists in finding a cyclic operator on $C^\bull(A,A)$ which does not always exist but does so if $A$ is what is called an anti Yetter-Drinfel'd contramodule over $\Ae = A \otimes A^\op$, see \cite{Kow:BVASOCAPB} for a treatment. The corresponding calculus structure can then be obtained by the specific form of the cyclic operator (see {\em op.~cit.}, Eq.~(6.7), for the explicit expression depending on the {\em contraaction}) and of the operators in Theorem \ref{nochkeinname}, see below.
\end{example}

It is immediate from the definition that a homotopy Cartan calculus on $CC^\bullet(M)$ (resp.\ $CC^\bullet_-(M)$) induces a Cartan calculus (with zero homotopy) of the cohomology of $\mathfrak{g}^\bullet$ on the (negative) cyclic cohomology of $M^\bullet$. 

\begin{dfn}
Let $(\mathfrak{g}^\bullet, \cL, \cI, \cT)$ be a homotopy Cartan calculus on $CC^\bullet_{\rm per}(M)$, and let $\mathcal{G}^\bullet_\g := \g^\bullet[-1]$. The datum of a \emph{$\cL$-homotopy Gerstenhaber algebra} structure on $\mathcal{G}^\bullet_\g$ is the datum of a differential graded associative (not necessarily commutative) algebra structure $(\mathcal{G}^\bullet_\g,\delta[-1],\smallsmile)$ such that  
\[
\iota\colon \mathcal{G}_\g\to \mathrm{End}(M^\bullet)
\]
is a morphism of dga's, that is to say, it satisfies 
\begin{equation}\label{morphism-of-dgas}
\iota_{f\smallsmile g}=\iota_f\,\iota_g,
\end{equation}
in addition to
$
[\iota_f,\bc]+\iota_{\gd f}=0
$
from \rmref{dellera1}.
We refer to the product $\smallsmile: \mathcal{G}^i_\g \otimes \mathcal{G}^j_\g \to \mathcal{G}^{i+j}_\g$ as the {\em cup product}. 
When this additional datum is given, we say that $(\mathfrak{g}^\bullet, \iota,\cL, \cS, \cT)$ defines a {\em homotopy Cartan-Gerstenhaber calculus} on $CC^\bullet_{\rm per}(M)$.
\end{dfn}

\begin{rem}
By {\em dgca} we will always mean a unital differential graded commutative algebra, and morphisms between these will be required to be unit preserving. In particular, $\iota$ will satisfy
\[
\iota_{1_{\mathcal{G}_\g}}=\mathrm{Id}_{M^\bullet},
\]
where $1_{\mathcal{G}_\g}$ is the unit element of ${\mathcal{G}_\g}$.
\end{rem}


%

\begin{rem}
In the literature,  $\mathcal{G}^\bullet_\g$ is usually required to be a Gerstenhaber algebra (at least up to homotopy), therefore it might come as a surprise that we do not ask for the compatibility between the cup product $\smallsmile$ and the Lie bracket $\{\cdot,\cdot\}$. The reason is that the morphism   $\cL\colon \mathfrak{g}^\bullet\to \mathrm{End}(CC_{\mathrm{per}}^\bullet(M)[n])$
actually does see $\mathcal{G}^\bullet_\g$ as a Gerstenhaber algebra up to homotopy. We will detail this in Corollaries \ref{I-will-need-you-as-well1} \& \ref{I-will-need-you-as-well2}.
\end{rem}

In the following proposition and the subsequent three corollaries, we are going to show that $\cL$ and $\cI$ satisfy the customary formulae for Cartan calculi, up to homotopy. We will correspondingly write $ a \simeq b$ to indicate that $a$ is homotopic to $b$.

\begin{prop}
For any homogeneous $f,g$ in  $\g^\bullet$, we have
\begin{equation*}
\cL_{f \smallsmile g} \simeq \cL_f \cI_g -(-1)^f \cI_f \cL_g.
\end{equation*}
More precisely,
let $\Phi_{f,g}:=\mathcal{S}_{f \smallsmile g}-\cS_f\cI_g-\cI_f\cS_g+u\cS_f\cS_g$. 
Then, 
\begin{equation}
\label{nochsowas2}
\cL_{f \smallsmile g} =\cL_f \cI_g -(-1)^f \cI_f \cL_g +[d_u, \Phi_{f,g}]+\Phi_{\gd f,g}-(-1)^f\Phi_{f,\gd g}.
\end{equation}
In particular, if $f$ and $g$ are cocycles, $\cL_{f \smallsmile g}$ acts as $\cL_f \cI_g -(-1)^f \cI_f \cL_g$ on $HC^\bullet_{\mathrm{per}}(M)$, on $HC^\bullet(M)$, and on $HC^\bullet_-(M)$.
\end{prop}
\begin{proof}
We work on $CC^\bullet_{\mathrm{per}}(M)$ first.
From $\iota_{f\smallsmile g}=\iota_f\,\iota_g$, we get
\begin{align*}
\cI_{f\smallsmile g}&=\iota_{f \smallsmile g}+u\cS_{f \smallsmile g}\\
&=\iota_f\, \iota_g+u\cS_{f \smallsmile g}\\
&=(\cI_f-u\cS_f)(\cI_g-u\cS_g)+u\cS_{f \smallsmile g}\\
&=\cI_f\, \cI_g+u(\cS_{f \smallsmile g}-\cS_f\cI_g-\cI_f\cS_g+u\cS_f\cS_g)\\
&=\cI_f\,\cI_g+u\Phi_{f,g}.
\end{align*}
From $\cI_{f\smallsmile g}=\cI_f\,\cI_g+u\Phi_{f,g}$, on the other hand, we get
\begin{align*}
u\cL_{f \smallsmile g}&=[ d_u,\cI_{f\smallsmile g}]  + \cI_{\gd (f\smallsmile g)}\\
&=[ d_u,\cI_{f\smallsmile g}]  + \cI_{\gd f\smallsmile g}-(-1)^f\cI_{f\smallsmile \gd  g}\\
&=[ d_u, \cI_f\, \cI_g] +u[ d_u,\Phi_{f,g}] + \cI_{\gd f}\,\cI_g-(-1)^f \cI_f\,\cI_{\delta g}+u\Phi_{\gd f,g}-(-1)^fu\Phi_{f,\gd g}\\
&=d_u\cI_f\, \cI_g -(-1)^{f+g}  \cI_f\, \cI_g d_u+ \cI_{\gd f}\,\cI_g -(-1)^f \cI_f\,\cI_{\delta g}+u([d_u,\Phi_{f,g}]
\\
&
\qqquad\qqquad
+\Phi_{\gd f,g}-(-1)^f\Phi_{f,\gd g})\\
&=d_u \cI_f\, \cI_g  +(-1)^f \cI_f\, d_u\cI_g  -(-1)^{f}(\cI_f\, d_u\cI_g+  (-1)^{g}\cI_f\, \cI_gd_u) + \cI_{\gd f}\,\cI_g 
\\
&
\qqquad\qqquad
-(-1)^f \cI_f\,\cI_{\delta g}+u([d_u, \Phi_{f,g}]+\Phi_{\gd f,g}-(-1)^f\Phi_{f,\gd g})\\
&=(d_u \cI_f\,   +(-1)^f \cI_f\, d_u+\cI_{\gd f})\cI_g -(-1)^{f}\cI_f(d_u\cI_g+  (-1)^{g} \cI_gd_u+\cI_{\delta g})
\\
&
\qqquad\qqquad
+u([d_u,\Phi_{f,g}]+\Phi_{\gd f,g}-(-1)^f\Phi_{f,\gd g})\\
&=u(\cL_f \cI_g -(-1)^f \cI_f \cL_g +[d_u,\Phi_{f,g}]+\Phi_{\gd f,g}-(-1)^f\Phi_{f,\gd g}).
\end{align*}
As multiplication by $u$ is an invertible operator on $CC^\bullet_{\mathrm{per}}(M)$, this concludes the proof of identity \rmref{nochsowas2} on the periodic cocyclic complex.
As $\cI$, $\cL$ and $\Phi$ preserve the subcomplexes $(M[[u]],d_u)$ and $(uM[[u]],d_u)$ of $CC_{\mathrm{per}}^\bullet(M)$, identity \rmref{nochsowas2} is also an identity on the subcomplex 
$CC^\bullet(M)$ and on the quotient $CC_{-}^\bullet(M)$ of $CC_{\mathrm{per}}^\bullet(M)$.
The second part of the statement is immediate. 
\end{proof}

\begin{corollary}
\label{I-will-need-you}
For any homogeneous $f, g$ in $\g^\bullet$,
one has
\begin{align*}
\cL_{f \smallsmile g} \simeq \cL_f \iota_g -(-1)^f \iota_f \cL_g. 
\end{align*}
More precisely,
\begin{align*}
\cL_{f \smallsmile g} &=\cL_f \iota_g -(-1)^f \iota_f \cL_g +[\bc,S_{f \smallsmile g}-\cS_f\iota_g-\iota_f\cS_g ]+S_{\gd f \smallsmile g}-\cS_{\gd f}\iota_g-\iota_{\gd f}\cS_g
\\
&
\qqquad\qqquad
-(-1)^f(S_{f \smallsmile \gd g}-\cS_f\iota_{\gd g}-\iota_f\cS_{\gd g}).
\end{align*}
\end{corollary}
\begin{proof}
Look at the coefficient of $u^0$ in the identity $\cL_{f \smallsmile g} =\cL_f \cI_g -(-1)^f \cI_f \cL_g +[\Phi_{f,g}, d_u]+\Phi_{\gd f,g}-(-1)^f\Phi_{f,\gd g}$.
\end{proof}

\begin{corollary}
\label{I-will-need-you-as-well1}
For any homogeneous $f,g$ in $\mathcal{G}_\g^\bullet$, one has 
\begin{align*}
\cL_{f \smallsmile g}-(-1)^{fg}\cL_{g \smallsmile f} \simeq 0.
\end{align*}
More precisely, 
\begin{align*}
\cL_{f \smallsmile g}-(-1)^{fg}\cL_{g \smallsmile f}&=[d_u, (-1)^{f}\cT(f,g)-(-1)^{(f-1)g}\cT(g,f)+\Phi_{f,g}-(-1)^{fg}\Phi_{g,f}] \\
& \quad - (-1)^{f}\cT(\gd f, g) +\cT(f, \delta g) \\
& \quad
  +(-1)^{(f-1)g} \cT(\gd g, f) -(-1)^{fg} \cT(g, \delta f)\\
  & \quad  +\Phi_{\gd f,g}+(-1)^{f}\Phi_{f,\gd g} -(-1)^{fg}(\Phi_{\gd g,f}+(-1)^{g}\Phi_{g,\gd f}).
\end{align*}
In particular, we have $\cL_{f\smallsmile g} =(-1)^{fg}\cL_{g\smallsmile f}$ on cohomology.
\end{corollary}
\begin{proof}
Straightforward computation using \rmref{nochsowas2} and \rmref{nunja}, and recalling that the degree of $f$ and $g$ in $\mathcal{G}_\g^\bullet$ is 
 their degree in $\g^\bullet$ minus 1.
\end{proof}

\begin{corollary}
\label{I-will-need-you-as-well2}
For any homogeneous $f,g$ in $\mathcal{G}_\g^\bullet$, one has
\begin{equation*}
\begin{split}
\cL_{\{f, g \smallsmile h\}} \simeq \cL_{\{f, g\} \smallsmile h} +(-1)^{(f-1)g} \cL_{g \smallsmile \{f, h\}}.
\end{split}
\end{equation*}
More precisely, 
\begin{equation*}
\begin{split}
\cL_{\{f, g \smallsmile h\}}=& \cL_{\{f, g\} \smallsmile h} +(-1)^{(f-1)g} \cL_{g \smallsmile \{f, h\}}
\\
&
 +[d_u, -  \Phi_{\{f,g\}, h} -(-1)^{(f-1)g} \Phi_{g,\{f,h\}} +(-1)^{f(g+h)} \cL_g \cT(h,f) \\
 &\qquad\qquad\qquad\qquad -(-1)^{(f-1)(g-1)} \cT(g,f) \cL_h - (-1)^f[\cL_f, \Phi_{g,h}]] 
\\
&
-(-1)^{(f-1)(g+h-1)} \cL_g\big( \cT(\gd h, f) -(-1)^h \cT(h, \gd f) \big) \\
&
+(-1)^{(f-1)(g-1)} \big( \cT(\gd g, f) -(-1)^g \cT(g, \gd f) \big)\cL_h 
\\
&
+(-1)^f [\cL_{\gd f}, \Phi_{g,h}] -(-1)^{f(g+h)} \cL_{\gd g} \cT(h,f) -(-1)^{fg} \cT(g,f) \cL_{\gd h}\\
&
+[\cL_f,\Phi_{\gd g,h}]+(-1)^g[\cL_f,\Phi_{ g,\gd h}] 
\\
&
-\Phi_{\{\gd f,g\},h}+(-1)^f\Phi_{\{f,\gd g\},h}+(-1)^{f+g}\Phi_{\{f,g\},\gd h}\\
&
-(-1)^{(f-1)g}\Phi_{\gd g,\{f,h,\}}-(-1)^{fg}\Phi_{g,\{\gd f,h\}}+(-1)^{f(g-1)}\Phi_{g,\{f,\gd h\}}.
\end{split}
\end{equation*}
Hence, we have $\cL_{\{f, g \smallsmile h\}} = \cL_{\{f, g\} \smallsmile h} +(-1)^{(f-1)g} \cL_{g \smallsmile \{f, h\}} $ on cohomology.
\end{corollary}
\begin{proof}
Again, a straightforward computation using \rmref{nochsowas2} and \rmref{nunja} along with \rmref{is-a-morphism-of-lie-algebras}, and recalling that the degree of $f$ and $g$ in $\mathcal{G}_\g^\bullet$ is 
 their degree in $\g^\bullet$ minus 1.
\end{proof}

\begin{corollary}\label{gerst-quotient}
Looking at $\cL$ as a morphism $\cL\colon H^\bullet(\mathfrak{g})\to \mathrm{End}(HC^\bullet(M))$ of graded Lie algebras, one has that $H^\bullet(\mathcal{G}_\g)/\ker(\cL[-1])$ is a Gerstenhaber algebra. 
The same is true if we look at $\cL$ as a morphism $\cL\colon H^\bullet(\mathfrak{g})\to \mathrm{End}(HC^\bullet_-(M))$.
\end{corollary}
Clearly, the statement of Corollary \ref{gerst-quotient} is true also if we look at $\cL$ as a morphism $\cL\colon H^\bullet(\mathfrak{g})\to \mathrm{End}(HC^\bullet_\mathrm{per}(M))$. In this case, however, the statement is trivial since, by Remark \ref{acts-trivially}, the morphism  $\cL\colon H^\bullet(\mathfrak{g})\to \mathrm{End}(HC^\bullet_\mathrm{per}(M))$ is the zero map.

\section{Induced Lie brackets on cyclic cohomology}
\label{fame}

In this section, we develop the abstract construction underlying the higher bracket on negative cyclic homology of Calabi-Yau algebras in \cite{dTdVVdB:CYDANCH} and show how this naturally arises from a general homotopy calculus in presence of a certain duality isomorphism.

Let $(\mathfrak{g}^\bullet, \cL, \cI, \cT)$ for the rest of this section be a homotopy Cartan-Gerstenhaber calculus on $CC^\bullet_{\rm per}(M)$. By restriction, this induces a homotopy Cartan-Gerstenhaber calculus on $CC^\bullet(M)$.
By Lemma \ref{lemma:is-a-morphism-of-lie-algebras}, the Lie derivative $\cL$ defines a dg-Lie algebra module structure on $CC^\bullet(M)$, which allows us to form a semidirect product with a canonical differential and Lie bracket on it. 
More precisely:

\begin{definition}
The dg-Lie algebra $\mathfrak{g}^\bullet \ltimes CC^\bullet(M)[-2]$ is the cochain complex $\mathfrak{g}^\bullet \oplus CC^\bullet(M)[-2]$ endowed with the Lie bracket
\begin{equation}
\label{fernandosorvariazionisuuntemadiMozart}
\big[(f,x), (g,y)\big] := \big(\{f,g\}, \cL_f y -(-1)^{g\,x} \cL_g x\big).
\end{equation}
\end{definition}

\begin{rem}\label{this-time-it-is-Lie}
The dg-Lie algebra $\mathfrak{g}^\bullet \ltimes CC^\bullet(M)[-2]$ fits into a short exact sequence of complexes
\[
0\to CC^\bullet(M)[-2]\to \mathfrak{g}^\bullet \ltimes CC^\bullet(M)[-2]\to \mathfrak{g}\to 0,
\] 
where the rightmost arrow is the projection on the first factor and the leftmost arrow is the inclusion of the second factor. As the projection  $\mathfrak{g}^\bullet \ltimes CC^\bullet(M)[-2]\to \mathfrak{g}$ is a morphism of dg-Lie algebras, $CC^\bullet(M)[-2]$ is a dg-Lie ideal (and so in particular a dg-Lie subalgebra) of $\mathfrak{g}^\bullet \ltimes CC^\bullet(M)[-2]$. Notice that the Lie bracket on $CC^\bullet(M)[-2]$ is the trivial one.
\end{rem}

As for any dg-Lie algebra, $\mathfrak{g}^\bullet \ltimes CC^\bullet(M)[-2]$ can be twisted by adding to its differential the degree $1$ linear operator $[(h,\xi),-]$, where $(h,\xi)$ is a Maurer-Cartan element. 
%
%
As $CC^\bullet(M)[-2]\hookrightarrow \mathfrak{g}^\bullet \ltimes CC^\bullet(M)[-2]$ is the inclusion of a dg-Lie subalgebra, where on $CC^\bullet(M)[-2]$ we have the trivial bracket, if $\xi$ is a Maurer-Cartan element in $CC^\bullet(M)[-2]$, that is, if $\xi$ is a cocycle in $CC^{-1}(M)$, then $(0,\xi)$ is a Maurer-Cartan element for $\mathfrak{g}^\bullet \ltimes CC^\bullet(M)[-2]$. The commutator $[(0,\xi),-]$ acts on $\mathfrak{g}^\bullet \ltimes CC^\bullet(M)[-2]$ as
\[
[(0,\xi),-]\colon (f,x)\mapsto (0,-(-1)^f \cL_f\xi).
\]
Therefore, by writing $\hat{\zeta}=-\xi$
we see that the $\hat{\zeta}$-twisted differential on $\mathfrak{g}^\bullet \ltimes CC^\bullet(M)[-2]$ is
\[
\partial_{\hat{\zeta}}\colon (f,x)\mapsto (\delta f, d_ux +(-1)^f  \cL_f\hat{\zeta}).
\]
Notice that we still have a short exact sequence of complexes
\[
0\to CC^\bullet(M)[-2]\to (\mathfrak{g}^\bullet \ltimes CC^\bullet(M)[-2],\partial_{\hat{\zeta}})\to \mathfrak{g}\to 0.
\]

\begin{lemma}
Let $\hat{\zeta}$ be a cocycle in $CC^{-1}(M)$. Then
the map 
\begin{align*}
\Psi_{\hat{\zeta}}: (\mathfrak{g}^\bullet  \ltimes CC^\bullet(M)[-2],\partial_{\hat{\zeta}}) &\to CC^\bullet(M),
\\
(f,x) &\mapsto (-1)^f \cI_f\hat{\zeta} + u x,
\end{align*}
is a morphism of complexes.
\end{lemma}

\begin{proof}
We have
\begin{equation*}
\begin{split}
d_u\Psi_{\hat{\zeta}}(f,x) &= (-1)^f d_u\cI_f(\hat{\zeta}) + ud_ux \\&
=  (-1)^f [d_u,\cI_f](\hat{\zeta}) + ud_ux
\\
&= 
 -(-1)^f \cI_{\gd f} \hat{\zeta}  +(-1)^f u\cL_f \hat{\zeta} +  ud_ux \\
&
=  (-1)^{\gd f}  \cI_{\gd f} \hat{\zeta} + u(d_ux +(-1)^f  \cL_f \hat{\zeta})
\\
&=\Psi_{\hat{\zeta}}(\gd f,d_ux +(-1)^f  \cL_f \hat{\zeta})
\\
&= \Psi_{\hat{\zeta}}(\partial_{\hat{\zeta}}(f,x)), 
\end{split}
\end{equation*}
where we used $d_u\hat{\zeta}=0$ in the second step and \rmref{schonlos2} in the third.
\end{proof}

\begin{lemma}
\label{lemma.morphism-of-exact-sequences}
Let $\hat{\zeta}$ be a cocycle in $CC^{-1}(M)$. Then $\Psi_{\hat{\zeta}}$ fits into a diagram of short exact sequences of cochain complexes
\begin{equation*}
\label{alltrianglesatonce}
	\xymatrix{
0 \ar[r] & CC^\bullet(M)[-2] \ar[r]\ar[d]_{\mathrm{id}} & (\mathfrak{g}^\bullet  \ltimes CC^\bullet(M)[-2],\partial_{\hat{\zeta}}) \ar[r]\ar[d]^{\Psi_{\hat{\zeta}}}  & \mathfrak{g}^\bullet
 \ar[r]\ar[d]^{\iota_{(\cdot)}\hat{\zeta}_0} & 0 
\\
0  \ar[r] & CC^\bullet(M)[-2] \ar[r]_{u}  & CC^\bullet(M) \ar[r]_{\mathrm{ev}_0} & M^\bullet  \ar[r] & 0,
}
\end{equation*}
where $\hat{\zeta}_0=\mathrm{ev}_0\hat{\zeta}$ and $\iota_{(\cdot)}\hat{\zeta}_0\colon \mathfrak{g}^\bullet \to M^\bullet$ is the morphism $f\mapsto (-1)^f\iota_f\hat{\zeta}_0$.
\end{lemma}

\begin{proof}
The proof is a straightforward direct check. To see that $\iota_{(\cdot)}\hat{\zeta}_0$ is indeed a morphism of complexes, one can notice that  $\iota_{(\cdot)}\hat{\zeta}_0$ is the morphism between $\mathfrak{g}^\bullet$ and $M^\bullet$ induced by $\Psi_{\hat{\zeta}}$, by the commutativity of the left square and by the universal property of the quotient. Otherwise, one can also directly compute
\begin{align*}
(-1)^{\delta f}\iota_{\delta f}\hat{\zeta}_0&=(-1)^f[\bc,\iota_f]\hat{\zeta}_0=\bc((-1)^f\iota_f\hat{\zeta}_0)
\end{align*}
since $\bc\hat{\zeta}_0=0$.
\end{proof}

\begin{definition}\label{definition-of-p}
An element $\zeta$ in $M^{-1}$ is called a \emph{Palladio}\footnote{{\sc Andrea Palladio}, 30 November 1508 -- 19 August 1580; the name chosen for the special cocycles wants to be a homage to the Paduan architect as well as to allude to the chemical element palladium whose symbol {\em Pd} is a reminder of {\em Poincar\'e duality}.}
 cocycle if 
\begin{enumerate}
\item $\zeta$ is a cocycle, {\em i.e.}, $\bc\zeta=0$;
\item $\Bc\zeta$ equals zero in cohomology, {\em i.e.}, there exists an element $\eta\in M^{-3}$ such that $\Bc\zeta=\bc\eta$;
\item $\iota_{(\cdot)}\zeta\colon \mathfrak{g}^\bull\to M^\bullet$ is a quasi-isomorphism of complexes.
\end{enumerate}
We denote by $p_\zeta\colon H^\bullet(\mathfrak{g})\xrightarrow{\sim}H^\bullet(M)$ the corresponding isomorphism in cohomology. 
\end{definition}
%
%
%

\begin{rem}
As we will see in Remark \ref{landweber} and Exs.~\ref{calabibabi1}--\ref{calabibabi2}, classical examples of Palladio cocycles are provided by the {\em fundamental classes} that appear in Poincar\'e duality in its various flavours.
\end{rem}

\begin{rem}
\label{bocconotti}
If $H^{-2}(M)=0$, then condition \emph{ii)} in Definition \ref{definition-of-p} is trivially satisfied. Therefore, if $\g$ is such that $H^{-2}(\g)=0$, we see that a Palladio cocycle is precisely a cocycle $\zeta$ in $M^{-1}$ such that  $\iota_{(\cdot)}\zeta\colon H^\bullet(\mathfrak{g}) \to H^\bullet(M)$ is an isomorphism.
\end{rem}

\begin{lemma}\label{cohomologous-to-palladio}
Let $\zeta$ be a Palladio cocycle. Then there exists a cocycle $\hat{\zeta}$ in $CC^{-1}(M)$ such that $\hat{\zeta}_0$ is in the same cohomology class of $\zeta$, where $\hat{\zeta}_0=\mathrm{ev}_0\hat{\zeta}$.
\end{lemma}
\begin{proof}
Since $\Bc\zeta$ is zero in cohomology, by Lemma \ref{connecting-one}, the cohomology class of $\zeta$ is in the kernel of the connecting homomorphism $\beta$ and so it is in the image of  the projection $\pi$. In other words, there exists a cocycle $\hat{\zeta}$ in $CC^{-1}(M)$ such that $\mathrm{ev}_0\hat{\zeta}$ is in the same cohomology class of $\zeta$. 
\end{proof}

\begin{rem}
\label{for-later-later-use}
When $\zeta$ is a Palladio cocycle
we have the following identity of cohomology classes in $HC^{\bullet}(M)$, for any cocycle $f$ in $\mathfrak{g}^\bullet$:
\begin{equation*}
{\cL_f\hat{\zeta}}=(-1)^f\beta p_\zeta(f),
\end{equation*}
where $\beta$ is the connecting homomorphism from the long exact sequence in \rmref{taylor-long} and $\hat{\zeta}$ is a cocycle  in $CC^{-1}(M)$ such that $\hat{\zeta}_0$ is the same cohomology class of $\zeta$ (see Lemma \ref{cohomologous-to-palladio}).
Namely, as $\zeta$ and $\hat{\zeta}_0$ are in the same cohomology class, we have $p_\zeta(f)=(-1)^f\iota_{f}\hat{\zeta}_0$ in cohomology.
By definition of the connecting homomorphism, $\beta p_\zeta(f)$ is the $d_u$-cohomology class of an element $g$ in $CC^{\deg(f)-1}(M)$ such that $ug=d_uh$, with $\mathrm{ev}_0(h)=(-1)^f\iota_f\hat{\zeta}_0$. An obvious choice for $h$ is $h=(-1)^f\cI_f\hat{\zeta}$ so that we are looking for an element $g$ such that $ug=(-1)^fd_u\cI_f\hat{\zeta}$. As $\hat{\zeta}$ and $f$ are cocycles, we have $d_u\cI_f\hat{\zeta}=u\cL_f\hat{\zeta}$ and hence we can take $g=(-1)^f\cL_f\hat{\zeta}$.
\end{rem}

\begin{prop}
Let $\zeta$ in $M^{-1}$ be a Palladio cocycle. Then $\Psi_{\hat{\zeta}}$ is a quasi-isomorphism, where $\hat{\zeta}$ is a cocycle  in $CC^{-1}(M)$ such that $\hat{\zeta}_0$ is the same cohomology class of $\zeta$.
\end{prop}

\begin{proof}
As $\zeta$ and $\hat{\zeta}_0$ are in the same cohomology class, $\iota_{(\cdot)}\hat{\zeta}_0=\iota_{(\cdot)}\zeta$ in cohomology. Since $\zeta$ is a Palladio cocycle, $\iota_{(\cdot)}\zeta$ is a quasi-isomorphism, and so the external vertical arrows in the diagram in Lemma \ref{lemma.morphism-of-exact-sequences} are both quasi-isomorphisms. Then also the central arrow $\Psi_{\hat{\zeta}}$ is a quasi-isomorphism by the naturality of the long exact sequence in cohomology and by the Five Lemma.
\end{proof}

As $(\mathfrak{g}^\bullet  \ltimes CC^\bullet(M)[-2],\partial_\zeta)$ is a dg-Lie algebra, its cohomology $H^\bullet(\mathfrak{g}^\bullet  \ltimes CC^\bullet(M)[-2],\partial_{\hat{\zeta}})$ is a graded Lie algebra. Since $\Psi_{\hat{\zeta}}$ is a quasi-isomorphism when $\zeta$ is a Palladio cocycle, we see that in this case we have an induced graded Lie bracket on $HC^\bullet(M)$ defined by
\[
[z,w] :={\Psi}_{\hat{\zeta}}([{\Psi}^{-1}_{\hat{\zeta}} {z},{\Psi}^{-1}_{\hat{\zeta}} {w}]),
\]
where on the right-hand side we have the (induced) Lie bracket \rmref{fernandosorvariazionisuuntemadiMozart} on $H^\bullet(\mathfrak{g}^\bullet  \ltimes CC^\bullet(M)[-2],\partial_{\hat{\zeta}})$, and
where to avoid cumbersome notation we have written and will also write in all what follows  
$x$
for both a $d_u$-cocycle $x$ and its cohomology class. Likewise, by abuse of notation, $\Psi_{\hat{\zeta}}$ also denotes 
the induced isomorphism
\[
H^\bullet(\Psi_{\hat{\zeta}})\colon H^\bullet(\mathfrak{g}^\bullet  \ltimes CC^\bullet(M)[-2],\partial_{\hat{\zeta}})\xrightarrow{\sim} HC^\bullet(M)
\]
on cohomology.

%
We can use the linear isomorphism $p_\zeta$ to transfer the $\smallsmile$ and $\{\cdot,\cdot\}$ operations from $H^\bullet(\g)$ to $H^\bullet(M)$. Namely, we write
\[
\smallsmile\colon H^\bullet(M)\otimes H^\bullet(M)\to H^\bullet(M)[-1]
\]
for the degree $-1$ product on $H^\bullet(M)$ defined by 
\begin{equation}\label{cup-on-M}
{a}\smallsmile{b}=p_\zeta(p_\zeta^{-1}({a})\smallsmile p_\zeta^{-1}({b}))
\end{equation}
and 
\[
\{\cdot,\cdot\}\colon H^\bullet(M)\otimes H^\bullet(M)\to H^\bullet(M)
\]
for  the bracket on $H^\bullet(M)$ defined by
\begin{equation}\label{bracket-on-M}
\{{a},{b}\} := p_\zeta(\{p_\zeta^{-1}{a}, p_\zeta^{-1}{b}\}),
\end{equation}
where the bracket on the right hand side is that of $\g$ and where $a, b$ by slight abuse of notation denote classes in $H^\bullet(M)$. 
Writing $\mathcal{G}^\bullet_M$ for the shifted complex $M^\bullet[-1]$, we see that when $H^\bullet(\cG_\g)$ is a Gerstenhaber algebra, then also $H^\bullet(\cG_M)$ is a  Gerstenhaber algebra (isomorphic to $H^\bullet(\cG_\g)$), with unit element the cohomology class of $\zeta$.

\begin{theorem}
\label{fapurefreddoquadentro}
The Lie bracket on $HC^\bullet(M)$ induced by a homotopy Cartan-Gerstenhaber calculus and a Palladio cocycle $\zeta$ has the form
\begin{equation}
\label{lastriot}
[{z},{w}]=(-1)^{z-1}\beta((\pi {z})\smallsmile (\pi {w})),
\end{equation}
where a $d_u$-cocycle $z$ and its cohomology class are denoted by the same letter and where $\pi: HC^\bullet(M) \to H^\bullet(M)$ is the map induced by the evaluation $\mathrm{ev}_0$ 
on cohomology. 
\end{theorem}

\begin{proof}
Let $\hat{\zeta}$ be a cocycle  in $CC^{-1}(M)$ such that $\hat{\zeta}_0$ is the same cohomology class of $\zeta$. Since ${\Psi}_{\hat{\zeta}}$ is an isomorphism on cohomology, we can write $z={\Psi}_{\hat{\zeta}}(f,x)$ and $w={\Psi}_{\hat{\zeta}}(g,y)$, and the statement is equivalent to
\[
{\Psi}_{\hat{\zeta}}([{(f,x)},{(g,y)}])=(-1)^{f-1}\beta p_\zeta(p_\zeta^{-1}(\pi{\Psi}_{\hat{\zeta}}(f,x))\smallsmile p_\zeta^{-1}(\pi{\Psi}_{\hat{\zeta}}(g,y))),
\]
where again a $\partial_{\hat{\zeta}}$-cocycle $(f,x)$ and its cohomology class are denoted by the same symbols. 
By Lemma \ref{lemma.morphism-of-exact-sequences} and by the definition of $p_\zeta$ in Definition \ref{definition-of-p}, we have on cohomology
\[
{\pi{\Psi}_{\hat{\zeta}}(f,x))}=(-1)^f{\iota_f\hat{\zeta}_0}=(-1)^f{\iota_f\zeta}=p_\zeta({f}).
\]
The statement we want to prove therefore reduces to
\[
{\Psi}_{\hat{\zeta}}([{(f,x)},{(g,y)}])=(-1)^{f-1}\beta p_\zeta({f}\smallsmile{g}),
\]
and so, thanks to Remark \ref{for-later-later-use}, to
\[
{\Psi}_{\hat{\zeta}}([{(f,x)},{(g,y)}])=(-1)^{g}{\cL_{f\smallsmile g}{\hat{\zeta}}}.
\]
As $(f,x)$ is a $\partial_{\hat{\zeta}}$-cocycle, we have the identities 
\begin{equation}
\begin{cases}
\delta f=0,
\\
-(-1)^f\cL_f{\hat{\zeta}}=d_u x,
\label{airwaves2}
\end{cases}
\end{equation}
and similarly for $(g,y)$.
We compute, in cohomology,
\begin{eqnarray*}
&& {\Psi}_{\hat{\zeta}}  ([{(f,x)},{(g,y)}])
\\
& = 
&
{{\Psi}_{\hat{\zeta}}\big(\{f,g\}, \cL_f y -(-1)^{g\, x} \cL_g x\big)}
\\
&= 
&
(-1)^{f+g}{\cI_{\{f,g\}}{\hat{\zeta}} + u( \cL_f y -(-1)^{g\, x} \cL_g x)}
\\
&
\overset{\scriptscriptstyle{\rmref{nunja}}}{=}
&
 (-1)^{f+g} {[\cI_f, \cL_g]{\hat{\zeta}}  + u( \cL_f y -(-1)^{g\, x} \cL_g x)}
\\
&
=
&
 (-1)^{f+g} \cI_f\cL_g{\hat{\zeta}} -(-1)^{gf+f}\cL_g\cI_f{\hat{\zeta}} + u( \cL_f y -(-1)^{g\, x} \cL_g x)
\\
&
\overset{\scriptscriptstyle{\rmref{nochsowas2}}}{=}
& 
  -(-1)^{gf+f}\cL_{g \smallsmile f}  {\hat{\zeta}} +(-1)^{f+g} \cI_f\cL_g{\hat{\zeta}} 
\\
&&
-(-1)^{gf+f+g}\cI_g\cL_f {\hat{\zeta}}
+ u\cL_f y -(-1)^{gf} u\cL_g x
\\
&
\overset{\scriptscriptstyle{\rmref{airwaves2}}}{=}
&
 (-1)^{g} {\cL_{f \smallsmile g}  {\hat{\zeta}} -(-1)^f \cI_f d_uy +(-1)^{gf+g} \cI_gd_ux  +u\cL_f y -(-1)^{g f} u\cL_g x}
\\
&
\overset{\scriptscriptstyle{\rmref{schonlos2}}}{=}
&
   (-1)^{g}{\cL_{f \smallsmile g}  {\hat{\zeta}} + d_u\cI_f y -(-1)^{g f} d_u\cI_g x}
\\
&
=
&
 (-1)^{g}{\cL_{f \smallsmile g} {\hat{\zeta}}},
\end{eqnarray*}
where we have used Corollary \ref{I-will-need-you-as-well1} and the fact that the degree of $f$ and $g$ as elements of $\g^\bullet$ is their degree as elements of $\mathcal{G}_\g^\bullet$ plus $1$.
\end{proof}

\begin{rem}
One could do the same in periodic cohomology, but one would get the zero bracket as $\pi$ is zero in the long exact sequence for periodic cohomology. This agrees with the fact that $\cL_f$ is zero in periodic cohomology.
\end{rem}

\begin{rem}
\label{mfgoverload}
Starting from a mixed (chain) complex $(N_\bullet, b, B)$, in view of Definition \ref{brot} and Remark \ref{lanciano}, the Lie bracket \rmref{lastriot} from Theorem \ref{fapurefreddoquadentro} passes into a Lie bracket on the {\em negative cyclic homology} $HC^-_\bullet(N)$ of the mixed complex $(N_\bullet, b, B)$.
In the context of associative algebras with Poincar\'e duality or rather $d$-Calabi-Yau algebras, this bracket coming from \rmref{lastriot} as well as the preceding way to obtain it has been dealt with in \cite{dTdVVdB:CYDANCH}, as will be shortly discussed in Example \ref{calabibabi1}.
\end{rem}

\section{Induced Lie brackets on negative cyclic cohomology}
\label{sete}

In this section, we show how Batalin-Vilkoviski\u\i\ algebras and the so-called {\em string topology bracket} from \cite{ChaSul:ST, Men:BVAACCOHA} are correlated in general with the notion of homotopy calculus in case a Palladio cocycle is given.

\subsection{Batalin-Vilkoviski\u\i\ algebras arising from calculi}

We will work in the same set up as in the preceding \S\ref{fame}, but now focusing on the quotient  $CC^\bullet_-(M)$ of $CC^\bullet_{\rm per}(M)$ rather than on the subcomplex $CC^\bullet(M)$. Namely, for the entire section $(\mathfrak{g}^\bullet, \cL, \cI, \cT)$ will be a homotopy Cartan-Gerstenhaber calculus on $CC^\bullet_{\rm per}(M)$. Passing to the quotient, this induces a homotopy Cartan-Gerstenhaber calculus on $CC^\bullet_-(M)$. Also, using the same notation and terminology as in the preceding section, $\zeta$ will be a Palladio cocycle, {\em i.e.}, an element in $M^{-1}$ such that $\bc \zeta =0$, $\Bc\zeta = \bc\eta$ for some $\eta\in M^{-3}$,  and such that $\iota_{(\cdot)}\zeta\colon \mathfrak{g}^\bullet\to M^\bullet$ is a quasi-isomorphism of complexes. The corresponding isomorphism in cohomology will be denoted again as $p_\zeta\colon H^\bullet(\mathfrak{g})\xrightarrow{\sim}H^\bullet(M)$. Finally, recall the notations $\mathcal{G}^\bullet_\g = \g^\bullet[-1]$ and $\mathcal{G}^\bullet_M = M^\bullet[-1]$.

\begin{rem}\label{filarete-rulez}
As $\zeta$ is a Palladio cocycle, we have the following identity on cohomology classes in $H^{\bullet}(M)$, for any homogeneous $\delta$-cocycle $f$ in $\mathfrak{g}^\bullet$:
\begin{equation}
\label{here-is-p2}
(-1)^f\cL_f\zeta=\Bc  p_\zeta(f)=\beta j p_\zeta(f),
\end{equation}
where $\beta\colon HC^\bullet_-(M)\to H^\bullet(M)[-1]$, is the connecting homomorphism from the long exact sequence \rmref{laurent-long}, whereas $j$ is the morphism $H^\bullet(M)\to HC^\bullet_-(M)$ induced by the inclusion of $M^\bullet$ in $CC^\bullet_-(M)$. Namely, by definition of $p_\zeta$, we have $\Bc  p_\zeta(f)=(-1)^f\Bc  \iota_f\zeta$. As $\zeta$ is a Palladio cocycle we have $\bc\zeta=0$ and $\Bc\zeta = \bc\eta$, and
use $\cL_f= [\Bc, \iota_f]+[\bc, \cS_f] + \cS_{\gd f}$ to get
$\Bc\iota_f\zeta=\cL_f\zeta- \bc\cS_f\zeta+(-1)^f\iota_f\Bc\zeta$ so that in cohomology we have $\Bc  p_\zeta(f)=(-1)^f\cL_f\zeta$.
%
%
%
%
\end{rem}

\begin{lemma}
\label{where-streets-have-no-name}
If $f$ is a $\delta$-cocycle and $a$ is a $\bc$-cocycle, we have the following identity in cohomology:
\[
p_\zeta({f})\smallsmile {a}= -(-1)^f{\iota_f a}.
\]
\end{lemma}
\begin{proof}
As $f$ is a $\delta$-cocycle, we have $[\bc,\iota_f]=0$ and so $\iota_f\colon M^\bullet\to M^\bullet[\deg f +1]$ is a morphism of complexes.
Let $g$ be a cocycle in $\mathfrak{g}^\bullet$ such that in cohomology we have $p_\zeta({g})={a}$. Then, in cohomology we have
\[
p_\zeta(f\smallsmile g)=(-1)^{f+g-1}\iota_{f\smallsmile g}\zeta=(-1)^{f-1}\iota_f(-1)^g\iota_g\zeta=(-1)^{f-1}\iota_fa.
\]
On the other hand, by definition of the $\smallsmile$ product on $H^\bullet(M)$, Eq.~\rmref{cup-on-M},
\[
{p_\zeta(f\smallsmile g)}=p_\zeta({f})\smallsmile p_\zeta({g})=p_\zeta({f})\smallsmile{a}
\]
in $H^\bullet(M)$.
\end{proof}
\begin{rem}
For $f$ and $g$ any two $\delta$-cocycles, applying both sides of the first of Eq.~\rmref{panem} to $\zeta$ and using Corollary \ref{I-will-need-you}, we get
\begin{align*}
\bc \cT(f,g)\zeta &=\iota_f\cL_g\zeta-(-1)^{(f+1)g}\cL_g\iota_f\zeta - \iota_{\{f, g\}}\zeta\\
&=\iota_f\cL_g\zeta-(-1)^{(f+1)g}\cL_{g \smallsmile f}\zeta -(-1)^{fg} \iota_g \cL_f\zeta 
\\
&
\qqquad\qqquad
+(-1)^{(f+1)g}\bc (S_{g \smallsmile f}-\cS_g\iota_f-\iota_g\cS_f )\zeta - \iota_{\{f, g\}}\zeta,
\end{align*}
and so in $H^\bullet(M)$ we find
\[
\iota_f\cL_g\zeta+(-1)^{f}\cL_{f \smallsmile g}\zeta -(-1)^{fg} \iota_g \cL_f\zeta  - \iota_{\{f, g\}}\zeta=0,
\]
where we have used Corollary \ref{I-will-need-you-as-well1} and the fact that the degree of $f$ and $g$ as elements of $\g^\bullet$ is their degree as elements of $\mathcal{G}_\g^\bullet$ plus $1$.
By Lemma \ref{where-streets-have-no-name}, this is equivalent to the following identity in $H^\bullet(M)$:
\begin{equation}\label{towards-Delta}
(p_\zeta({f})\smallsmile\cL_g\zeta) -(-1)^{fg+g+f} ( p_\zeta({g})\smallsmile \cL_f\zeta) -\cL_{f \smallsmile g}\zeta + (-1)^{g}p_\zeta({\{{f}, {g}\}})=0.
\end{equation}
\end{rem}

\begin{theorem}
\label{ruwenogien}
The degree $-1$ differential $\Bc$ on $H^\bullet(\mathcal{G}_M)$ from Remark \ref{rem.Bc-in-cohomology}
satisfies
\begin{equation}
\label{avvisoagliutenti1}
 \{a, b\} =(-1)^{a}\Bc  (a \smallsmile b) -(-1)^{a} (\Bc  a \smallsmile b)   -(a\smallsmile\Bc  b),
\end{equation}
for any homogeneous $a, b$ in $H^\bullet(\mathcal{G}_M)$.
Therefore, when $H^\bullet(\mathcal{G}_\g)$ is a Gerstenhaber algebra, $(H^\bullet(\mathcal{G}_M),\{\cdot,\cdot\},\smallsmile, \Bc)$ is a Batalin-Vilkoviski\u\i\ algebra.
%
\end{theorem}

\begin{proof}
%
Choose ${f}=p_\zeta^{-1}({a})$ and ${g}=p_\zeta^{-1}({b})$ in Eqs.~\rmref{towards-Delta} \& \rmref{here-is-p2}, and recall the definition of $\{\cdot,\cdot\}$ and $\smallsmile$ on $H^\bullet(M)$, that is, Eqs.~\rmref{cup-on-M}--\rmref{bracket-on-M}. By Remark \ref{filarete-rulez}, and recalling that the degree of $z$ and $w$ as elements of $M^\bullet$ is their degree as elements of $\mathcal{G}_M^\bullet$ plus $1$, Equation \rmref{towards-Delta} translates into Equation \rmref{avvisoagliutenti1}. Finally, notice that, trivially, if $H^\bullet(\mathcal{G}_\g)$ is a Gerstenhaber algebra, then $p_\zeta\colon (H^\bullet(\mathcal{G}_\g),\{\cdot,\cdot\},\smallsmile)\to  (H^\bullet(\mathcal{G}_M),\{\cdot,\cdot\},\smallsmile)$ is an isomorphism of Gerstenhaber algebras. In particular, the unit element $\mathbf{1}_M$ of $H^\bullet(\mathcal{G}_M)$ is the cohomology class of $\zeta$, so we have $\Bc(\mathbf{1}_M)=0$, by definition of a Palladio element.
\end{proof}


\subsection{The string topology bracket arising from calculi}

We are now going to show that $HC^\bullet_-(\mathcal{G}_M)$ carries a natural degree $-2$ Lie bracket $[\cdot,\cdot]$,  {\em i.e.}, equivalently, that $HC^\bullet_-(\mathcal{G}_M)[2]$ carries a natural graded Lie algebra structure such that the morphism of graded vector spaces $\beta\colon HC^\bullet_-(\mathcal{G}_M)[2]\to H^\bullet(M)$ becomes a morphism of Lie algebras $(HC^\bullet_-(\mathcal{G}_M)[2],[\cdot,\cdot])\to (H^\bullet(M),\{\cdot,\cdot\})$.

\begin{definition}
The {\em Chas-Sullivan-Menichi bracket} is the degree $-2$ bracket 
\[
[\cdot,\cdot]\colon HC^\bullet_-(\mathcal{G}_M)\otimes HC^\bullet_-(\mathcal{G}_M)\to HC^\bullet_-(\mathcal{G}_M)[-2]
\]
on $HC^\bullet_-(\mathcal{G}_M)$ defined by
\[
[{x},{y}] := (-1)^x  j ((\beta{x})\smallsmile(\beta{y})).
\]
\end{definition}

This bracket is often also referred to as {\em string topology bracket} and has been introduced in \cite{ChaSul:ST} and abstractly formulated in \cite{Men:BVAACCOHA}.

\begin{theorem}
\label{chas-sullivan-menichi}
If $H^\bullet(\mathcal{G}_\g)$ is a Gerstenhaber algebra, then the Chas-Sullivan-Menichi bracket is a Lie bracket, 
and has the property that 
$$
\beta[\cdot, \cdot] = \{ \beta(\cdot), \beta(\cdot)\}.
$$
\end{theorem}

\begin{proof} 
As $\beta$ is a degree $1$ operator as a linear morphism from $HC^\bullet_-(\mathcal{G}_M)[2]$ to $H^\bullet(\mathcal{G}_M)$, for homogeneous elements $x,y$ in $HC^\bullet_-(\mathcal{G}_M)[2]$, we have
\begin{align*}
[x,y]&=(-1)^xj ((\beta{x})\smallsmile(\beta{y}))\\
&=(-1)^{x+(x+1)(y+1)}j((\beta{y})\smallsmile(\beta{x}))\\
&=-(-1)^{y+xy}j((\beta{y})\smallsmile(\beta{x}))\\
&=-(-1)^{xy}[y,x].
\end{align*}
Now recall the seven terms identity for Batalin-Vilkoviski\u\i\ algebras, that is, the $BV$-operator $\Bc$ is of second order \cite[Eq.~(1.3)]{Kos:CDSNEC}:
\begin{align*}
\Bc({a}&\smallsmile {b}\smallsmile {c})+
((\Bc{a})\smallsmile {b}\smallsmile {c})\\
&+(-1)^a({a}\smallsmile (\Bc{b})\smallsmile {c}))+(-1)^{a+b}
({a}\smallsmile {b}\smallsmile (\Bc{c}))\\\
&-(\Bc({a}\smallsmile {b})\smallsmile {c}) -(-1)^a
({a}\smallsmile \Bc({b}\smallsmile {c}))-(-1)^{(a+1)b}
({b}\smallsmile \Bc({a}\smallsmile {c})) =0.
\end{align*}
Choose then ${a}=\beta{x}$, ${b}=\beta{y}$, ${c}=\beta{z}$, 
and use the identity $\Bc=\beta j$ from Lemma \ref{B-returns} to get
\begin{align*}
\beta j(\beta{x}&\smallsmile \beta{y}\smallsmile \beta{z})+
((\beta j\beta{x})\smallsmile \beta{y}\smallsmile \beta{z})\\
&+(-1)^{x+1}(\beta{x}\smallsmile (\beta j\beta{y})\smallsmile \beta{z}))+(-1)^{x+y}
(\beta{x}\smallsmile \beta{y}\smallsmile (\beta j\beta{z}))\\
&-(\beta j(\beta{x}\smallsmile \beta{y})\smallsmile \beta{z}) -(-1)^{x+1}
(\beta{x}\smallsmile \beta j(\beta{y}\smallsmile \beta{z}))
\\
&
\qqquad\qqquad\qqquad\qqquad
-(-1)^{x(y+1)}
(\beta{y}\smallsmile \beta j(\beta{x}\smallsmile \beta{z})) =0.
\end{align*}
As $j\beta=0$, this simplifies to
\begin{align}\label{gottagivee3}
\beta j(\beta{x}&\smallsmile \beta{y}\smallsmile \beta{z})-(\beta j(\beta{x}\smallsmile \beta{y})\smallsmile \beta{z})\notag
\\
& -(-1)^{x+1}
(\beta{x}\smallsmile \beta j(\beta{y}\smallsmile \beta{z}))-(-1)^{x(y+1)}
(\beta{y}\smallsmile \beta j(\beta{x}\smallsmile \beta{z})) =0.
\end{align}
Apply $j$ to the above expression to get
\begin{align*}
j(\beta j(\beta{x}\smallsmile \beta{y})\smallsmile \beta{z})
& +(-1)^{x+1}j(\beta{x}\smallsmile \beta j(\beta{y}\smallsmile \beta{z}))
\\
&
+(-1)^{x(y+1)}
j(\beta{y}\smallsmile \beta j(\beta{x}\smallsmile \beta{z})) =0,
\end{align*}
by using once more the identity $j\beta=0$. Now recall the definition of the Chas-Sullivan-Menichi bracket to rewrite the above identity as
\[
[x, [y,z]]= [[x,y],z] +(-1)^{xy}[y,[x,z]].
\]
This shows that the bracket $[\cdot,\cdot]$ satisfies the Jacobi identity.
As for the second part, the claimed identity is immediate from Eq.~\rmref{avvisoagliutenti1} along with $j\beta = 0$ and $\beta j = \Bc$ again.
\end{proof}

\begin{theorem}
\label{arkopharma}
If $\{\cdot,\cdot\}$ vanishes identically on $H^\bullet(\mathcal{G}_M)$, then 
\[
\{\!\!\{ x,y\}\!\!\}:=(-1)^x (\Bc{x})\smallsmile(\Bc{y})
\]
defines a degree $-2$ Lie bracket on $H^\bullet(\mathcal{G}_M)$ such that $j\{\!\!\{ x,y\}\!\!\}=[jx,jy]$ and $\Bc\{\!\!\{ x,y\}\!\!\}=0$. This bracket turns $\big(H^\bullet(\mathcal{G}_M)\smallsmile,\{\!\!\{\cdot,\cdot\}\!\!\}\big)$ into an $e_3$-algebra, {\em i.e.},
\begin{eqnarray}
\label{e31}
\{\!\!\{ x,y\}\!\!\} &=& - (-1)^{xy} \{\!\!\{y,x\}\!\!\} \\
\label{e32}
\{\!\!\{ x,\{\!\!\{ y,z\}\!\!\}\}\!\!\} &=& \{\!\!\{\{\!\!\{ x,y\}\!\!\}, z \} \!\!\} + (-1)^{xy}  \{\!\!\{ y,\{\!\!\{ x,z\}\!\!\}\}\!\!\} \\ 
\label{e34}
\{\!\!\{ x,y \smallsmile z \}\!\!\} &=&  \{\!\!\{ x,y\}\!\!\} \smallsmile z + (-1)^{xy} y\smallsmile \{\!\!\{ x,z\}\!\!\}.  
\end{eqnarray}
\end{theorem}

\begin{proof}
Since the cup product $\smallsmile$ on $H^\bullet(\mathcal{G}_M)$ is graded commutative and $\Bc$ has degree $-1$, we have
\begin{align*}
\{\!\!\{x,y\}\!\!\}&=(-1)^x (\Bc{x})\smallsmile(\Bc{y})\\
&=(-1)^{x+(x-1)(y-1)}(\Bc{y})\smallsmile(\Bc{x})\\
&=-(-1)^{y+xy}(\Bc{y})\smallsmile(\Bc{x})\\
&=-(-1)^{xy}\{\!\!\{y,x\}\!\!\},
\end{align*}
which proves equation \rmref{e31}.
By equation \rmref{gottagivee3}, we have
\begin{align*}
\beta j(\beta{a}&\smallsmile \beta{b}\smallsmile \beta{c})-(\beta j(\beta{a}\smallsmile \beta{b})\smallsmile \beta{c})
\\
& -(-1)^{a+1}
(\beta{a}\smallsmile \beta j(\beta{b}\smallsmile \beta{c}))-(-1)^{a(b+1)}
(\beta{b}\smallsmile \beta j(\beta{a}\smallsmile \beta{c})) =0,
\end{align*}
for any $a,b,c$ in $HC^\bullet_-(\mathcal{G}_M)$. Choosing $a=jx$, $b=jy$, $c=jz$ and recalling that $\beta j=\Bc$, this becomes
\begin{align*}
\Bc(\Bc x&\smallsmile \Bc y\smallsmile \Bc z)-(\Bc(\Bc x\smallsmile \Bc y)\smallsmile \Bc z)
\\
& -(-1)^{x+1}
(\Bc x\smallsmile \Bc(\Bc y\smallsmile \Bc z))-(-1)^{x(y+1)}
(\Bc y\smallsmile \Bc(\Bc x\smallsmile \Bc z)) =0,
\end{align*}
that is,
\[
\{\!\!\{ x, \{\!\!\{ y\,z\}\!\!\}\}\!\!\}=\{\!\!\{\{\!\!\{x,y\}\!\!\}, z\}\!\!\}
 +(-1)^{xy}
\{\!\!\{ y, \{\!\!\{  x, z\}\!\!\}\}\!\!\}
 -(-1)^y\Bc(\Bc x\smallsmile \Bc y\smallsmile \Bc z).
\]
If $\{\cdot,\cdot\}$ vanishes identically on $H^\bullet(\mathcal{G}_M)$, then Eq.~\rmref{avvisoagliutenti1} tells us that $\Bc$ is a derivation of the cup product. Since $\Bc^2=0$, this implies 
\[
\Bc(\Bc x\smallsmile \Bc y\smallsmile \Bc z)=0,
\]
and so
\[
\{\!\!\{ x, \{\!\!\{ y\,z\}\!\!\}\}\!\!\}=\{\!\!\{\{\!\!\{x,y\}\!\!\}, z\}\!\!\}
 +(-1)^{xy}
\{\!\!\{ y, \{\!\!\{  x, z\}\!\!\}\}\!\!\},
\]
which proves equation \rmref{e32}. This shows that $\{\!\!\{\cdot,\cdot\}\!\!\}$
defines a degree $-2$ Lie bracket on $H^\bullet(\mathcal{G}_M)$. Recalling the definition of the Lie bracket $[\cdot,\cdot]$ on $HC^\bullet_-(\mathcal{G}_M)$, we find
\[
[jx,jy]=(-1)^x  j ((\beta{jx})\smallsmile(\beta{jy}))=(-1)^x  j ((\Bc x)\smallsmile(\Bc y))=j\{\!\!\{ x,y\}\!\!\}.
\]
Applying $\beta$ to both sides and recalling that $
\beta[\cdot, \cdot] = \{ \beta(\cdot), \beta(\cdot)\}$, we have
\[
\Bc\{\!\!\{ x,y\}\!\!\}=\beta j\{\!\!\{ x,y\}\!\!\}=\beta[jx,jy]=\{ \beta jx, \beta jy\}=\{ \Bc x, \Bc y\}=0
\]
since we are assuming $\{\cdot,\cdot\}$ vanishes identically on $H^\bullet(\mathcal{G}_M)$. 
Finally, using again that if $\{\cdot,\cdot\}=0$ then $\Bc$ is a derivation of the cup product,
\begin{align*}
\{\!\!\{ x,y\smallsmile z\}\!\!\}&=(-1)^x (\Bc{x})\smallsmile(\Bc({y\smallsmile z}))\\
&=(-1)^x (\Bc{x})\smallsmile((\Bc y)\smallsmile z+(-1)^y y\smallsmile(\Bc z))\\
&=(-1)^x (\Bc{x})\smallsmile(\Bc y)\smallsmile z+(-1)^{x+y} (\Bc x)\smallsmile y\smallsmile(\Bc z)\\
&=\{\!\!\{x,y\}\!\!\} \smallsmile z+(-1)^{x(y-1)} y \smallsmile (\Bc x) \smallsmile(\Bc z)\\
&=\{\!\!\{x,y\}\!\!\} \smallsmile z+(-1)^{xy} y \smallsmile \{\!\!\{ x, z \}\!\!\}, 
\end{align*}
which proves the last identity \rmref{e34}.
\end{proof}

\begin{rem}
In the proof of Theorem \ref{arkopharma}, we have repeatedly used that $\Bc$ is a derivation of the cup product as soon as $\{\cdot,\cdot\}$ vanishes. Observe that then $\Bc$ also becomes (trivially) a derivation of the bracket $\{\!\!\{ \cdot,\cdot\}\!\!\}$. Namely, one has 
\begin{equation*}
\Bc \{\!\!\{ x,y\}\!\!\} =  \{\!\!\{ \Bc x,y\}\!\!\} + (-1)^{x} \{\!\!\{ x, \Bc y\}\!\!\} 
\end{equation*}
as both sides are zero: that the left hand side is zero was just shown above, whereas for the right hand side this follows from  $\Bc^2=0$.
\end{rem}


\section{Brackets for opposite modules over operads}
\label{pling}

In this section, we will see that the notion of a {\em cyclic opposite module} over an operad with multiplication provides a quite general class of examples that illustrates the constructions in \S\ref{fame}. For a collection of basic notions in operad theory that will be used in the sequel, see Appendix \ref{pamukkale}.

\subsection{Cyclic opposite modules over operads with multiplication}
\label{sansonitedescoitalianoitalianotedesco}

Let $\cO$ be an operad in the sense described in \S\ref{pamukkale1}.

\begin{definition}
\label{molck}
An {\em opposite (left) $\cO$-module} is
a sequence of $\K$-modules $\cN = \{ \cN(n) \}_{n \geq 0}$ along with $\K$-linear
operations,  
$$
        \bullet_i : 
        \cO(p) \otimes \cN(n) \to \cN({n-p+1}), \quad \mbox{for} \ 
i = 1, \ldots, n- p +1, \quad 0 \leq p \leq n, 
$$
that are declared to be zero if $p > n$, and 
for $\gvf \in \cO(p)$, $\psi \in \cO(q)$, and $x \in
\cN(n)$ subject to the identities  
\begin{equation}
\label{TchlesischeStr}
\gvf \bullet_i (\psi \bullet_j x) = 
\begin{cases} 
\psi \bullet_j (\gvf \bullet_{i + q - 1}  x) \quad & \mbox{if} \ j < i, 
\\
(\gvf \circ_{j-i+1} \psi) \bullet_{i}  x \quad & \mbox{if} \ j - p < i \leq j, \\
\psi \bullet_{j-p + 1} (\gvf \bullet_{i}  x) \quad & \mbox{if} \ 1 \leq i \leq j - p,
\end{cases}
\end{equation}
where $p > 0$, $q \geq 0$, $ n \geq 0$. 
In case $p=0$, the index $i$ runs from $1$ to $n+1$, and the above relations have to be read as
 \begin{equation}
\label{TchlesischeStr0}
\gvf \bullet_i (\psi \bullet_j x) = 
\begin{cases} 
\psi \bullet_j (\gvf \bullet_{i + q - 1}  x) \quad & \mbox{if} \ j < i, 
\\
\psi \bullet_{j+1} (\gvf \bullet_{i}  x) \quad & \mbox{if} \ 1 \leq i \leq j.
\end{cases}
\end{equation}
An opposite $\cO$-module is called {\em unital} if  
 \begin{equation}
\label{TchlesischeStr-1}
\mathbb{1} \bullet_i x = x, \quad \mbox{for} \ i = 1, \ldots, n,
\end{equation}
for all $x \in \cN(n)$.
\end{definition}

\begin{example}
\label{schnee1}
For a $\K$-module $X$, one defines the {\em endomorphism operad} $\mathcal{E}\hspace*{-1pt}{nd}_\ikks$ by  $\mathcal{E}\hspace*{-1pt}{nd}_\ikks(p) := \Hom(X^{\otimes p}, X)$ with identity element $\mathbb{1} := \id_\ikks$. A unital opposite module over this operad is defined by $\cN_\ikks(n) := X^{\otimes n +1 }$ along with
composition maps defined by
$$
\gvf \bullet_i (x_0, \ldots, x_n) := (x_0, \ldots, x_{i-1}, \gvf(x_{i}, \ldots, x_{i+p-1}), x_{i+p}, \ldots, x_n),
$$
for  $i= 1, \ldots, n -p +1$, 
where $\gvf \in \mathcal{E}\hspace*{-1pt}{nd}_\ikks(p)$ and $x := (x_0, \ldots, x_n) := x_0 \otimes \cdots \otimes x_n \in \cN_\ikks(n)$. 
When $X$ happens to be an associative $\K$-algebra, the operad $\mathcal{E}\hspace*{-1pt}{nd}_\ikks$ becomes an operad with multiplication.
\end{example}

The preceding example suggests the following picture for an opposite module:

\begin{center}
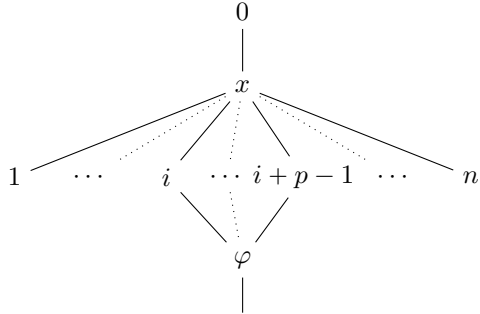

\begin{tikzpicture}
\node(0) at (0, 0) {$0$};
\node(x) at (0,-1) {$x$};
\node(1) at (-3,-2.2) {$1$};
\node(2) at (-2,-2.2) {$\cdots$};
\node(i) at (-1,-2.2) {$i$};
\node(4) at (-0.2,-2.2) {$\cdots$};
\node(i+p-1) at (0.8,-2.2) {$i+p-1$};
\node(6) at (2,-2.2) {$\cdots$};
\node(n) at (3,-2.2) {$n$};
\node(phi) at (0,-3.3) {$\varphi$};

\draw(x)--(1);
\draw[dotted](x)--(2);
\draw(x)--(i);
\draw[dotted](x)--(4);
\draw(x)--(i+p-1);
\draw[dotted](x)--(6);
\draw(x)--(n);
\draw(i)--(phi);
\draw[dotted](4)--(phi);
\draw(i+p-1)--(phi);
\draw(0)--(x);

\draw(phi)--(0,-4);
\end{tikzpicture}
\captionof{figure}{Opposite modules}\label{tratto}
\end{center}

\begin{definition}
\label{piantadiroma}
A {\em para-cyclic (unital, left) opposite $\cO$-module} 
is a (unital, left) opposite $\cO$-module $\cN$ 
endowed with two additional structures:
first, an {\em extra} ($\K$-linear) composition map
$$
\bullet_0: \cO(p) \otimes \cN(n) \to \cN({n-p+1}), \quad 0 \leq p \leq n+1,
$$
declared to be zero if $p > n+1$, 
and such that the relations
\rmref{TchlesischeStr}--\rmref{TchlesischeStr-1} are fulfilled for $i=0$ as well, that is, 
for $\gvf \in \cO(p)$, $\psi \in \cO(q)$, and $x \in \cN(n)$, with $q \geq 0$, $n \geq 0$, 
\begin{eqnarray}
\label{SchlesischeStr}
\gvf \bullet_i (\psi \bullet_j x) \!\!\!\!&=\!\!\!\!& 
\begin{cases} 
\psi \bullet_j (\gvf \bullet_{i + q - 1}  x)  & \mbox{if} \ j < i, 
\\
(\gvf \circ_{j-i+1} \psi) \bullet_{i}  x  & \mbox{if} \ j - p < i \leq j, \\
\psi \bullet_{j-p + 1} (\gvf \bullet_{i}  x)  & \mbox{if} \ 0 \leq i \leq j - p,
\end{cases} \qquad \mbox{(cases for $p > 0$)\qqquad}
\\
\label{SchlesischeStr0}
\gvf \bullet_i (\psi \bullet_j x) \!\!\!\!&=\!\!\!\!& 
\begin{cases} 
\psi \bullet_j (\gvf \bullet_{i + q - 1}  x)  & \mbox{if} \ j < i, 
\\
\psi \bullet_{j+1} (\gvf \bullet_{i}  x)  & \mbox{if} \ 0 \leq i \leq j,
\end{cases}
\qqquad\quad  \mbox{(cases for $p = 0$)}
\\
\label{SchlesischeStr-1}
\mathbb{1} \bullet_i x \!\!\!\!&=\!\!\!\!&  x \hspace*{3cm} \, \mbox{for \ } i = 0, \ldots, n,
\end{eqnarray}
and second, a degree-preserving morphism $t: \cN(n) \to \cN(n)$ for all $ n \geq 1$ such that
\begin{equation}
\label{lagrandebellezza1}
t(\gvf \bullet_{i} x) = \gvf \bullet_{i+1} t(x),  \qquad i = 0, \ldots, n-p,
\end{equation}
is true for $\gvf \in \cO(p)$ and $x \in \cN(n)$. 
We call a para-cyclic opposite $\cO$-module {\em cyclic} if
\begin{equation}
\label{lagrandebellezza2}
t^{n+1} = \id
\end{equation}
holds on $\cN(n)$.
\end{definition}

\begin{example}
\label{schnee2}
The opposite module $\cN_\ikks$ from Ex.~\ref{schnee1} can be easily extended to a cyclic opposite module over  $\mathcal{E}\hspace*{-1pt}{nd}_\ikks$: 
define
$
\gvf \bullet_0 (x_0, \ldots, x_n) := \big(\gvf(x_0, \ldots, x_{p-1}), x_p, \ldots, x_n\big),
$
along with 
$
t(x_0, \ldots, x_n) := (x_n, x_0, \ldots, x_{n-1}).
$
Observe, however, that in more general contexts such as for Hopf algebroids things are not as simple; see \cite[Eq.~(6.20)]{Kow:GABVSOMOO}. 
\end{example}

Graphically, the condition \rmref{lagrandebellezza1} can be understood by a picture similar to those that depict cyclic operads:

\begin{center}
\begin{tikzpicture}
\node(0) at (-6, -3) {$0$};
\node(x) at (-4,-1.4) {$x$};
\node(2) at (-5.5,-3) {$\cdots$};
\node(i) at (-5,-3) {$i$};
\node(4) at (-4.6,-3) {$\cdots$};
\node(i+p-1) at (-3.7,-3) {$i+p-1$};
\node(6) at (-2.7,-3) {$\cdots$};
\node(n-1) at (-2,-3) {$n-1$};
\node(n) at (-4,0) {$n$};
\node(phi) at (-4.5,-4.3) {$\varphi$};

\draw[dotted](x)--(2);
\draw(x)--(i);
\draw[dotted](x)--(4);
\draw(x)--(i+p-1);
\draw[dotted](x)--(6);
\draw(x)--(n-1);
\draw(i)--(phi);
\draw[dotted](4)--(phi);
\draw(i+p-1)--(phi);

\draw(phi)--(-4.5,-5.3);




\draw  (-3.7,-1.5) to [out=330,in=270] (-1.5,-2) to [out=90,in=270] (-4,-0.2) ;
\draw (x) to [out=90, in=38] (-5.8, -2.8);

\node(0) at (1, -3) {$0$};
\node(x) at (3,-1.4) {$x$};
\node(2) at (1.5,-3) {$\cdots$};
\node(i+1) at (2.2,-3) {$i+1$};
\node(4) at (2.85,-3) {$\cdots$};
\node(i+p) at (3.5,-3) {$i+p$};
\node(6) at (4.3,-3) {$\cdots$};
\node(n-1) at (5,-3) {$n-1$};
\node(n) at (3,0) {$n$};
\node(phi) at (2.8,-4.3) {$\varphi$};
\node(=) at (-0.3, -3) {$=$};

\draw[dotted](x)--(2);
\draw(x)--(i+1);
\draw[dotted](x)--(4);
\draw(x)--(i+p);
\draw[dotted](x)--(6);
\draw(x)--(n-1);
\draw(i+1)--(phi);
\draw[dotted](4)--(phi);
\draw(i+p)--(phi);

 \draw(phi)--(2.8,-5.3);



\draw  (3.3,-1.5) to [out=330,in=270] (5.5,-2) to [out=90,in=270] (3,-0.2) ;
\draw (x) to [out=90, in=38] (1.2, -2.8);

\draw [thick, rounded corners] (-6.3, -0.5) rectangle (-1.2, -4.8);
\draw [thick, rounded corners] (0.7, -0.5) rectangle (5.8, -3.7);

 \end{tikzpicture}
 
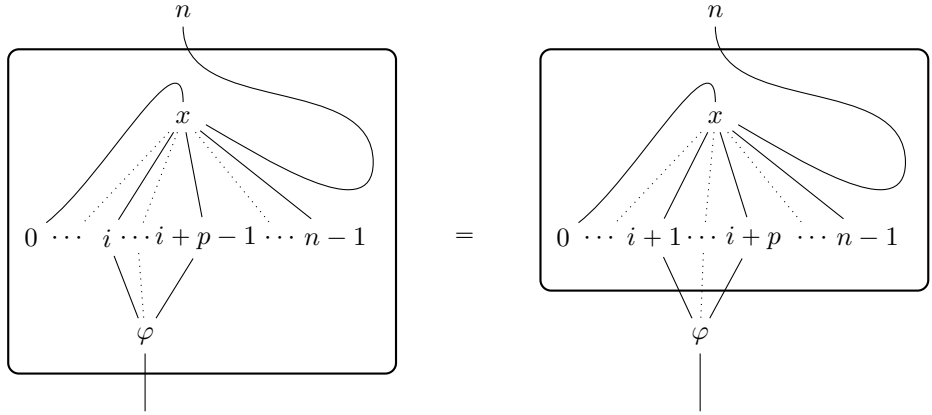
\captionof{figure}{The relation $t(\gvf \bullet_{i} x) = \gvf \bullet_{i+1} t(x)$ for cyclic opposite modules}\label{rel}
 \end{center}

In particular \cite[Prop.~3.5]{Kow:GABVSOMOO}, a cyclic
unital opposite module 
$\cN$ 
over an operad with multiplication 
$(\cO, \mu, e)$ carries the structure of a cyclic $\K$-module
 with cyclic operator $t: \cN(n) \to \cN(n)$, along with faces $d_i: \cN(n) \to \cN({n-1})$ and degeneracies $s_j: \cN(n) \to \cN({n+1})$ of the underlying simplicial object given by 
\begin{equation}
\label{colleoppio}
\begin{array}{rcll}
d_i(x) & = & \mu \bullet_{i} x, & i = 0, \ldots, n-1, \\
d_n(x) & = & \mu \bullet_0 t (x), & \\
s_j (x)& = & e \bullet_{j+1} x, & j = 0, \ldots, n, \\
\end{array}
\end{equation}
where $x \in \cN(n)$. Defining 
the differential $b\colon \cN(n)\to \cN({n-1})$ as
$
b=\sum_{i=0}^n (-1)^id_i,
$
the pair $(\cN, b)$ becomes a chain complex.
In a standard way, by setting $B := (\mathrm{id}-t) \, s_{-1} \, N$, where 
$N := \sum_{i=0}^n (-1)^{in} t^i$ and $s_{-1} := t \, s_n$, the triple $(\cN, b, B)$ 
becomes a mixed (chain) complex. Explicitly, one has by \rmref{lagrandebellezza2}
\begin{equation}
\label{cine40}
       s_{-1}(x) = t \, s_n(x) = t(e \bullet_{n+1} x) =  t(e \bullet_{n+1} t^{n+1}x)
=  t^{n+2}(e \bullet_0 x) = e \bullet_0 x,
\end{equation}
and hence
$ 
        B(x) = \sum^n_{i=0} (-1)^{in} (\mathrm{id}-t) e \bullet_0 t^i(x).
$

\begin{rem}
\label{normalised}
To simplify all expressions involved, we will work on the {\em normalised} complex $\widebar{\cN}$
of $\cN$, meaning its quotient by the (acyclic) subcomplex spanned by the
images of the degeneracy maps of this simplicial 
$\K$-module.
For example, the coboundary
$B$ coincides on $\widebar{\cN}$ 
with the map (induced by) $s_{-1} \, N$, and so by abuse of notation, we 
will denote the latter by 
$B$ as well, and write
\begin{equation}
\label{extra2}
B(x) = \sum^n_{i=0} (-1)^{in} e \bullet_0 t^i(x).
\end{equation}

%
Similarly, $\widebar{\cO}$ denotes the intersection of the kernels of the codegeneracies in the cosimplicial $\K$-module $\cO$. Notice that $\widebar{\cO}(0)=\cO(0)$, as the intersection of an empty family of subsets of $\cO(0)$ is $\cO(0)$ itself.
\end{rem}

\subsection{The calculus operations for cyclic opposite modules over operads}
\label{pasticciottoleccese}

For the rest of this section, let $(\cO, \mu, e)$ be an operad with multiplication and $\cN$ a cyclic unital opposite $\cO$-module. 

\begin{definition}
\
\label{volterra}
\begin{enumerate}
\compactlist{99}
\item
The {\em cap product} or {\em contraction operator}
\begin{align*}
i\colon  \cO(p)\otimes  \cN(n)& \to \cN({n-p})\\
(\gvf,x)&\mapsto i_\varphi x:=\gvf \smallfrown x
\end{align*}
of $\varphi \in \cO(p)$ with 
$x \in \cN(n)$ is defined by
\begin{equation}
\label{alles4}
\gvf \smallfrown x := 
(\mu \circ_2 \gvf) \bullet_0 x.  
\end{equation}
\item
The {\em Lie derivative}  
\begin{align*}
        L : \cO(p)\otimes \cN(n) &\rightarrow 
        \cN({n-p+1})\\
        (\varphi,x)&\mapsto L_\varphi x
\end{align*}
of $x \in \cN(n)$
along $\gvf \in \cO(p)$ with $p < n+1$ is defined to be 
\begin{equation}
\label{messagedenoelauxenfantsdefrance2}
        L_\varphi x :=   \!\!      
\sum^{n-p+1}_{i=1} \!\! (-1)^{(p-1)(i-1)} \varphi \bullet_i x 
+ 
\sum^{p}_{i=1} (-1)^{n(i-1) + p - 1} \varphi \bullet_0 t^{i-1} (x).
\end{equation}
In case $p = n+1$, this means that
$$
       L_\varphi x := (-1)^{p-1} \varphi \bullet_0  N(x),
$$
and for $p > n+1$, we define 
$L_\varphi := 0$. 
\item
Finally, define 
\begin{align*}
         S\colon \cO(p)\otimes \cN(n) & \rightarrow 
        \cN({n-p+2})\\
        (\varphi,x)&\mapsto S_\varphi x
\end{align*}
for every $\varphi \in  \cO(p)$ and $x\in \cN(n)$, as follows:
for $0 \leq p \le n$ by
\begin{equation}
\label{capillareal1}
        S_\varphi x := 
        \sum^{n-p+1}_{j=1} \, 
        \sum^{n - p+1}_{i=j} (-1)^{ n(j-1) + (p-1)(i-1)} e \bullet_0 \big(\gvf \bullet_i t^{j-1}(x)\big),
\end{equation}
while, if $p>n$, set
$
        S_\varphi := 0.
$
\end{enumerate}
\end{definition}

A direct verification using the Eqs.~\rmref{SchlesischeStr} as well as those in \rmref{danton} leads to the identity:

\begin{lemma}
For any $\gvf,\psi\in \cO$ one has
\begin{equation}
\label{missing-equation-see-email}
i_\gvf i_\psi=i_{\gvf\smallsmile \psi}.
\end{equation}
\end{lemma}


On top of that, 
the following was proven in {\cite[Thm.~5.4]{Kow:GABVSOMOO}}.

\begin{theorem}
 Restricting to normalised elements $\gvf \in \widebar{\cO}$, one obtains 
the {\em Cartan-Rinehart homotopy} formulae 
\begin{equation}
\label{dellera-operad}
\begin{cases}
L_\gvf= [B, i_\gvf]+[b,  S_\gvf] - S_{\gd \gvf},\\
[b,i_\gvf]-i_{\gd \gvf}=0,\\
[B, S_\gvf]=0,
\end{cases}
\end{equation}
%
on the normalised complex $\widebar{\cN}$. 
\end{theorem}


\begin{rem}
\label{formsagain}
With the right input data (see \cite[\S6]{KowKra:BVSOEAT}), this can be seen as a generalisation of the classical calculus from differential geometry as in the first part of Example \ref{forms'n'fields} of ``fields acting on forms''.
\end{rem}

Whereas the classical relation $[\cL_X, \iota_Y] = \iota_{[X,Y]}$ between Lie derivative and contraction from differential geometry for two (multi-)vector fields $X, Y$ transfers identically to the operadic context when descending to (co)homology (see \cite[Thm.~4.5]{Kow:GABVSOMOO}), the algebraic HKR theorem hints at the fact that this cannot be that simple when staying on the (co)chain level. More precisely, one needs to introduce an additional homotopy:

\begin{theorem}
\label{nightmare}
Define the {\em Gel'fand-Daletski\u\i-Tsygan} homotopy $T\colon \cO\otimes\cO\otimes \cN\to \cN[-2]$ as 
\begin{equation}
\label{loesungsschluessel}
\begin{split}
&T: \cO(p) \otimes \cO(q) \otimes \cN(n) \to \cN(n-p-q+2), \\
&\quad
(\gvf,\psi,x) \mapsto 
\displaystyle \sum^{p-1}_{j=1}  \sum^{p-1}_{i=j} (-1)^{n(j-1) + (q-1)(i-j) + p} (\gvf \circ_{p-i+j} \psi) \bullet_0 t^{j-1}(x).
\end{split}
\end{equation}
Writing $T(\gvf, \psi)(x) := T(\gvf, \psi, x)$,
one separately has for all $\gvf, \psi \in \cO$
\begin{equation}
\label{panem2}
[i_\psi, L_\gvf] - i_{\{\psi, \gvf\}} = [b, T(\gvf,\psi)] - T(\gd \gvf, \psi) - (-1)^{\gvf-1} T(\gvf, \delta \psi) 
\end{equation}
on $\cN$ along with 
\begin{equation}
\label{etcircenses}
[S_\psi, L_\gvf] - S_{\{\psi,\gvf\}} = [B, T(\gvf,\psi)], 
\end{equation}
for $\gvf, \psi \in \widebar{\cO}$ on the normalised complex $\overline{\cN}$.
\end{theorem}

The proof basically only relies on heavy use of \rmref{SchlesischeStr}--\rmref{colleoppio}, the associativity properties \rmref{danton} of the operadic composition, as well as enforced multiple sum yoga resulting in a {\em shock \& awe} computation, which is why it can be found in Appendix \ref{shock&awe}.

\begin{rem}
The operator $T$ was implicitly introduced in \cite{GelDalTsy:OAVONCDG} in the context of associative algebras. The obvious advantage of the explicit expression \rmref{loesungsschluessel} is that it not only applies to the endomorphism operad of an associative algebra acting on chains that governs Hochschild theory, but more general to any operad with multiplication, such as for example those arising in the context of Hopf algebras, differential operators, or more general bialgebroids \cite{Kow:GABVSOMOO}, and moreover also to cyclic operads as we will see in \S\ref{plong}.
\end{rem}


Finally, the following (partially) transfers the concept of Poincar\'e duality (as happens for example in the case of Calabi-Yau algebras) to the operadic context:

\begin{dfn}
Let $\cN$ be a cyclic (unital) opposite module over an operad $\cO$ with multiplication.
If there is a cocycle $\zeta \in \cN(d)$ such that the map 
$$
\cO \to \cN, \quad \gvf \mapsto i_\gvf \zeta = \gvf \smallfrown \zeta  
$$
induces an isomorphism
$$
H^n(\cO) \cong H_{d-n}(\cN),
$$
then we say that there is {\it (Poincar\'e) duality} between $\cO$ and $\cN$ with {\it fundamental class} $[\zeta]$, and that $(\cO, \cN, \zeta)$ induces a {\em duality calculus}.
\end{dfn}

In this case, one immediately has the result from \cite[Cor.~1.6]{Lam:VDBDABVSOCYA}:

\begin{prop}
If $(\cO, \cN, \zeta)$ induces a duality calculus, $H^\bullet(\cO)$ carries the structure of a BV algebra.
\end{prop}

We can now apply Theorem \ref{fapurefreddoquadentro} to get a degree $(1-d)$ Lie bracket on the negative cyclic homology of a cyclic opposite module.
Namely, let $\mathfrak{g}^\bullet$ and $M^\bullet$ be defined by $\mathfrak{g}^k :=\overline{\cO}({k+1})$ and $M^k :=\overline{\cN}({d-k-1})$, and set
\[
\bc:=b\colon M^k=\overline{\cN}({d-k-1})\to \overline{\cN}({d-k-2})=M^{k+1}
\]
along with
\[
\Bc:=B\colon M^k=\overline{\cN}({d-k-1})\to \overline{\cN}({d-k})=M^{k-1}.
\]
Then $\mathfrak{g}^\bullet$ is a differential graded Lie algebra with the Gerstenhaber bracket and differential (as in Appendix \ref{pamukkale1}), and $(M^\bullet, \bc, \Bc)$ constitutes a mixed (cochain) complex.
Also set
\begin{small}
\begin{align*}
\iota :=i\colon & \mathfrak{g}^h\otimes M^k=  \overline{\cO}({h+1})\otimes  \overline{\cN}({d-k-1}) \to \overline{\cN}({d-k-h-2})=M^{h+k+1},
\\
 \cL :=L \colon &\mathfrak{g}^h \otimes M^k=\overline{\cO}({h+1})\otimes  \overline{\cN}({d-k-1}) \to \overline{\cN}_{d-k-h-1})=M^{h+k},
\\
\cS :=S \colon &\mathfrak{g}^h \otimes M^k=\overline{\cO}({h+1})\otimes  \overline{\cN}({d-k-1}) \to \overline{\cN}({d-k-h})=M({h+k-1}),
\\
\cT :=T\colon &\mathfrak{g}^h \otimes \mathfrak{g}^j \otimes M^k =\overline{\cO}({h+1})\otimes \overline{\cO}({j+1})\otimes  \overline{\cN}({d-k-1})
\to \overline{\cN}({d-k-h-j-1})=M^{k+h+j}.
\end{align*}
\end{small}
Then the Equations \rmref{dellera-operad}, \rmref{missing-equation-see-email}, and \rmref{panem2}--\rmref{etcircenses} precisely say that $(\mathfrak{g}^\bullet,\iota,\cL,\cS,\cT)$ is a homotopy Cartan-Gerstenhaber calculus on $CC^\bullet_{\mathrm{per}}(M)$. 

\begin{rem}
\label{landweber}
Operadic duality calculi can always be induced by Palladio cocycles:
indeed, a cocycle $\zeta$ in $\cN(d)$ inducing a duality calculus is an element $\zeta$ in $M^{-1}$ such that $\bc\zeta=0$ and such that 
\[
\iota_{(\cdot)}\zeta\colon H^k(\mathfrak{g})\to H^k(M)
\]
is an isomorphism. 
%
%
As $\g^{-2}=\overline{\cO}({-1}) = 0$, we have $H^{-2}(\g)=0$, and so a fundamental class $[\zeta]$ is always defined by a Palladio cocycle (see Remark \ref{bocconotti}).
\end{rem}

Recalling Definition \ref{negative-cyclic-homology}, the isomorphisms $H^\bullet(\cO) \cong H^\bullet(\overline{\cO})$ and $H_\bullet(\cN) \cong H_\bullet(\overline{\cN})$, and the isomorphism $HC^-_\bullet(\overline{\cN}) \cong HC^-_\bullet(\cN)$ (which follows from Prop.~1.1.15  in \cite{Lod:CH} combined with {\em ibid.}, Prop.~5.1.6), Theorem \ref{fapurefreddoquadentro} then becomes:

\begin{theorem}
\label{omonicesoir}
Let $(\cO, \mu, e)$ be an operad with multiplication and $\cN$ a cyclic unital opposite $\cO$-module. Assume we have Poincar\'e duality between $\cO$ and $\cN$ induced by a fundamental class $[\zeta]\in H_d(\cN)$. 
Then $HC^{-}_\bullet(\cN)$ carries a degree $(1-d)$ Lie bracket 
\[
[\cdot,\cdot]\colon HC^-_p(\cN)\otimes HC^-_q(\cN)\to HC^-_{p+q+1-d}(\cN)
\]
defined by
\[
[{z},{w}]=(-1)^{z+d}\beta((\pi {z})\smallsmile (\pi {w})),
\]
where $\pi: HC^-_p(\cN) \to H_p(\cN)$ and $\beta\colon H_p(\cN)\to HC^{-}_{p+1}(\cN)$ are the morphisms in the long exact sequence \rmref{exact-sequence-negative-cyclic}, and where
\[
\smallsmile\colon  H_p(\cN)\otimes H_q(\cN)\to H_{p+q-d}(\cN)
\]
is induced by $\smallsmile\colon H^{d-p}(\cO)\otimes HC^{d-q}(\cO)\to H^{d-(p+q-d)}(\cO)$ via the Poincar\'e duality $H^{d-p}(\cO)\cong H_p(\cN)$.
\end{theorem}

\begin{example}[Higher brackets on negative cyclic homology of $d$-Calabi-Yau algebras]
\label{calabibabi1}
In this paragraph, we indicate how the original construction in \cite{dTdVVdB:CYDANCH} is recovered. The definition of a $d$-Calabi-Yau algebra in \cite{Gin:CYA} was reformulated in \cite[p.~8]{dTdVVdB:CYDANCH} by stating that a $d$-Calabi-Yau algebra consisted of a couple $(A, \eta)$, where $A$ is a homologically smooth $\K$-algebra and $\eta \in H_d(A,A)$ is a {\em nondegenerate} element in the $d$-th Hochschild homology group (see {\em loc.~cit.} for all necessary definitions used here). In particular, Poincar\'e duality holds, {\em i.e.}, the map 
$$
H^i(A,A) \to H_{d-i}(A,A), \quad \gvf \mapsto \iota_\gvf \eta
$$
between Hochschild cohomology and homology groups is an isomorphism (see {\em op.~cit.}, Prop.~5.5). 
Given these data, and considering the endomorphism operad $\cO = \End_\ahha$ along with the cyclic opposite module $\cN_\ahha = A^{\otimes \bullet+1}$ as described in Exs.~\ref{schnee1} and \ref{schnee2}, our construction in Theorem \ref{fapurefreddoquadentro} combined with Remark \ref{mfgoverload} then reproduces the bracket of degree $(1-d)$ on negative cyclic homology $HC^-_\bullet(A,A)$ for the $d$-Calabi-Yau algebra obtained in {\em op.~cit.}, Thm.~10.2. In particular, the operators $\iota,\cL,\cS,\cT$ occurring here are those made explicit in Example \ref{algebras} above.
\end{example}

\begin{example}[Higher brackets on negative cyclic homology of twisted $d$-Calabi-Yau algebras]
\label{calabibabi2}
 One can easily generalise the results in \cite{dTdVVdB:CYDANCH} mentioned in 
 Example \ref{calabibabi1} to {\em twisted} Calabi-Yau algebras, that is, 
 those where the $A$-bimodule structure on $A$ itself is twisted by a {\em Nakayama} (or {\em modular}) automorphism $\gs$, that is, the algebra $A$ as an $\Ae$-module has a finitely generated projective resolution of finite length and there exists a $d \in \N$ such that one has a right $\Ae$-module isomorphism $\Ext^i_\Ae(A,\Ae) = A_\gs$ for $i = d$ and zero otherwise.
 To be able to apply our formalism to these objects, as in \cite[\S7.4]{KowKra:BVSOEAT} the Nakayama automorphism has to be semisimple (diagonalisable). Classical examples of twisted Calabi-Yau algebras are standard quantum groups \cite{BroZha:DCATHCFNHA}, Manin's quantum plane \cite{VdB:ARBHHACFGR}, and the Podle\'s quantum $2$-sphere \cite{Kra:OTHCOQHS}.
\end{example}

\begin{example}[Higher brackets on negative cyclic homology of bialgebroids]
The preceding Examples \ref{calabibabi1} \& \ref{calabibabi2} are actually particular cases of the 
notion of Poincar\'e duality for bialgebroids, that is, a pair $(U,A)$, where $A$ is a (usually noncommutative) associative algebra and $U$ is both an algebra over $\Ae = A \otimes \Aop$ and a coalgebra over $A$ subject to certain compatibility conditions. 
The necessary product structures that allow for Poincar\`e duality for this case were formulated in \cite[Thm.~1]{KowKra:DAPIACT}: in particular, as in the Calabi-Yau case, if $A$ as a $U$-module admits a finitely generated projective resolution of finite length and there exists a $d \geq 0$ such that $\Ext^k_U(A, U) = 0$ for all $k \neq d$ and $A \cong \Ext_U^d(A,U)$, then there exists a Palladio cocycle $\zeta$ with $[\zeta] \in \Tor^U_d(A,A)$ such that $\iota_{(\cdot)} \zeta = (\cdot) \smallfrown \zeta$ establishes an isomorphism $\Ext^{n}_U(M,U) \to \Tor^U_{n-d}(M,U)$ for a left $U$-module $M$ with $\Tor_k^A(M,A) = 0$ in positive degrees; 
see {\em op.~cit.} for all technical details. Not only the preceding Examples \ref{calabibabi1} \& \ref{calabibabi2} by setting $(U,A) := (\Ae, A)$ are included in this theory, but also the {\em duality Lie-Rinehart algebras} in \cite{Hue:DFLRAATMC} (geometrically, Lie algebroids that lead to the space of differential operators on a smooth manifold); in particular, this also comprises Poisson algebras. As described in detail in \cite[\S6.3]{Kow:GABVSOMOO}, 
the operad needed here is not precisely an endomorphism operad, but rather
$$
\cO := \Hom_\Aee(U^{\otimes_\Aopp \bullet}, P),
$$
where $P$ is a braided commutative Yetter-Drinfel'd algebra; in the simplest case, one can consider $P:=A$. If the bialgebroid is a (left) Hopf algebroid, 
one has sort of an antipode (in the simplest case: an ordinary Hopf algebra), which induces a cyclic structure on the opposite $\cO$-module
$$
\cN := (Q \otimes_\Aopp P) \otimes_\Aopp U^{\otimes_\Aopp \bullet},
$$
which therefore becomes a cyclic opposite $\cO$-module; here, $P$ is as above, whereas $Q$ is an {\em anti} Yetter-Drinfel'd module (see {\em op.~cit.} and references therein for all technical details). By the constructions described in \S\ref{sansonitedescoitalianoitalianotedesco}, the $\K$-module $\cN$ then carries the structure of a mixed (chain) complex, which again by Remark \ref{mfgoverload} allows to consider its negative cyclic homology $HC^-_\bull(\cN)$, and analogous to the preceding Example \ref{calabibabi1}, in case of Poincar\'e duality, one obtains a bracket of degree $(1-d)$ on $HC^-_\bull(\cN)$. 
\end{example}

\section{Brackets for modules over cyclic operads}
\label{plong}

\subsection{Modules over operads}

In this section, we will see that the notion of a module over a cyclic operad with multiplication provides a quite general class of examples that illustrates the constructions in \S\ref{sete}. To this end, we show how a  module $\cM$ over a cyclic operad with multiplication can be equipped with a calculus structure first. 
To fix the notation, recall the following definition from \cite[Def.~1.3]{Mar:MFO}:

\begin{dfn}
A {\em module} $\cM$ over an operad $\cO$ is a collection 
$\{\cM(q)\}_{q \geq 0}$ of $\K$-modules together with two linear composition maps, the {\em left composition}
$$
\circ_i = \circ^\ell_i: \cO(p) \otimes \cM(q) \to \cM(p+q-1), \quad i = 1, \ldots, p,
$$
as well as the {\em right composition}
$$
\circ_i = \circ^r_i: \cM(q) \otimes \cO(p) \to \cM(q+p-1), \quad i = 1, \ldots, q,
$$
subject to associativity properties analogous to those in the definition of an operad in \S\ref{pamukkale1}, that is, Eqs.~\rmref{danton} are required to hold, {\em mutatis mutandis}, for triples of elements in $\cO(p) \otimes \cO(q) \otimes \cM(r)$, $\cO(p) \otimes \cM(q) \otimes \cO(r)$, or $\cM(p) \otimes \cO(q) \otimes \cO(r)$.
\end{dfn}

As there is sort of an action from both sides, one might want to call such an object a {\em bimodule}; but we won't.
It is immediately clear that an operad is a module over itself with both 
left and right composition given by the partial compositions $\circ_i$. For better readability, in the sequel we will not distinguish in notation between 
$\circ^\ell_i$ and $ \circ^r_i$,
even for general $\cM$.

The necessary observation to be made here is that a module over an operad is obviously not the same as an opposite module and hence, even in case a cyclic operator $t$ is given, one cannot directly use the formulae given in \S\ref{pasticciottoleccese}, in particular Definition \ref{volterra}, to obtain the structure of a calculus.

More precisely, whereas a module over an operad $\cO$ is not an opposite $\cO$-module, its $\K$-linear dual is so: recall that we assume all $\cM(q)$ to be in $\kmod$, so it makes sense to define $\cM^*$ as a sequence 
$
\{\cM^*(q)\}_{q \geq 0},
$
where for each degree one has
$$
\cM^*(q) := \Hom_\K(\cM(q), \K),
$$
and we declare the natural pairing 
$$
\langle \cdot, \cdot \rangle: \cM^* \otimes \cM \to \K, \quad x \otimes m \mapsto \langle x, m \rangle
$$
to be zero whenever the degrees of elements in $\cM^*$ and $\cM$ do not match.

\begin{lem}
\label{forzafrosinone1}
If $\cM$ is an $\cO$-module, its linear dual $\cM^*$ is an opposite $\cO$-module in the sense of Definition \ref{molck} by setting
\begin{equation}
\label{mustazzoli1}
\langle \gvf \bullet_i x, m \rangle := \langle x, m \circ_i \gvf \rangle, \qquad \mbox{for} \ i = 1, \ldots, q,
\end{equation}
where $\gvf \in \cO(p)$, $m \in \cM(q)$, and $x \in \cM^*(p+q-1)$.
\end{lem}

\begin{proof}
That the required associative properties \rmref{TchlesischeStr} are fulfilled is a straightforward computation: for example, assume that $\gvf \in \cO(p)$ and $i \leq j < i +p$. Then applying the middle equation in \rmref{danton} to the defining equation \rmref{mustazzoli1}, one obtains 
$$
\langle \gvf \bullet_i (\psi \bullet_j x), m \rangle  = 
\langle x, (m \circ_i \gvf) \circ_j \psi  \rangle
=
\langle x, m \circ_i (\gvf \circ_{j-i+1} \psi) \rangle
=
\langle (\gvf \circ_{j-i+1} \psi) \bullet_{i}  x, m \rangle
$$
for $\psi \in \cO$, $x \in \cM^*(n)$, and $m \in \cM(n-p-q+2)$, which is the desired middle equation in  \rmref{TchlesischeStr}.
\end{proof}

\subsection{The calculus operations for modules over cyclic operads}

The definition of a {\em cyclic operad} $(\cO, \tau)$ with cyclic operator $\tau$ goes back to \cite[\S2]{GetKap:COACH}, whereas a {\em cyclic operad with multiplication} $(\cO, \mu, e, \tau)$ was introduced in \cite[Def.~3.11]{Men:BVAACCOHA}.
In such a situation, the following definition makes sense: 

\begin{dfn}
\label{benagol}
A {\em cyclic module over a cyclic operad}  $(\cO, \tau)$ is an $\cO$-module $\cM$ together with a degree-preserving linear map $\tau: \cM(q) \to \cM(q)$ for $q \in \N$ such that 
\begin{eqnarray*}
\tau(\gvf \circ_i m) &=& \tau\gvf \circ_{i-1} m, \qquad \mbox{for} \ p \geq 2, q \geq 0, 2 \geq i \geq p, \\
\tau(\gvf \circ_1 m) &=& \tau m \circ_q \tau\gvf, \qquad \mbox{for} \ p \geq 1, q \geq 1, \\ 
\tau^{q+1} &=& \id_{\cM(q)},
\end{eqnarray*}
for each $\gvf \in O(p)$ and $m \in \cM(q)$.
\end{dfn}

As is Lemma \ref{forzafrosinone1}, also the following assertion is easy to verify.

\begin{lem}
\label{forzafrosinone2}
Let $(\cM, \tau)$ be a cyclic module over the cyclic operad 
$(\cO, \tau)$. Defining then
\begin{equation}
\label{mustazzoli2}
\langle tx, n \rangle := \langle x , \tau m \rangle,
\end{equation}
where $m \in \cM$ and $x \in \cM^*$, 
along with
\begin{equation}
\label{mustazzoli3}
\langle \gvf \bullet_0 x, m \rangle := \langle x, \tau\gvf \circ_p m \rangle,
\end{equation}
where $\gvf \in \cO(p)$, $m \in \cM(q)$, and $x \in \cM^*(p+q-1)$,
turns $\cM^*$ together with the compositions \rmref{mustazzoli1} into a cyclic (unital) opposite $\cO$-module. 
If $(\cO, \mu, e, \tau)$ is a cyclic operad with multiplication, then $\cM^*$ by virtue of Eqs.~\rmref{colleoppio} is a cyclic $\K$-module.
\end{lem}

\begin{rem}
In case $\gvf \in \cO(0)$ is of degree zero, Eq.~\rmref{mustazzoli3} does not make sense and has to be read appropriately: using \rmref{cine40}, one has
\begin{equation}
\label{cerrate}
\langle \gvf \bullet_0 x, m \rangle = \langle t(\gvf \bullet_q x), m \rangle
\overset{\scriptscriptstyle{\rmref{mustazzoli2}}}{=}
  \langle \gvf \bullet_q x, \tau m \rangle
\overset{\scriptscriptstyle{\rmref{mustazzoli1}}}{=}
 \langle x, \tau m \circ_m \gvf \rangle
\end{equation}
 for $m \in \cM(q)$ and $x \in \cM^*(q-1)$.
\end{rem}

One then has automatically by the considerations in \S\ref{pasticciottoleccese}:

\begin{prop}
\label{baldda}
If $\cM$ is a cyclic module over a cyclic operad with multiplication, then there is the structure of a homotopy Cartan calculus on $\cM^*$ resp.\ $CC^\bullet_{\mathrm{per}}(\cM^*)$.
\end{prop}

On the other hand, at this point we are not really interested in the calculus structure on $\cM^*$ or $CC^\bullet_{\mathrm{per}}(\cM^*)$ but rather how this induces a homotopy Cartan calculus on  $\cM$ or $CC^\bullet_{\mathrm{per}}(\cM)$ via the canonical pairing: remember that one could {\em not} obtain the calculus maps for $\cM$ or $CC^\bullet_{\mathrm{per}}(\cM)$ directly from \S\ref{pasticciottoleccese} as $\cM$ is not an opposite $\cO$-module. One rather has:

\begin{theorem}
\label{nochkeinname}
Let $(\cM, \tau)$ be a cyclic module over the cyclic operad with multiplication $(\cO, \mu, e, \tau)$. Then the degree $1$ differential
\begin{eqnarray*}
\bc: \cM(q) &\to& \cM(q+1), 
\\
m &\mapsto&  \mu \circ_2 m + \sum^q_{i=1} (-1)^{i} m \circ_i \mu + (-1)^{q+1} \mu \circ_1 m,
\end{eqnarray*}
the degree $-1$ differential
\begin{eqnarray*}
\Bc: \cM(q) &\to& \cM(q-1),
\\
m &\mapsto& \sum^{q-1}_{i=0} (-1)^{(q-1)i} \tau^i(\tau m \circ_q e),
\end{eqnarray*}
along with the contraction
\begin{eqnarray}
\label{biennale}
\iota: \cO(p) \otimes \cM(q) &\to& \cM(p+q),
\nonumber
\\
\gvf \otimes m &\mapsto& 
(\mu \circ_2 m) \circ_1 \gvf, \phantom{ \sum^q_{i=1}}
\end{eqnarray}
the Lie derivative
\begin{eqnarray}
\cL: \cO(p) \otimes \cM(q) &\to& \cM(p+q-1),
\nonumber
\\
\gvf \otimes m &\mapsto& \sum^q_{i=1} (-1)^{(p-1)(i-1)} m \circ_i \gvf 
+ \sum^p_{i=1} (-1)^{(p-1)i + q(i-1)} \tau^i \gvf \circ_{p-i+1} m,
\nonumber
\end{eqnarray}
the degree $-2$ operator
\begin{eqnarray}
\cS: \cO(p) \otimes \cM(q) &\to& \cM(p+q-2),
\nonumber
\\
\gvf \otimes m &\mapsto&
\sum^{q-1}_{i=1} \sum^{i}_{j=1} (-1)^{p(i-j) + q(j-1) + i-1} (\tau^j m \circ_{q-j+1} e) \circ_{i-j+1} \gvf,
\nonumber
\end{eqnarray}
and the homotopy
\begin{eqnarray}
\cT: \cO(p) \otimes \cO(r) \otimes \cM(q) &\to& \cM(p+r+q-2),
\nonumber
\\
\gvf \otimes \psi \otimes m &\mapsto& 
\sum^{p-1}_{j=1} \sum^{p-1}_{i=j} (-1)^{q(j-1) + r(i-1) + (p-1)j +i } (\tau^j \gvf \circ_{p-j+1} m) \circ_{p-i} \psi
\nonumber
\end{eqnarray}
satisfy $[\Bc,\bc]=0$ as well as Equations \rmref{dellera1}, \rmref{morphism-of-dgas}, and
\begin{equation}
\label{radiofranceinternational}
\begin{cases}[\iota_\gvf, \cL_\psi] - \iota_{\{\gvf, \psi\}} = [\bc, \cT(\gvf,\psi)] - \cT(\gd \gvf, \psi) -(-1)^{\gvf-1} \cT(\gvf, \delta \psi) \\
[\cS_\gvf, \cL_\psi] - \cS_{\{\gvf,\psi\}} = [\Bc, \cT(\gvf,\psi)] 
\end{cases}
\end{equation}
on the normalised cochain complex $\overline{\cM}$, when $\gvf,\psi$ are in $\overline{\cO}$.
\end{theorem}

\begin{proof}
This follows from Proposition \ref{baldda} together with the fact that the given operators are so-to-say the {\em adjoint} operators of those on $\cM^*$. More precisely, for an $x \in \cM^*(p+q-1)$, one has
\begin{align*}
\langle x, \Bc m \rangle = \langle Bx, m \rangle, \quad
\langle i_\gvf x, m \rangle = \langle x, \iota_\gvf m \rangle, \quad
\langle L_\gvf x, m \rangle = \langle x, \cL_\gvf m \rangle, \\
\langle S_\gvf x, m \rangle = \langle x, \cS_\gvf m \rangle, \quad
\langle T(\gvf, \psi)(x), m \rangle = \langle x, \cT(\gvf, \psi)(m) \rangle,
\end{align*}
and that this is indeed the case is computed directly using \rmref{mustazzoli1}--\rmref{mustazzoli3}. We show this only for the operators $\iota$ and $\cS$ and leave the rest to the reader. For example, one has
\begin{equation*}
\label{petrel2}
\begin{split}
\langle \iota_\gvf x, m \rangle &\overset{\scriptscriptstyle{\rmref{alles4}}}{=} \langle (\mu \circ_2 \gvf) \bullet_0 x, m \rangle
\overset{\scriptscriptstyle{\rmref{mustazzoli3}}}{=}
\langle x, \tau(\mu \circ_2 \gvf) \circ_{p+1} m \rangle 
\\
&
=
\langle x, (\mu \circ_1 \gvf) \circ_{p+1} m \rangle
=
\langle x, (\mu \circ_2 m) \circ_1 \gvf \rangle = \langle x, \iota_\gvf m \rangle,  
\end{split}
\end{equation*}
for $\gvf \in \cO(p)$, $m \in \cM$, and $x \in \cM^*$. Likewise, for $\gvf \in \cO(p)$, $m \in \cM(q)$, and $x \in \cM^*(p+q-2)$
\begin{eqnarray*}
\langle \cS_\gvf x , m \rangle 
&\overset{\scriptscriptstyle{\rmref{lagrandebellezza1}, \rmref{capillareal1}}}{=}& 
\sum^{q-1}_{i=1} \sum^{i}_{j=1} (-1)^{p(i-j) + q(j-1) + i-1} \langle e \bullet_0 t^{j-1}(\gvf \bullet_{i-j+1} x), m \rangle
\\
&\overset{\scriptscriptstyle{\rmref{cerrate}}}{=}&
\sum^{q-1}_{i=1} \sum^{i}_{j=1} (-1)^{p(i-j) + q(j-1) + i-1} \langle t^{j-1}(\gvf \bullet_{i-j+1} x), \tau m \circ_q e \rangle
\\
&\overset{\scriptscriptstyle{\rmref{mustazzoli2}}}{=}&
\sum^{q-1}_{i=1} \sum^{i}_{j=1} (-1)^{p(i-j) + q(j-1) + i-1} \langle \gvf \bullet_{i-j+1} x , \tau^{j-1}(\tau m \circ_q e) \rangle
\\
&\overset{\scriptscriptstyle{\rmref{mustazzoli1}}}{=}&
\sum^{q-1}_{i=1} \sum^{i}_{j=1} (-1)^{p(i-j) + q(j-1) + i-1} \langle x, (\tau^j m \circ_{q-j+1} e) \circ_{i-j+1} \gvf \rangle \\
&=& \langle x, \cS_\gvf m \rangle,
\end{eqnarray*}
for $q \geq 2$ and zero otherwise.
\end{proof}

\begin{rem}
\label{fieldsagain}
Dually to Remark \ref{formsagain}, this can be seen as a generalisation of the classical calculus from differential geometry as in the second part of Example \ref{forms'n'fields} of ``fields acting on fields'' themselves.
\end{rem}

\subsection{Cyclic operads as cyclic modules over themselves}
\label{dresdenerstollen}

The most obvious example of a cyclic module over a cyclic operad is of course given by the operad itself: if $(\cO, \mu, e, \tau)$ is a cyclic operad with multiplication, then 
\begin{equation}\label{contraction-by-e-is-identity}
\iota_{(\cdot)} e: \cO \to \cO, \quad \gvf \mapsto \iota_\gvf e
\end{equation}
is the identity map, as is immediately seen from \rmref{biennale} along with $\mu \circ_2 e = \mathbb{1}$ and $\mathbb{1} \circ_1 \gvf = \gvf$; more general, we have
\begin{equation}
\label{tragisch1}
\iota_\gvf \psi = \gvf \smallsmile \psi 
\end{equation}
for $\gvf, \psi \in \cO$. 
Similarly, applying the operators in Theorem \ref{nochkeinname} to the unit $e$, one obtains 
\begin{equation}
\label{tragisch2}
\begin{array}{rcl}
\cL_\gvf e &=& \Bc \gvf, \\
\bc (e) = \cS_\gvf e &=& 0, \\
\cT(\gvf, \psi)(e) &=& (-1)^\psi \cS_\psi\gvf, 
\end{array}
\end{equation}
where the last one is seen by a substitution in the summation indices.
 Finally, note that in this case
$$
\bc = (-1)^{q-1} \gd.
$$
on an element of degree $q$.

\begin{theorem}
\label{nocheins}
For a cyclic operad with multiplication $(\cO, \mu, e, \tau)$, one has
\begin{equation}
\begin{split}
\label{nunjanaja1}
\{\psi,\gvf\} &= - \psi \smallsmile \Bc(\gvf) - (-1)^{(\gvf-1)\psi} \cL_\gvf \psi 
\\
&
\quad + (-1)^{\gvf} \delta(S_\psi\gvf) +  (-1)^{\psi } S_{\psi}\gd\gvf + (-1)^{\psi+\gvf-1} S_{\gd \psi} \gvf.
\end{split}
\end{equation}
Passing to the normalised complex, this is equivalent to
\begin{equation}
\label{nunjanaja2}
\begin{split}
\{\psi,\gvf\} &= - \psi \smallsmile \Bc(\gvf) + (-1)^{(\gvf-1)\psi} \Bc(\gvf \smallsmile \psi) - (-1)^{(\gvf-1)(\psi-1)} \gvf \smallsmile \Bc(\psi)  
\\
        &\quad  - (-1)^{(\gvf-1)\psi} \delta\big(S_\gvf \psi)  + (-1)^{(\gvf-1)\psi} S_{\gd \gvf} \psi  + (-1)^{(\gvf-1)(\psi-1)} S_{\gvf} \gd \psi
\\
&\quad
+ (-1)^{\gvf} \delta(S_\psi\gvf)  +  (-1)^\psi S_{\psi}\gd\gvf + (-1)^{\psi+\gvf-1} S_{\gd \psi} \gvf.
\end{split}
\end{equation}
\end{theorem}

\begin{proof}
Eq.~\rmref{nunjanaja1} is proven by applying the first equation in \rmref{radiofranceinternational} to the unit $e$, taking \rmref{tragisch1} and \rmref{tragisch2} into account. Eq.~\rmref{nunjanaja2} follows from \rmref{nunjanaja1} and \rmref{dellera-operad}. 
\end{proof}

\begin{rem}
Note that inserting \rmref{dellera-operad} again in \rmref{nunjanaja1}, one can also write for the bracket:
$$
\{\psi,\gvf \} = (-1)^\psi\bigl(\cL_\psi \gvf -(-1)^{\psi\gvf} \cL_\gvf\psi - \Bc(\psi \smallsmile \gvf)\bigr),
$$
which in a sense resembles the bracket for differential forms in Poisson geometry (see, for example, \cite[Eq.~(4.3)]{Vai:LOTGOPM} and the numerous references mentioned there).
\end{rem}

Theorem \ref{nocheins} allows to recover the following result first proven by Menichi in \cite[Thm.~1.4]{Men:BVAACCOHA}:

\begin{cor}
\label{erkundungen}
A cyclic operad with multiplication carries the structure of a (co)cyclic $\K$-module, and the cohomology $H^\bullet(\cO)$ of the underlying cosimplicial $\K$-module that of a Batalin-Vilkoviski\u\i\ algebra.
\end{cor}

\begin{proof}
In cohomology, all $\delta$-coboundaries in \rmref{nunjanaja2} disappear, and since the cup product is (graded) commutative in cohomology, that is 
$
\gvf \smallsmile \psi = (-1)^{pq} \psi \smallsmile \gvf,
$
 Eq.~\rmref{nunjanaja2} becomes the Batalin-Vilkoviski\u\i\ equation, see Eq.~\ref{avvisoagliutenti}.
\end{proof}

\subsection{The string topology bracket for operad modules}
Finally, we can put together Theorem \ref{nochkeinname} and Theorem \ref{chas-sullivan-menichi} to get the Chas-Sullivan-Menichi degree $-2$ Lie bracket on the negative cyclic cohomology of a cyclic module $\cM$ endowed with a degree $d$ fundamental class, {\em i.e}, an element  $[\zeta]\in H^d(\cM)$ such that $\iota_{(\cdot)}[\zeta]\colon H^p(\cO)\to H^{p+d}(\cM)$ is an isomorphism for any $p$.
Namely, as in Section \ref{pasticciottoleccese}, let $\mathfrak{g}^\bullet$  be defined by $\mathfrak{g}^k :=\overline{\cO}({k+1})$, with the differential graded Lie algebra structure induced by the Gerstenhaber bracket and differential. Also, let $(M^\bullet, \bc, \Bc)$ be the mixed (cochain) complex defined by $M^k :=\overline{\cM}(d+k+1)$ with
\[
\bc:=b\colon M^k=\overline{\cM}(d+k+1)\to \overline{\cM}(d+k+2)=M^{k+2}
\]
and
\[
\Bc:=B\colon M^k=\overline{\cM}(d+k+1)\to \overline{\cM}(d+k)=M^{k-1},
\]
given by the expressions in Theorem \ref{nochkeinname}.
Also define
\begin{small}
\begin{align*}
\iota \colon & \mathfrak{g}^h\otimes M^k=  \overline{\cO}(h+1)\otimes  \overline{\cM}(d+k+1) \to \overline{\cM}(d+h+k+2)=M^{h+k+1},
\\
 \cL  \colon &\mathfrak{g}^h \otimes M^k=\overline{\cO}(h+1)\otimes  \overline{\cM}(d+k+1)  \to \overline{\cM}(d+h+k+1)=M^{h+k},
\\
\cS  \colon &\mathfrak{g}^h \otimes M^k=\overline{\cO}(h+1)\otimes \overline{\cM}(d+k+1)  \to \overline{\cM}(d+h+k)=M^{h+k-1},
\\
\cT \colon &\mathfrak{g}^h \otimes \mathfrak{g}^j \otimes M^k =\overline{\cO}(h+1)\otimes \overline{\cO}(j+1)\otimes  \overline{\cM}(d+k+1) 
\to \overline{\cM}(d+h+j+k+1)=M^{h+j+k}
\end{align*}
\end{small}
as in Theorem \ref{nochkeinname}.
Then $(\mathfrak{g}^\bullet,\iota,\cL,\cS,\cT)$ is a homotopy Cartan-Gerstenhaber calculus on $CC^\bullet_{\mathrm{per}}(M)$, and as $H^{-2}(\g)=0$, the fundamental class $[\zeta]$ is always represented by a Palladio cocycle $\zeta\in M^{-1}=\overline{\cM}(d)$. Therefore, recalling again that $H^\bullet(\cM)\cong H^\bullet(\overline{\cM})$ and $HC_-^\bullet(\cM)\cong HC^\bullet_-(\overline{\cM})$, Theorem \ref{chas-sullivan-menichi} gives

\begin{theorem}
\label{chas-sullivan-menichi-operad}
Let  $(\cO, \mu, e)$ be a cyclic operad with multiplication and $\cM$ a cyclic $\cO$-module with fundamental class $[\zeta]\in H^d(\cM)$. 
Then $HC_{-}^\bullet(\cM)$ carries a degree $(-d-2)$ Lie bracket 
\[
[\cdot,\cdot]\colon HC_-^p(\cM)\otimes HC_-^q(\cM)\to HC_-^{p+q-d-2}(\cM)
\]
defined by
\[
[{x},{y}] := (-1)^{x+d}  j ((\beta{x})\smallsmile(\beta{y})),
\]
where $\pi: H^p(\cM) \to HC^p_-(\cM)$ and $\beta\colon HC^p_-(\cM)\to H^{p-1}(\cM)$ are the morphisms in the long exact sequence \rmref{laurent-long}, and where the cup product
\[
\smallsmile\colon  H^p(\cM)\otimes H^q(\cM)\to H^{p+q-d}(\cM)
\]
is induced by the cup product $\smallsmile\colon H^{p-d}(\cO)\otimes HC^{q-d}(\cO)\to H^{(p+q-d)-d}(\cO)$ via the isomorphism $H^{d-p}(\cO)\cong H^p(\cM)$ given by contraction with the fundamental class. 
Moreover, 
$$
\beta[\cdot, \cdot] = \{ \beta(\cdot), \beta(\cdot)\},
$$
where 
\[
\{ \cdot, \cdot\}\colon H^p(\cM)\otimes H^q(\cM)\to H^{p+q-d-1}(\cM)
\]
is the degree $(-d-1)$ Lie bracket on $H^\bullet(\cM)$ induced by the Gerstenhaber bracket $\{ \cdot, \cdot\}\colon H^{p-d}(\cO)\otimes H^{q-d}(\cO)\to H^{(p+q-d-1)-d}(\cO)$ again via $H^{d-p}(\cO)\cong H^p(\cM)$.
\end{theorem}

\begin{example}[The Menichi bracket for cyclic operads]
If $(\cO, \mu, e, \tau)$ is a cyclic operad with multiplication, then as in \S\ref{dresdenerstollen} we can take as a cyclic module the operad itself. This has a degree zero fundamental class represented by the unit element $e\in \cO(0)$, see Equations \rmref{contraction-by-e-is-identity} \& \rmref{tragisch2}. Then one gets a degree $-2$ Lie bracket on the negative cyclic cohomology of $\cO$, and the connecting homomorphism $\beta\colon HC^p_-(\cO)\to H^{p-1}(\cO)$ defines a morphism of shifted Lie algebras between  the negative cyclic cohomology of $\cO$ and the cohomology of $\cO$ (with the Gerstenhaber bracket). This degree $-2$ Lie bracket on $HC^\bullet_-(\cO)$ has been introduced by Menichi in \cite[Cor.~1.5]{Men:BVAACCOHA}, generalising the string topology bracket of Chas-Sullivan \cite[Thm.~6.1]{ChaSul:ST}.
\end{example}

\appendix

\section{Basic operadic definitions}
\label{pamukkale}

\subsection{Operads and Gerstenhaber algebras}
\label{pamukkale1}
Recall that a {\em non-$\gS$ operad} $\cO$ in the category 
of $\K$-modules is a sequence $\{\cO(n)\}_{n \geq 0}$ of $\K$-modules 
with $\K$-bilinear operations $\circ_i: \cO(p) \otimes \cO(q) \to \cO({p+q-1})$, $i = 1, \ldots, p$,
subject to the compatibility relations: 
\begin{eqnarray}
\label{danton}
\nonumber
\gvf \circ_i \psi &=& 0 \qquad \qquad \qquad \qquad \qquad \! \mbox{if} \ p < i \quad \mbox{or} \quad p = 0, \\
(\varphi \circ_i \psi) \circ_j \chi &=& 
\begin{cases}
(\varphi \circ_j \chi) \circ_{i+r-1} \psi \qquad \mbox{if} \  \, j < i, \\
\varphi \circ_i (\psi \circ_{j-i +1} \chi) \qquad \hspace*{1pt} \mbox{if} \ \, i \leq j < q + i, \\
(\varphi \circ_{j-q+1} \chi) \circ_{i} \psi \qquad \mbox{if} \ \, j \geq q + i.
\end{cases}
\end{eqnarray}

See, for example, \cite[Def.~1.1]{Mar:MFO}. We call an operad {\em unital} if there is an {\em identity} $\mathbb{1} \in \cO(1)$ such that 
$
\gvf \circ_i \mathbb{1} = \mathbb{1} \circ_1 \gvf = \gvf
$ 
for all $\gvf \in \cO(p)$ and $i \leq p$, and we say that the operad is {\em with multiplication} if there exists a {\em multiplication}  $\mu \in \cO(2)$ along with a {\em unit} $e \in \cO(0)$ such that $\mu \circ_1 \mu = \mu \circ_2 \mu$ and 
$\mu \circ_1 e = \mu \circ_2 e = \mathbb{1}$. We denote such an object by the triple $(\cO, \mu, e)$.

To an operad with multiplication  $(\cO, \mu, e)$ one naturally associates the structure of a cosimplicial $\K$-module \cite{McCSmi:ASODHCC} with faces and degeneracies $\gvf \in \cO(p)$ given by $\gd_0 \gvf := \mu \circ_1 \gvf$, $\gd_i \gvf := \gvf \circ_{p-i+1} \mu$ for $i = 1, \ldots, p$, and $\gd_{p+1} \gvf := \mu \circ_2 \gvf$, along with $\sigma_j(\gvf) := \gvf \circ_{p-j} e$ for $j = 0, \ldots, p-1$; observe the opposite convention we adopt here. Hence (by the Dold-Kan correspondence), one obtains a cochain complex,
which we will denote by the same symbol $\cO$, with $\cO(n)$ in degree $n$ and 
 with differential $\gd: \cO(n) \to \cO({n+1})$ given by $\gd := \sum^{n+1}_{i=0} (-1)^i \gd_i$, and cohomology defined by 
$
H^\bullet(\cO) := H(\cO, \gd).
$ 

Moreover,  for $\gvf \in \cO(p)$ and $\psi \in \cO(q)$ one defines the cup product to be $\psi \smallsmile \gvf := (\mu \circ_2 \psi) \circ_1 \gvf \in \cO(p+q)$ and then the triple $(\cO, \smallsmile, \gd)$ forms a dg algebra. Defining on top the bracket ${\{} \varphi,\psi \}
:= \varphi\{\psi\} - (-1)^{(p-1)(q-1)} \psi\{\varphi\}$, where $\varphi\{\psi\} := \sum^{p}_{i=1}
        (-1)^{(q-1)(i-1)} \varphi \circ_i \psi  \in \cO({p+q-1})$ is the {\em brace} \cite{Ger:TCSOAAR, Kad:AIASICATRHT, Get:BVAATDTFT}.

Descending to cohomology, it is a straightforward check that the triple $(H^\bullet(\cO), \smallsmile, \{\cdot, \cdot\})$ forms a Gerstenhaber algebra \cite{Ger:TCSOAAR, GerSch:ABQGAAD, McCSmi:ASODHCC}. Notice that one has $\{\mu,\mu\} = 0$ as well as $\gd = \{\mu, \cdot \}$.

\subsection{Cyclic operads and Batalin-Vilkoviski\u\i\ algebras}
\label{pamukkale2}
On the other hand, a {\em cyclic operad} \cite{GetKap:COACH} is an operad $\cO$ in the sense of the preceding paragraph plus a degree-preserving linear map $\tau: \cO(p) \to \cO(p)$ for $p \in \N$ such that 
\begin{eqnarray*}
\tau(\gvf \circ_i \psi) &=& \tau\gvf \circ_{i-1} \psi, \qquad \mbox{for} \ p \geq 2, q \geq 0, 2 \geq i \geq p, \\
\tau(\gvf \circ_1 \psi) &=& \tau \psi \circ_q \tau\gvf, \qquad \mbox{for} \ p \geq 1, q \geq 1, \\ 
\tau^{p+1} &=& \id_{\cO(p)},
\end{eqnarray*}
for each $\gvf \in O(p)$ and $\psi \in \cO(q)$, and we denote this situation by $(\cO, \tau)$.

A {\em cyclic operad with multiplication} \cite{Men:BVAACCOHA}
is simultaneously an operad with multiplication and a cyclic operad, such that 
$
\tau \mu = \mu,
$
and we will denote such an object by $(\cO, \mu, e, \tau)$. Descending to cohomology, it is a not-so-straightforward check that the quadruple $(H^\bullet(\cO), \smallsmile, \{\cdot, \cdot\}, B)$ forms a Batalin-Vilkoviski\u\i\ algebra 
\cite{Men:BVAACCOHA}, see also Corollary \ref{erkundungen}.

\section{Proof of Theorem \ref{nightmare}}
\label{shock&awe}

This appendix contains the computations by which Theorem \ref{nightmare} is proven and which has been moved here so as not to impede the reading flow of the main text. 

\begin{proof}[Proof of Theorem \ref{nightmare}]

Assume without loss of generality that 
$p + q \leq n+1$. 
By direct computation using \rmref{messagedenoelauxenfantsdefrance2}, \rmref{alles4}, and the identities in \rmref{SchlesischeStr}, one obtains (after a while):
\begin{footnotesize}
\begin{equation*}
\label{nichschonwieder}
\begin{split}
[i_\psi, L_\gvf](x)  
&= 
\sum^{q}_{i=1} (-1)^{(p-1)(i-1)} (\mu \circ_2 \psi) \bullet_0(\gvf  \bullet_i x)
\\
&
\qquad 
- \sum^{p}_{i=1} (-1)^{(n-q)(i-1) + (p-1)q} (\gvf \circ_i (\mu \circ_2 \psi)) \bullet_0 t^{i-1}(x) \\
& \qquad + \sum^{p}_{i=1} (-1)^{n(i-1) + p-1} ((\mu \circ_2 \psi) \circ_1 \gvf)) \bullet_0 t^{i-1}(x) 
\\
&= i_{\psi\{\gvf\}} x
- \sum^{p}_{i=1} (-1)^{(n-q)(i-1) + (p-1)q} (\gvf \circ_i (\mu \circ_2 \psi)) \bullet_0 t^{i-1}(x) \\
& \qquad + \sum^{p}_{i=1} (-1)^{n(i-1) + p-1} ((\mu \circ_2 \psi) \circ_1 \gvf)) \bullet_0 t^{i-1}(x).
\end{split}
\end{equation*}
\end{footnotesize}
Hence, to prove \rmref{panem2}, we are left to show that 
\begin{footnotesize}
\begin{equation}
\label{incubo}
\begin{split}
&
\sum^{p}_{i=1} (-1)^{n(i-1) + p-1} ((\mu \circ_2 \psi) \circ_1 \gvf)) \bullet_0 t^{i-1}(x) 
\\
& - 
\sum^{p}_{i=1} (-1)^{(n-q)(i-1) + (p-1)q} (\gvf \circ_i (\mu \circ_2 \psi)) \bullet_0 t^{i-1}(x) 
\\
&
 = - (-1)^{(p-1)(q-1)} i_{\gvf\{\psi\}}x  
+ [b, T(\gvf,\psi)] - T(\gd \gvf, \psi) - (-1)^{p-1} T(\gvf, \delta \psi). 
\end{split}
\end{equation}
\end{footnotesize}
We proceed by explicitly writing down all terms and compare them one by one, making heavy use of \rmref{SchlesischeStr}--\rmref{colleoppio}, the ``associativity'' properties \rmref{danton} of the partial compositions $\circ_i$ of the operad, as well as enforced double (or triple) sum yoga plus {\em numerous} summation variable substitutions. 
To this end, we number the terms so that each number corresponds to a (single, double, or triple) sum, with all signs. For example, write the left hand side in \rmref{incubo} as
\begin{footnotesize}
\begin{equation*}
\begin{split}
&
\sum^{p}_{i=1} (-1)^{n(i-1) + p-1} ((\mu \circ_2 \psi) \circ_1 \gvf)) \bullet_0 t^{i-1}(x) 
\\
&
- 
\sum^{p}_{i=1} (-1)^{(n-q)(i-1) + (p-1)q} (\gvf \circ_i (\mu \circ_2 \psi)) \bullet_0 t^{i-1}(x) 
\\
&
=: (1) + (2),
\end{split}
\end{equation*}
\end{footnotesize}
whereas the right hand side in \rmref{incubo} becomes 
\begin{footnotesize}
\begin{equation*}
\begin{split}
&
- (-1)^{(p-1)(q-1)} i_{\gvf\{\psi\}}x  
+ [b, T(\gvf,\psi)] - T(\gd \gvf, \psi) - (-1)^{p-1} T(\gvf, \delta \psi). 
\\
&
= 
- (-1)^{(p-1)(q-1)} \sum^p_{i=1} (-1)^{(q-1)(i-1)} (\mu \circ_2 (\gvf \circ_i \psi)) \bullet_0 x  
\\
&
\quad +  
\sum^{n-p-q+1}_{k=1} (-1)^k \sum^{p-1}_{j=1} \sum^{p-1}_{i=j} (-1)^{n(j-1) + (q-1)(i-j) + p} \mu \bullet_k \big((\gvf \circ_{p-i+j} \psi) \bullet_0 t^{j-1} x\big) 
\\
&
\quad 
+
  \sum^{p-1}_{j=1} \sum^{p-1}_{i=j} (-1)^{n(j-1) + (q-1)(i-j) + p} (\mu \circ_1 (\gvf \circ_{p-i+j} \psi)) \bullet_0 t^{j-1} x 
\\
&
\quad 
+
(-1)^{n-p-q}   \sum^{p-1}_{j=1} \sum^{p-1}_{i=j} (-1)^{n(j-1) + (q-1)(i-j) + p} (\mu \circ_2 (\gvf \circ_{p-i+j} \psi)) \bullet_0 t^{j} x 
\\
&
\quad 
- (-1)^{p+q}
  \sum^{p-1}_{j=1} \sum^{p-1}_{i=j} (-1)^{(n-1)(j-1) + (q-1)(i-j) + p} (\gvf \circ_{p-i+j} \psi) \bullet_0 t^{j-1} (\mu \bullet_0 x)  
\\
&
\quad 
- (-1)^{p+q+n}
  \sum^{p-1}_{j=1} \sum^{p-1}_{i=j} (-1)^{(n-1)(j-1) + (q-1)(i-j) + p} (\gvf \circ_{p-i+j} \psi) \bullet_0 t^{j-1} (\mu \bullet_0 t x)
\\
&
\quad 
- (-1)^{p+q}
 \sum^{n-1}_{k=1} (-1)^k \sum^{p-1}_{j=1} \sum^{p-1}_{i=j} (-1)^{(n-1)(j-1) + (q-1)(i-j) + p} (\gvf \circ_{p-i+j} \psi) \bullet_0 t^{j-1} (\mu \bullet_k x)  
\\
\end{split}
\end{equation*}
\end{footnotesize}

\begin{footnotesize}
\begin{equation*}
\begin{split}
&
\quad
-
   \sum^{p}_{j=1} \sum^{p}_{i=j} (-1)^{n(j-1) + (q-1)(i-j) + p+1} ((\mu \circ_1 \gvf) \circ_{p+1-i+j} \psi) \bullet_0 t^{j-1} x 
\\
&
\quad 
- (-1)^{p+1}
   \sum^{p}_{j=1} \sum^{p}_{i=j} (-1)^{n(j-1) + (q-1)(i-j) + p+1} ((\mu \circ_2 \gvf) \circ_{p+1-i+j} \psi) \bullet_0 t^{j-1} x 
\\
&
\quad 
- 
\sum^{p}_{k=1} (-1)^{p-k+1}  \sum^{p}_{j=1} \sum^{p}_{i=j} (-1)^{n(j-1) + (q-1)(i-j) + p+1} ((\gvf \circ_k \mu) \circ_{p+1-i+j} \psi) \bullet_0 t^{j-1} x 
\\
&
- (-1)^p
\sum^{p-1}_{j=1} \sum^{p-1}_{i=j} (-1)^{n(j-1) + q(i-j) + p} (\gvf \circ_{p-i+j} (\mu \circ_1 \psi)) \bullet_0 t^{j-1} x 
\\
&
\quad 
- (-1)^{p+q+1}
\sum^{p-1}_{j=1} \sum^{p-1}_{i=j} (-1)^{n(j-1) + q(i-j) + p} (\gvf \circ_{p-i+j} (\mu \circ_2 \psi)) \bullet_0 t^{j-1} x 
\\
&
\quad 
- (-1)^p
\sum^{q}_{k=1} (-1)^{q-k+1}  \sum^{p-1}_{j=1} \sum^{p-1}_{i=j} (-1)^{n(j-1) + q(i-j) + p-1} (\gvf \circ_{p-i+j} (\psi \circ_k \mu)) \bullet_0 t^{j-1} x 
\\
&
=: (3) + (4) + (5)  + (6)  + (7)  + (8)  + (9)  + (10)  + (11)  + (12)  + (13)  + (14) + (15).
\end{split}
\end{equation*}
\end{footnotesize}
One now sees that 
\begin{footnotesize}
\begin{equation*}
\begin{split}
(10) 
&
=
 -
   \sum^{p-1}_{j=1} \sum^{p}_{i=j+1} (-1)^{n(j-1) + (q-1)(i-j) + p+1} ((\mu \circ_1 \gvf) \circ_{p+1-i+j} \psi) \bullet_0 t^{j-1} x  
\\
&
\quad 
-
   \sum^{p}_{j=1}  (-1)^{n(j-1) + p} ((\mu \circ_1 \gvf) \circ_{p+1} \psi) \bullet_0 t^{j-1} x 
\\
&
= 
-(5) + (1),
\end{split}
\end{equation*}
\end{footnotesize}
whereas
\begin{footnotesize}
\begin{equation*}
\begin{split}
(11) 
&
=
 (-1)^{p}
   \sum^{p}_{j=2} \sum^{p}_{i=j} (-1)^{n(j-1) + (q-1)(i-j) + p+1} (\mu \circ_2 (\gvf \circ_{p-i+j} \psi)) \bullet_0 t^{j-1} x  
\\
&
\qquad 
+  (-1)^{p}
   \sum^{p}_{i=1} (-1)^{(q-1)(i-1) +p} (\mu \circ_2 (\gvf \circ_{p-i+1} \psi)) \bullet_0 x 
\\
&
=  
-(6) - (3).
\end{split}
\end{equation*}
\end{footnotesize}
Furthermore, 
\begin{footnotesize}
\begin{equation*}
\begin{split}
-(2) + (7) + (12) &= 
-(2) + (7) 
- 
\sum^{p}_{k=1} (-1)^{p-k+1}  \sum^{p}_{j=1} (-1)^{n(j-1) + (q-1)j + qp +1} ((\gvf \circ_k \mu) \circ_{j+1} \psi) \bullet_0 t^{j-1} x
\\
&
\quad 
- 
\sum^{p}_{k=1} (-1)^{p-k+1}  \sum^{p-1}_{j=1} \sum^{p-1}_{i=j} (-1)^{n(j-1) + (q-1)(i-j) + p+1} ((\gvf \circ_k \mu) \circ_{p+1-i+j} \psi) \bullet_0 t^{j-1} x
\\
&
=
(7) 
- 
\sum^{p}_{k=2} (-1)^{p-k+1}  \sum^{k-1}_{j=1} (-1)^{n(j-1) + (q-1)j + qp+1} ((\gvf \circ_k \mu) \circ_{j+1} \psi) \bullet_0 t^{j-1} x
\\
&
\quad 
- 
\sum^{p-1}_{k=1} (-1)^{p-k+1}  \sum^{p}_{j=k+1} (-1)^{n(j-1) + (q-1)j + qp+1} ((\gvf \circ_k \mu) \circ_{j+1} \psi) \bullet_0 t^{j-1} x
\\
&
\quad 
- 
\sum^{p-1}_{j=1} \sum^{p-1}_{i=j} (-1)^{n(j-1) + (q-1)(i-j) + p+j+1} ((\gvf \circ_j \mu) \circ_{p+1-i+j} \psi) \bullet_0 t^{j-1} x
\\
&
\quad 
- 
\sum^{p-1}_{j=2} \sum^{p-1}_{i=j} \sum^{j-1}_{k=1} (-1)^{n(j-1) + (q-1)(i-j) + k} ((\gvf \circ_k \mu) \circ_{p+1-i+j} \psi) \bullet_0 t^{j-1} x
\\
&
\quad 
- 
\sum^{p-1}_{j=1} \sum^{p-1}_{i=j} \sum^{p}_{k=j+1} (-1)^{n(j-1) + (q-1)(i-j) + k} ((\gvf \circ_k \mu) \circ_{p+1-i+j} \psi) \bullet_0 t^{j-1} x
\\
&
=
- 
\sum^{p}_{k=2} \sum^{k-1}_{j=1} (-1)^{n(j-1) + (q-1)j + qp+k + p+1} ((\gvf \circ_k \mu) \circ_{j+1} \psi) \bullet_0 t^{j-1} x
\\
&
\quad 
- 
\sum^{p-1}_{k=1} \sum^{p}_{j=k+1} (-1)^{n(j-1) + (q-1)j + qp+k+p+1} ((\gvf \circ_k \mu) \circ_{j+1} \psi) \bullet_0 t^{j-1} x
\\
&
\quad 
- 
\sum^{p-1}_{j=2} \sum^{p-1}_{i=j} \sum^{j-1}_{k=1} (-1)^{n(j-1) + (q-1)j + qp+k+p+1} ((\gvf \circ_k \mu) \circ_{p+1-i+j} \psi) \bullet_0 t^{j-1} x
\\
&
\quad 
- 
\sum^{p-1}_{j=1} \sum^{p-1}_{i=j} \sum^{p}_{k=j+1} (-1)^{n(j-1) + (q-1)j + qp+k+p+1} ((\gvf \circ_k \mu) \circ_{p+1-i+j} \psi) \bullet_0 t^{j-1} x
\\
&
=: (16) + (17) + (18) + (19).
\end{split}
\end{equation*}
\end{footnotesize}
On the other hand, observe that 
\begin{footnotesize}
\begin{equation*}
\begin{split}
(4) &
= 
\sum^{n-p-q+1}_{k=1} (-1)^k \sum^{p-1}_{j=1} \sum^{p-1}_{i=j} (-1)^{n(j-1)+(q-1)(i-j)+p} (\gvf \circ_{p-i+j} \psi) \bullet_0 (\mu \bullet_{k+p+q-2} t^{j-1} x) 
\\
&
=
\sum^{n-1}_{k=p+q-1} (-1)^{p+q-k} \sum^{p-1}_{j=1} \sum^{p-1}_{i=j} (-1)^{n(j-1)+(q-1)(i-j)+p} (\gvf \circ_{p-i+j} \psi) \bullet_0 (\mu \bullet_k t^{j-1} x).
\end{split}
\end{equation*}
\end{footnotesize}
One then has
\begin{footnotesize}
\begin{equation*}
\begin{split}
(4) + (9) &
= (4) 
 - (-1)^{p+q}
 \sum^{p-1}_{j=1} \sum^{p-1}_{i=j}  \sum^{n-j}_{k=1} (-1)^{n(j-1)+(q-1)(i-j)+p+k} (\gvf \circ_{p-i+j} \psi) \bullet_0 (\mu \bullet_{k+j-1} t^{j-1} x) 
\\
&
\qquad
- 
(-1)^{p+q}
\sum^{p-1}_{j=1} \sum^{p-1}_{i=j}  \sum^{n-1}_{k=n-j+1} (-1)^{n(j-1)+(q-1)(i-j)+p+k} (\gvf \circ_{p-i+j} \psi) \bullet_0 t^{j-1+k} (\mu \bullet_0 t^{n+1-k} x) 
\\
&
=
 - (-1)^{p+q}
 \sum^{p-1}_{j=1} \sum^{p-1}_{i=j}  \sum^{p+q-2}_{k=j} (-1)^{n(j-1)+(q-1)(i-j)+p+k} (\gvf \circ_{p-i+j} \psi) \bullet_0 (\mu \bullet_{k} t^{j-1} x) 
\\
&
\qquad
- 
(-1)^{p+q}
\sum^{p-1}_{j=2} \sum^{p-1}_{i=j}  \sum^{j-2}_{k=0} (-1)^{n(j-1)+(q-1)(i-j)+p+k} ((\gvf \circ_{p-i+j} \psi) \circ_{k+1} \mu) \bullet_0 t^j x 
\\
&
=: (20) + (21).
\end{split}
\end{equation*}
\end{footnotesize}
At this point, we still have to deal with the terms $(8)$ and $(13)$--$(21)$.
We continue by
\begin{footnotesize}
\begin{equation*}
\begin{split}
(14) &
+ (15) + (19) + (20) 
\\
&
=
(14) + (15)+ (19)
 - (-1)^{p+q}
 \sum^{p-1}_{j=1} \sum^{p-1}_{i=j}  \sum^{p+q-2}_{k=j} (-1)^{n(j-1)+(q-1)(i-j)+p+k} ((\gvf \circ_{p-i+j} \psi) \circ_{k+1} \mu) \bullet_0 t^{j-1} x
\\
&
=
(14) + (15) + (19)
 - (-1)^{p+q}
 \sum^{p-1}_{j=1} \sum^{p-1}_{i=j}  \sum^{p-i+j-2}_{k=j} (-1)^{n(j-1)+(q-1)(i-j)+p+k} ((\gvf \circ_{k+1} \mu) \circ_{p-i+j+1} \psi) \bullet_0 t^{j-1} x
 \\
&
\qquad
 - (-1)^{p+q}
 \sum^{p-1}_{j=1} \sum^{p-1}_{i=j}  \sum^{p+q-2-i+j}_{k=p-i+j-1} (-1)^{n(j-1)+(q-1)(i-j)+p+k} (\gvf \circ_{p-i+j} (\psi \circ_{k+2-p+i-j} \mu)) \bullet_0 t^{j-1} x
 \\
&
\qquad
 - (-1)^{p+q}
 \sum^{p-1}_{j=1} \sum^{p-1}_{i=j}  \sum^{p+q-2}_{k=p+q-i+j+1} (-1)^{n(j-1)+(q-1)(i-j)+p+k} ((\gvf \circ_{k+2-q} \mu) \circ_{p-i+j} \psi) \bullet_0 t^{j-1} x
 \\
&
= (14) - (-1)^{p+q}
 \sum^{p-1}_{j=1} \sum^{p-1}_{i=j}  \sum^{p-1}_{k=p-i+j-1} (-1)^{n(j-1)+(q-1)(i-j)+p+k} ((\gvf \circ_{k+1} \mu) \circ_{p-i+j+1} \psi) \bullet_0 t^{j-1} x
 \\
&
\qquad
 - (-1)^{p+q}
 \sum^{p-1}_{j=1} \sum^{p-1}_{i=j}  \sum^{p+q-2}_{k=p+q-i+j+1} (-1)^{n(j-1)+(q-1)(i-j)+p+k} ((\gvf \circ_{k+2-q} \mu) \circ_{p-i+j} \psi) \bullet_0 t^{j-1} x
 \\
&
= - (-1)^{p+q}
 \sum^{p-1}_{j=1} \sum^{p-1}_{i=j}  \sum^{p-1}_{k=p-i+j} (-1)^{n(j-1)+(q-1)(i-j)+p+k} ((\gvf \circ_{k+1} \mu) \circ_{p-i+j+1} \psi) \bullet_0 t^{j-1} x
 \\
&
\qquad
 - (-1)^{p+q}
 \sum^{p-1}_{j=1} \sum^{p-1}_{i=j}  \sum^{p+q-2}_{k=p+q-i+j+1} (-1)^{n(j-1)+(q-1)(i-j)+p+k} ((\gvf \circ_{k+2-q} \mu) \circ_{p-i+j} \psi) \bullet_0 t^{j-1} x
 \\
&
= 
- (-1)^{p+q}
 \sum^{p-2}_{j=1} \sum^{p-2}_{i=j}  \sum^{p-2}_{k=i+1} (-1)^{n(j+p-1)+(q-1)(i-j)+k} ((\gvf \circ_{k+2} \mu) \circ_{i+2} \psi) \bullet_0 t^{j-1} x
\\
&
\qquad
- (-1)^{p+q}
 \sum^{p-2}_{j=1} \sum^{p-2}_{i=j}  (-1)^{n(j+p-1)+(q-1)(i-j)+p+i+1} ((\gvf \circ_{i+2} \mu) \circ_{i+2} \psi) \bullet_0 t^{j-1} x
\\
&
\qquad
 - (-1)^{p+q}
 \sum^{p-2}_{j=1} \sum^{p-1}_{i=j+1}  \sum^{p-1}_{k=i+1} (-1)^{nj+(q-1)(i-j+1)+p+k} ((\gvf \circ_{k+1} \mu) \circ_{i+1} \psi) \bullet_0 t^{j-1} x
\\
&
\qquad
- (-1)^{p+q}
 \sum^{p-2}_{j=1} \sum^{p-1}_{k=j+1} (-1)^{nj+(q-1)(j-1)+p+j} ((\gvf \circ_{k+1} \mu) \circ_{j+1} \psi) \bullet_0 t^{j-1} x
\\
&
=: (22) + (23).
\end{split}
\end{equation*}
\end{footnotesize}
Now we are left with terms $(8)$, $(13)$, $(16)$--$(18)$, as well as $(21)$--$(23)$. Consider first
\begin{footnotesize}
\begin{equation*}
\begin{split}
(8) &
= 
- (-1)^{p+q+n}
  \sum^{p-1}_{j=1} \sum^{p-1}_{i=j} (-1)^{n(j-1)+(q-1)(i-j)+p} (\gvf \circ_{p-i+j} \psi) \bullet_0 t^{j-1} (\mu \bullet_0 t x)
\\
&
= 
- (-1)^{p+q+n}
  \sum^{p-1}_{j=1} \sum^{p-1}_{i=j} (-1)^{n(j-1)+(q-1)(i-j)+p} ((\gvf \circ_{p-i+j} \psi) \circ_j \mu) \bullet_0 t^j x
\\
&
= 
- (-1)^{p+q+n}
  \sum^{p-1}_{j=1} \sum^{p-1}_{i=j} (-1)^{n(j-1)+(q-1)(i-j)+p} ((\gvf \circ_j \mu) \circ_{p-i+j+1} \psi) \bullet_0 t^j x,
\end{split}
\end{equation*}
\end{footnotesize}
along with
\begin{footnotesize}
\begin{equation*}
\begin{split}
(21) &
= 
- 
(-1)^{p+q}
\sum^{p-1}_{j=2} \sum^{p-1}_{i=j}  \sum^{j-2}_{k=0} (-1)^{n(j-1)+(q-1)(i-j)+p+k} ((\gvf \circ_{p-i+j} \psi) \circ_{k+1} \mu) \bullet_0 t^j x 
\\
&=
- (-1)^{p+q}
\sum^{p-1}_{j=2} \sum^{p-1}_{i=j}  \sum^{j-2}_{k=0} (-1)^{n(j-1)+(q-1)(i-j)+p+k} ((\gvf \circ_{k+1} \mu) \circ_{p-i+j+1} \psi) \bullet_0 t^j x 
\end{split}
\end{equation*}
\end{footnotesize}
Hence, 
\begin{footnotesize}
\begin{equation*}
\begin{split}
(8) + (17)
&
= 
- (-1)^{p+q+n}
  \sum^{p-1}_{k=2} \sum^{k-1}_{j=1} (-1)^{n(j-1)+qj+p+k} ((\gvf \circ_j \mu) \circ_{k+2} \psi) \bullet_0 t^j x 
\\
&
\qquad
- (-1)^{p+q+n}
  \sum^{p-1}_{k=1} (-1)^{(q-1)k+p} ((\gvf \circ_k \mu) \circ_{k+2} \psi) \bullet_0 t^k x 
\\
&
\qquad
- 
\sum^{p-2}_{k=1} \sum^{p-1}_{j=k+1} (-1)^{k + n(j-1)+qj+p} ((\gvf \circ_k \mu) \circ_{j+2} \psi) \bullet_0 t^{j} x
\\
&
\qquad
- 
\sum^{p-1}_{k=1} (-1)^{k + qk+1} ((\gvf \circ_k \mu) \circ_{k+2} \psi) \bullet_0 t^k x
\\
&
= 
- (-1)^{p+q+n}
  \sum^{p-1}_{k=2} \sum^{k-1}_{j=1} (-1)^{n(j-1)+qj+p+k} ((\gvf \circ_j \mu) \circ_{k+2} \psi) \bullet_0 t^j x 
\\
&
\qquad
- 
\sum^{p-2}_{k=1} \sum^{p-1}_{j=k+1} (-1)^{k + n(j-1)+qj+p} ((\gvf \circ_k \mu) \circ_{j+2} \psi) \bullet_0 t^{j} x.
\end{split}
\end{equation*}
\end{footnotesize}
Likewise,
\begin{footnotesize}
\begin{equation*}
\begin{split}
(18) + (21) &= 
\sum^{p-2}_{j=1} \sum^{p-1}_{i=j+1} (-1)^{n(j-1) + q(j-i) +p+1 + j} ((\gvf \circ_j \mu) \circ_{i+2} \psi) \bullet_0 t^{j} x
\\
&
\qquad
- 
\sum^{p-2}_{j=2} \sum^{p-1}_{i=j+1} \sum^{j-1}_{k=1} (-1)^{n(j-1) + q(j-i) + j + k} ((\gvf \circ_k \mu) \circ_{i+2} \psi) \bullet_0 t^{j} x
\\
&
\qquad
- (-1)^{p+q}
\sum^{p-1}_{j=2} \sum^{j-2}_{k=0} (-1)^{n(j-1) + (q-1)j +p+k} ((\gvf \circ_{k+1} \mu) \circ_{j+2} \psi) \bullet_0 t^j x 
\\
&
\qquad
- (-1)^{p+q}
\sum^{p-2}_{j=2} \sum^{p-1}_{i=j+1}  \sum^{j-2}_{k=0} (-1)^{n(j-1) + q(j-i) + j+k} ((\gvf \circ_{k+1} \mu) \circ_{i+2} \psi) \bullet_0 t^j x.
\\
\end{split}
\end{equation*}
\end{footnotesize}
Therefore,
$$
(8) + (17) + (18) + (21) = 0.
$$
In a similar spirit we deal with the four remaining terms $(13)$, $(16)$, $(22)$, and $(23)$. One has
\begin{footnotesize}
\begin{equation*}
\begin{split}
(13) + (16) &= 
- (-1)^p
\sum^{p-2}_{j=1} \sum^{p-2}_{i=j} (-1)^{n(j-1) + q(i-j) + p} ((\gvf \circ_{i+2} \mu) \circ_{i+2} \psi) \bullet_0 t^{j-1} x 
\\
&
\qquad
- (-1)^p
\sum^{p-1}_{j=1} (-1)^{n(j-1) + q(j-1) + p} ((\gvf \circ_{j+1} \mu) \circ_{j+1} \psi) \bullet_0 t^{j-1} x
 \\
&
\qquad
- 
\sum^{p-2}_{j=1} \sum^{p-1}_{k=j+1} (-1)^{n(j-1) + (q-1)j + qp+k + p+1} ((\gvf \circ_{k+1} \mu) \circ_{j+1} \psi) \bullet_0 t^{j-1} x
 \\
&
\qquad
- 
\sum^{p-1}_{j=1} (-1)^{n(j-1) + qj + qp+ p+1} ((\gvf \circ_{j+1} \mu) \circ_{j+1} \psi) \bullet_0 t^{j-1} x
\\
&= 
- (-1)^p
\sum^{p-2}_{j=1} \sum^{p-2}_{i=j} (-1)^{n(j-1) + q(i-j) + qp+ p+1} ((\gvf \circ_{i+2} \mu) \circ_{i+2} \psi) \bullet_0 t^{j-1} x 
 \\
&
\qquad
- 
\sum^{p-2}_{j=1} \sum^{p-1}_{k=j+1} (-1)^{n(j-1) + q(j +p) +k} ((\gvf \circ_{k+1} \mu) \circ_{j+1} \psi) \bullet_0 t^{j-1} x.
\end{split}
\end{equation*}
\end{footnotesize}
Therefore,
$$
(13) + (16) + (22) + (23) = 0.
$$
This concludes the proof of Eq.~\rmref{panem2}. The identity \rmref{etcircenses} is proven by similar (and similarly tedious) computations, which is why we leave it to the reader. 
\end{proof}

\providecommand{\bysame}{\leavevmode\hbox to3em{\hrulefill}\thinspace}
\providecommand{\MR}{\relax\ifhmode\unskip\space\fi M`R }
\providecommand{\MRhref}[2]{%
  \href{http://www.ams.org/mathscinet-getitem?mr=#1}{#2}}
\providecommand{\href}[2]{#2}

\end{document}